\documentclass{TeFlon}

\begin{document}
\incorporarDatos{no}{UCM}{Ciencias Matemáticas}{Matemáticas Avanzadas}{\textit{Directed by}\\ Jon Asier Bárcena Petisco \& Luis Vega González\vspace{2mm}}{2nd of September, 2024}{Madrid}{Gonzalo Romera Oña}{}

\tituloTFG
\begin{titlepage}
\vspace{4cm}
\centering
{\textsf{\huge  \textbf{Neural Networks in Numerical Analysis \\ and Approximation Theory}}}

\vspace{0.3cm}

\vspace{7cm}
\textit{Memory presented for the Master's Thesis}\\
\vspace{5mm}
{\large \textbf{Gonzalo Romera Oña}}

\vspace{2cm}
\textit{Directed by}\\
\vspace{5mm}
{\bfseries Jon Asier Bárcena Petisco \& Luis Vega González}

\vfill

\textbf{{Departamento Análisis Matemático y Matemática Aplicada\\
Facultad de Ciencias Matemáticas\\
Universidad Complutense de Madrid}}\\
\vspace{5mm}
\textbf{Madrid, 2024}
\end{titlepage}
\fantasma
\newpage
\chapter*{Abstract}
\noindent

In this Master Thesis, we study the approximation capabilities of Neural Networks in the context of numerical resolution of elliptic PDEs and Approximation Theory. First of all, in Chapter \ref{chap: intro NN}, we introduce the mathematical definition of Neural Networks and perform some basic estimates on their composition and parallelization. Then, we implement in Chapter \ref{chap: Mult NN} the Galerkin method using Neural Network. In particular, we manage to build a Neural Network that approximates the inverse of positive-definite symmetric matrices, which allows to get a Garlerkin numerical solution of elliptic PDEs. Finally, in Chapter \ref{chap: space NN}, we introduce the approximation space of Neural Networks, a space which consists of functions in $L^p$ that are approximated at a certain rate when increasing the number of weights of Neural Networks. We find the relation of this space with the Besov space: the smoother a function is, the faster it can be approximated with Neural Networks when increasing the number of weights.
\chapter*{Resumen}
\noindent

En esta Trabajo de Fin de Máster, estudiamos las capacidades de aproximación de las Redes Neuronales en el contexto de la resolución numérica de EDPs elípticas y la Teoría de Aproximación. En primer lugar, en el Capítulo \ref{chap: intro NN}, introducimos la definición matemática de la red neuronal y damos estimaciones básicas en su composición y paralelización. Entonces, implementamos en el Capítulo \ref{chap: Mult NN} el método de Galerkin utilizando Redes Neuronales. En particular, conseguimos construir una Red Neuronal que aproxima la inversa de matrices simétricas definidas positivas, lo que permite obtener una solución numérica Garlerkin de EDPs elípticas. En el capítulo \ref{chap: space NN}, introducimos el espacio de aproximación de las Redes Neuronales, un espacio que consiste en funciones en $L^p$ que se aproximan a un cierto ratio cuando se aumenta el número de pesos de las Redes Neuronales. También encontramos la relación de este espacio con el espacio de Besov: cuanto más suave es la función, más rápido se puede aproximar una función con Redes Neuronales.
\newpage
\chapter*{Introduction}
\noindent

In recent years, we have witnessed the rise of neural networks. Although its primary applications are for the resolution of classification problems as it is summarized in \cite{LIU201711}, in this Master Thesis we focus on its usage to numerical resolutions of partial differential equations (PDEs) and Approximation Theory.

Despite neural network popularity and the amount of related research, almost all knowledge comes from empirical evidence. For example, the paper that popularized Physics-Informed Neural Networks (PINNs), \cite{raissi2019physics}, substantiates their use uniquely with implementation examples. As presented in \cite{raissi2019physics}, the idea is to minimize, among neural networks, some functional measuring the error on the equation, initial and boundary conditions to get an approximate solution. The said experiments are a success in the sense that the approximation solution are close to the known exact solutions for several classic PDEs. Other example of this tendency is the paper introducing the \textit{transformers}, \cite{46201}, a major breakthrough, again based on experimental results.

The theoretical study of Neural Network is an emerging field with many open questions. Although significatively fewer than empirical evidences, there are some theoretical results. Let us present some of the main ones: 

Regarding approximation theory, the seminal paper \cite{cybenko1989approximation} in 1989 shows the density of neural networks on the space of continuous functions with a bounded domain. On the same year, the hypotheses of this result are reduced in the paper \cite{hornik1989multilayer}, and four years later in \cite{leshno1993multilayer}. Concerning density theorems, there are early results on Sobolev spaces of first order, see \cite{abdeljawad2022approximations} and \cite{shen2022approximation}. In the same idea, \cite{guhring2021approximation} shows that neural networks approximate all except the higher order derivatives of a function on a Sobolev space. Also, in \cite{gribonval2022approximation} the authors link the Besov spaces and the pace of approximation of functions by Neural Networks. 

Concerning numerical results, the paper \cite{kutyniok2022theoretical} gives Neural Networks which approximates the product and inverse of matrices. With these Neural Networks, the authors show the approximation capabilities of Neural Networks of solution of elliptic PDEs. The density theorems of the previous paragraph have numerous uses on numerical too.

Also, we can find in the literature an usage of neural network for inverse and control problems. The papers \cite{puthawala2022globally} and \cite{furuya2024globally} show the arbitrary proximity of injective neural networks and Neural Operators respectively. These results are specially relevant to approximate solutions of an inverse problem with neural networks or approximate the resolver directly. In \cite{garcia2023control}, the authors present how to use PINNs in the context of control problems and conduct some experiments.

In this Master Thesis we seek to explain some of the previously exposed theoretical results. For that, the structure of the thesis is the following one:

In Chapter \ref{chap: intro NN}, we define neural networks as a sequence of alternating affine and a component wise acting functions. The composition of these functions gives us the realization of the neural network. The properties of this family of functions are studied in this Master Thesis, considering the amount of parameters of a neural network, namely the number of associated matrices and vectors and the number of affine functions, need to be tracked. To this purpose, we introduce the notation from \cite{Elbr_chter_2021} so we can operate neural networks while getting estimates on the number of parameters. 

In Chapter \ref{chap: Mult NN}, our main objective is using the Galerkin method to show that neural networks can numerically solve elliptic PDEs, like our primary source \cite{kutyniok2022theoretical} does. The Galerkin method, as it is explained on the book \cite{Quarteroni2016}, relies on the truncation of a Hilbert base to give a discrete solution of an elliptic PDE. The main problem is then transformed to the inversion of a matrix, where we can apply neural networks. We show that neural networks can perform such a task with a reasonable size, and consequently we show the capability of neural networks to solve numerically PDEs.

We finish this work in Chapter \ref{chap: space NN} with a study on approximation spaces introduced in our primary source \cite{gribonval2022approximation}. As many universal approximation theorems (densisty results for Neural Networks) states, like the one in \cite{leshno1993multilayer}, increasing the size allows proximity to any reasonable function. This suggests working on the spaces of functions approximated at a fixed rate with respect the size of the neural network. The study culminates relating these spaces and Besov spaces: the smoother a function is, the faster it will be approximated when increasing the number of weights. 

\indiceTFG
\pagenumbering{arabic}
\parindent=0em
\chapter[Introduction to Neural Networks]{Introduction to Neural Networks} \label{chap: intro NN}

In this chapter, our focus is to introduce the Neural Networks and general results concerning them. The notation is selectively picked from \cite{Elbr_chter_2021} and \cite{gribonval2022approximation}. The main point is to allow a better flow of results on the following chapters.

\section{Neural Network Definition and its Parameters.}

\begin{deff}[Neural Network]\label{def: NN}
    Let $L\in\N$, $n_0, \dots, n_L\in\N$ and, for $\ell = 1,\dots, L$, $A_\ell \in \R^{n_\ell\times n_{\ell-1}}$, $b_\ell \in \R^{n_\ell}$. Then, a \emph{Neural Network}, from now on \emph{NN} for short, $\phi$ is a tuple of pairs consisting of affine functions
    \[
    \func{T_\ell}{\R^{n_{\ell-1}}}{\R^{n_\ell}}{x}{A_\ell x + b_\ell} \quad \txt{ for } \ell = 1, \dots, L
    \]
    and so called \emph{activation functions} $\alpha_\ell: \R^{n_\ell} \longrightarrow \R^{n_\ell}$ for $\ell = 1, \dots, L$. We denote
    \[
    \phi = ((T_1, \alpha_1), (T_2, \alpha_2), \dots, (T_L, \alpha_L)).
    \]
\end{deff}

We immediately observe that, when working with NN, we have a lot of degrees of freedom. To keep track of said parameters, we introduce the following notation:

\begin{deff}\label{norm 0 def}
    Let $n,m\in\N$ and $A = (a_{i,j})_{i,j=1}^{n.m}\in\R^{n\times m}$. Then, we define
    \[
    \Norm{A}_0 \coloneqq \#\left\{(i,j) \;\middle |\; a_{i,j} \neq 0\right\},
    \]
    where $\#$ is the cardinal of the set.
\end{deff}
It can be seen that $\Normm_0$ satisfies all properties of a norm except the absolute homogeneity. Also we have the following property:

\begin{lemma}\label{norm 0 prop}
    Let $A\in\R^{d\times d}$, $B\in\R^{d\times l}$ be matrices and $v\in\R^{1\times n}$ a row vector. Then
    \[
    \Norm{vA}_0\leq\Norm{A}_0
    \]
    and if every row of $B$ has at most one non-zero entry then
    \[
    \Norm{AB}_0\leq\Norm{A}_0.
    \]
\end{lemma}

For a proof, see Lemma A.1 in \cite{kutyniok2022theoretical}.

\begin{deff}\label{NN parameters}
    Let $\phi = ((T_1, \alpha_1), \dots, (T_L, \alpha_L))$ be a NN where $n_0, \dots, n_L \in \N$ and for $\ell = 1, \dots, L$,
    \begin{align*}
        A_\ell \in \R^{n_\ell \times n_{\ell - 1}}, && b_\ell \in \R^{n_\ell}, && \func{T_\ell}{\R^{n_{\ell - 1}}}{\R^{n_\ell}}{x}{A_\ell x + b_\ell}.
    \end{align*}
    We define the following terms related to the parameters of $\phi$:
    \begin{enumerate}
        \item \label{enu: number of layers}
        Number of layers of $\phi$
        \[
        L(\phi) \coloneqq L.
        \]

        \item \label{enu: input output}
        Input and output dimension of $\phi$, respectively,
        \begin{align*}
            \dimin \phi \coloneqq n_0 && \txt{ and } && \dimout \phi \coloneqq n_L.
        \end{align*}

        \item \label{enu: number neurons}
        Number of neurons of $\phi$
        \[
        N(\phi) \coloneqq \sum_{\ell = 1}^{L - 1} n_\ell.
        \]

        \item \label{enu: conectivity}
        Connectivity of $\phi$
        \[
        C(\phi) \coloneqq \sum_{\ell = 1}^L \Norm{A_\ell}_0.
        \]

        \item \label{enu: number weights at a layer}
        Number of weights at the layer $\ell \in \{1, \dots, L\}$ of $\phi$
        \[
        M_\ell(\phi) \coloneqq \Norm{A_\ell}_0 + \Norm{b_\ell}_0.
        \]

        \item \label{enu: number weights}
        Number of weights of $\phi$
        \[
        M(\phi) \coloneqq \sum_{\ell = 1}^L M_\ell(\phi).
        \]

        \item \label{enu: weights}
        Weights of $\phi$
        \[
        W(\phi) \coloneqq ((A_1, b_1), \dots, (A_L, b_L)).
        \]
    \end{enumerate}
\end{deff}
The number of layers, neurons, connectivity and weights of a NN $\phi$ let us describe its \emph{size}. Indeed, since NNs are implemented on computers, we need to track these parameters as well as the non zero entries of the involved matrices. The input and output dimensions tell us the domain and codomain of the following map related to NN.
\begin{deff}[Realization]\label{NN realization}
    Let $\phi = ((T_1, \alpha_1), \dots, (T_L, \alpha_L))$ be a NN. Then, the realization of $\phi$ is the map
    \[
    R(\phi) \coloneqq \alpha_L \circ T_L \circ \dots \circ \alpha_1 \circ T_1.
    \]
\end{deff}
Since we can add the identity as affine and activation functions to any NN and get the same realization function, it is convenient to differentiate the NN and its realization. This is important specially in its computational applications.

\section{Basic Results.}

We are going to focus on the following networks:

\begin{deff}
    Let $\varrho: \R \longrightarrow \R$ be a function and $\phi = ((T_1, \alpha_1), \dots, (T_L, \alpha_L))$ a NN such that for all $\ell \in \{1, \dots, L-1\}$,
    \[
    \func{\alpha_\ell}{\R^{n_\ell}}{\R^{n_\ell}}{x}{ ( \alpha_{\ell, 1}(x_1), \dots, \alpha_{\ell, n_\ell}(x_{n_\ell}))}
    \]
    where $n_\ell \in \N$ and $\alpha_L = \Id_{n_L}: \R^{n_L} \longrightarrow \R^{n_L}$ is the identity.
    
    If $\alpha_{\ell, i} \in \{\varrho, \Id\}$ for all $\ell \in \{1, \dots, L - 1\}$ and $i \in \{1, \dots, n_\ell\}$ then we say that $\phi$ is a \emph{$\varrho$-NN}. 
    
    If $\alpha_{\ell, i} = \varrho$ for all $\ell \in \{1, \dots, L - 1\}$ and $i \in \{1, \dots, n_\ell\}$ then we say that $\phi$ is a \emph{strict $\varrho$-NN}.
\end{deff}

Note that a strict $\varrho$-NN is univocally determined by its weights, so, when working with strict $\varrho$-NN, it induces us to introduce the following notation: 
\begin{deff}\label{def: strict NN by weights}
    Let $w = ((A_1, b_1), \dots, (A_L, b_L))$ be a list of pairs of matrices and vectors. If the number of columns of $A_\ell$ coincides with the number of rows of $A_{\ell + 1}$ for $\ell = 1, \dots, L - 1$ and the number of rows of $A_\ell$ and $b_{\ell}$ are the same for $\ell = 1, \dots, L$, then we denote by $W^{-1}_\varrho(w)$ the only strict $\varrho$-NN such that
    \[
    W(W^{-1}_\varrho((w)) = w.
    \]
\end{deff}

We now introduce an important operation between $\varrho$-NN.

\begin{deff}\label{concatenation def}
    Let $\varrho: \R \longrightarrow \R$ be a function and
    \begin{align*}
        \phi^1 = ((T_1^1, \alpha_1^1), \dots, (T_{L_1}^1, \alpha_{L_1}^1)) && \txt{ and } && \phi^2 = ((T_1^2, \alpha_1^2), \dots, (T_{L_2}^2, \alpha_{L_2}^2))
    \end{align*}
    two $\varrho$-NN such that $\dimout \phi^2 = \dimin \phi^1$. Then, we define the concatenation of $\phi^1$ and $\phi^2$ as the $\varrho$-NN
    \[
    \phi^1 \bullet \phi^2 = ((T_1^2, \alpha_1^2), \dots, (T_{L_2-1}^2, \alpha_{L_2-1}^2), (T_1^1 \circ T_{L_2}^2, \alpha_1^1), \dots, (T_{L_1}^1, \alpha_{L_1}^1)).
    \]
    Remember that since $\phi^2$ is a $\varrho$-NN (see Definition \ref{def: NN}), $\alpha^2_{L_2} = \Id_{\R^{\dimout \phi^2}}$.
\end{deff}

For some of the results that we will use, we need to use concrete activation functions. We introduce
\begin{equation}\label{def: ReLU}
    \func{\ReLU}{\R}{\R}{x}{\max\{0, x\}}
\end{equation}
and for $r \in \N$,
\begin{equation}\label{def: ReLUr}
    \func{\varrho_r}{\R}{\R}{x}{\ReLU(x)^r}.
\end{equation}
Their interest is justified by its use on the applied field: see for example the first appearance of $\ReLU$ in the literature \cite{46201}, a contest winning NN \cite{fukushima1969visual} and its use in transformers \cite{krizhevsky2012imagenet}.
These functions have interesting properties but to show them, we introduce the following term.

\begin{deff}
    Let $\varrho: \R \longrightarrow \R$ and $f: \R \longrightarrow \R$ be two functions. We say that $\varrho$ can represent $f$ with $n \in \N$ terms if there exists a strict $\varrho$-NN $\phi^f$ with $L(\phi^f) = 2$ and $N(\phi^f) = n$. In other words, if there exists $a = (a_i)_{i = 1}^n, b = (b_i)_{i = 1}^n, c = (c_i)_{i = 1}^n \in \R^n$ and $d \in \R$ such that, for all $x \in \R$,
    \[
    f(x) = d + \sum_{i = 1}^n c_i \varrho(a_i x + b_i).
    \]
\end{deff}

\begin{lemma}\label{ReLU represents id}
    $\ReLU$ can represent $\Id$ with 2 terms
\end{lemma}
\begin{proof}
    It is as easy as making the following observation
    \begin{equation} \label{eq: ReLU represents id}
        x = \ReLU(x) - \ReLU(-x).
    \end{equation}
\end{proof}

As an immediate consequence we get

\begin{corollary}\label{ReLU represents idn}
    For every $n, L \in \N$, there exists a strict $\ReLU$-NN $\phi^{\Id}_{n, L}$ such that
    \begin{align*}
        R(\phi^{\Id}_{n, L}) = \Id_n && \txt{ and } && L(\phi^{\Id}_{n, L}) = L.
    \end{align*}
\end{corollary}
\begin{proof}
    For $L = 1$, the strict $\ReLU$-NN $\phi^{\Id}_{n, 1}$ with $W(\phi^{\Id}_{n, 1}) = ((\Eye_n, 0))$ where \\$\Eye_n = (\delta_{i,j})_{i, j = 1}^n \in \{0, 1\}^{n \times n}$ satisfies trivially the statement.

    For $L \geq 2$, the strict $\ReLU$-NN $\phi^{\Id}_{n, L}$ such that
    \[
    W(\phi^{\Id}_{n, L}) = \left ( \left ( \begin{bmatrix*} \Eye_n \\ - \Eye_n \end{bmatrix*}, 0 \right ), \overbrace{(\Eye_n, 0), \dots, (\Eye_n, 0)}^{L - 2 \txt{ times}}, \left ( \begin{bmatrix*} \Eye_n & - \Eye_n \end{bmatrix*}, 0 \right ) \right )
    \]
    satisfies the statement by equation \eqref{eq: ReLU represents id}.
\end{proof}

Throughout the next chapter, where we focus on strict $\ReLU$-NNs, our objective is to bound the number of weights and layers of different NNs. With this aim in mind, we define the following operations.

\begin{deff}\label{sparse concatenation}
    Let $\phi^1$ and $\phi^2$ be two strict $\ReLU$-NNs. If $n = \dimin\phi^1 = \dimout\phi^2$ then the sparse concatenation of $\phi^1$ and $\phi^2$ is the strict $\varrho$-NN defined as
    \[
    \phi^1 \odot \phi^2 \coloneqq \phi^1\bullet\phi_{n, 2}^{\Id}\bullet\phi^2.
    \]
\end{deff}
\begin{deff}\label{def: paralelization}
    Let $\{\widetilde{\phi}^i\}_{i=1}^k$ be a sequence of strict $\ReLU$-NNs such that, for any $i\in\{1,\dots,k\}$,
    \begin{align*}
        W(\widetilde{\phi}^i) =((A_1^i,b_1^i),\ldots,(A_{L_i}^i, b_{L_i}^i)), && 
        \dimin\widetilde{\phi}^1 =\dots = \dimin\widetilde{\phi}^k && \txt{ and } &&  L = \max\limits_{i =1,\ldots,k} L_i.
    \end{align*}
    If $L = L_i$ for all $i\in\{1,\ldots,k\}$ then we define the parallelization of $\widetilde{\phi}^1,\ldots,\widetilde{\phi}^k$ as the strict $\ReLU$-NN such that
    \begin{multline*}
        W(P(\widetilde{\phi}^1,\ldots,\widetilde{\phi}^k)) \coloneqq \lr{(}{\lr{(}{\begin{bmatrix*}
        A_1^1 & 0 &\cdots & 0\\
        0 & A_1^2 & \cdots &0\\
        \vdots &\vdots &\ddots &\\
        0 & 0 & & A_1^k
        \end{bmatrix*}, \begin{bmatrix*}
        b_1^1\\b_1^2\\ \vdots \\ b_1^k
        \end{bmatrix*}}{)},\ldots,}{.}\\
        \lr{.}{\lr{(}{\begin{bmatrix*}
        A_L^1 & 0 &\cdots & 0\\
        0 & A_L^2 & \cdots &0\\
        \vdots &\vdots &\ddots &\\
        0 & 0 & & A_L^k
        \end{bmatrix*}, \begin{bmatrix*}
        b_L^1\\b_L^2\\ \vdots \\ b_L^k
        \end{bmatrix*}}{)}}{)}.
    \end{multline*}
    If $L > L_j$ for some $j\in\{1,\ldots,k\}$, we set 
    \begin{align*}
        \text{if } L_i<L,\; \phi^i =& \phi_{\dimout\widetilde{\phi}^i,L-L_i}^{\Id}\odot\widetilde{\phi}^i,\\
        \text{if } L_i=L,\; \phi^i =& \widetilde{\phi}^i,
    \end{align*}
    and we define the parallelization of $\widetilde{\phi}^1,\ldots,\widetilde{\phi}^k$ as
    \[P(\widetilde{\phi}^1,\ldots,\widetilde{\phi}^k) = P(\phi^1,\ldots,\phi^k).\]
\end{deff}
Regarding the parameters of NN, we have the following estimates. 
\begin{lemma}\label{bounds para concat}
    Let $\phi^1, \ldots, \phi^k$ be strict $\ReLU$-NNs. If $\dimin\phi^1 = \dimout\phi^2$ then
    \begin{enumerate}
        \item[$(a.1)$] $R(\phi^1\odot\phi^2) = R(\phi^1) \circ R(\phi^2)$,
        \item[$(a.2)$] $L(\phi^1\odot\phi^2) = L(\phi^1)+L(\phi^2)$,
        \item[$(a.3)$] $M(\phi^1\odot\phi^2) \leq M(\phi^1) + M(\phi^2) + M_1(\phi^1) + M_{L(\phi^2)}(\phi^2) \leq 2(M(\phi^1)+M(\phi^2)),$
        \item[$(a.4)$] $M_1(\phi^1 \odot \phi^2) = M_1(\phi^2)$ and $M_{L(\phi^1 \odot \phi^2)}(\phi^1 \odot \phi^2) = M_{L(\phi^1)}(\phi^1)$.
    \end{enumerate}
    If $n = \dimin \phi^i$ for $i \in \{1, \ldots, k\}$ then
    \begin{enumerate}
        \item[$(b.1)$] $R(P(\phi^1, \ldots, \phi^k))(x_1, \ldots, x_k) = (R(\phi^1)(x_1), \ldots, R(\phi^k)(x_k))$ for all $x_1, \ldots, x_k \in \R^n$,
        \item[$(b.2)$] $L(P(\phi^1, \ldots, \phi^k)) = \max\limits_{i=1, \ldots, k} L(\phi^i)$,
        \item[$(b.3)$] $M_1(P(\phi^1, \ldots, \phi^k)) = \sum\limits_{i = 1}^k M_1(\phi^i)$,
        \item[$(b.4)$] $M_L(P(\phi^1, \ldots, \phi^k)) \leq \sum\limits_{i = 1}^k \max \{2 \dimout(\phi^i), M_{L(\phi^i)}(\phi^i)\}$,
        \item[$(b.5)$] $M(P(\phi^1, \ldots, \phi^k)) \leq 2 \sum\limits_{i = 1}^k M(\phi^i) + 4L \sum\limits_{i=1}^k \dimout(\phi^i)$,
    \end{enumerate}
    and if $L = L(\phi^1) = \ldots = L(\phi^k)$,
    \begin{enumerate}
        \item[$(b.6)$] $M_L(P(\phi^1, \ldots, \phi^k)) = \sum\limits_{i = 1}^k M_L(\phi^i)$,
        \item[$(b.7)$] $M(P(\phi^1, \ldots, \phi^k)) = \sum\limits_{i = 1}^k M(\phi^i)$.
    \end{enumerate}
\end{lemma}
\begin{proof}
    We only prove $(a.1)$, $(a.3)$ and $(b.4)$ as an example since all the others can be proved with similar techniques and their proof can be found at \cite{Elbr_chter_2021}.
    \begin{itemize}
        \item[$(a.1)$] Using Definition \ref{sparse concatenation}, it suffice to show
        \[
            R(\phi^1 \bullet \phi^2) = R(\phi^1) \circ R(\phi^2)
        \]
        with $\phi^i = ((T_1^i, \alpha_1^i), \dots, (T_{L_i}^i, \alpha_{L_i}^i))$ for $i = 1, 2$. By Definition \ref{NN realization}, 
        \[
            R(\phi^1) \circ R(\phi^2) = (\alpha_{L_1}^1 \circ T_{L_1}^1 \circ \dots \circ \alpha_1^1 \circ T_1^1) \circ (\alpha_{L_2}^2 \circ T_{L_2}^2 \circ \dots \circ \alpha_1^2 \circ T_1^2).
        \]
        Since $\phi_2$ is strict NN, $\alpha_{L_2} = \Id_{\dimout \phi^2}$ and so
        \[
            R(\phi^1) \circ R(\phi^2) = \alpha_{L_1}^1 \circ T_{L_1}^1 \circ \dots \circ \alpha_1^1 \circ T_1^1 \circ T_{L_2}^2 \circ \dots \circ \alpha_1^2 \circ T_1^2 = R(\phi^1 \bullet \phi^2).
        \]
        following Definition \ref{concatenation def}.
        
        \item [$(a.3)$] Let $\phi^1, \phi^2$ be two strict $\ReLU$-NNs such that
        \[
        W(\phi^1) = ((A^1_1, b_1^1), \dots, (A^1_{L_1}, b^1_{L_1}))
        \]
        and
        \[
        W(\phi^2) = ((A^2_1, b_1^2), \dots, (A^2_{L_2}, b^2_{L_2})).
        \]
        We have then
        \begin{align*}
            W(\phi^1 \odot \phi^2) =& ((A_1^2, b_1^2), \dots, (A_{L_2-1}^2, b^2_{L_2-1}), \left (\begin{bmatrix*} \Eye_n\\ -\Eye_n \end{bmatrix*} A_{L_2}^2, \begin{bmatrix*} \Eye_n\\ -\Eye_n \end{bmatrix*} b_{L_2}^2 + 0 \right ), \\
            &\left(A_1^1 \begin{bmatrix*} \Eye_n & -\Eye_n \end{bmatrix*}, A_1^1 0 + b_1^1 \right ), (A_2^1, b_2^1), \dots, (A_{L_1}^1, b_{L_1}^1))\\
            =&\left ( (A_1^2, b_1^2), \dots, \left ( \begin{bmatrix*} A_{L_2}^2\\ -A_{L_2}^2 \end{bmatrix*}, \begin{bmatrix*} b_{L_2}^2\\ -b_{L_2}^2 \end{bmatrix*} \right ) \left ( \begin{bmatrix*} A_1^1 & -A_1^1 \end{bmatrix*}, b_1^1 \right ), \dots, (A_{L_1}^1, b_{L_1}^1) \right )
        \end{align*}
        where $n = \dimin\phi^1 = \dimout\phi^2$. Since
        \begin{align*}
            M_{L_2}(\phi^1\odot\phi^2) =& 2(\Norm{A_{L_2}^2}+\Norm{b_{L_2}^2}) = 2M_{L_2}(\phi^2),\\ M_{L_2+1}(\phi^1\odot\phi^2) =& 2\Norm{A_1^1}_0+\Norm{b_1^1}_0 \leq 2 M_{1}(\phi^1),
        \end{align*}
        we conclude
        \[
            M(\phi^1\odot\phi^2) \leq M(\phi^1)+M(\phi^2)+M_{L_2}(\phi^2)+M_{1}(\phi^1).
        \]
        \item[$(b.4)$] Let $\phi^1,\dots,\phi^k$ be $k$ strict $\ReLU$-NNs and $L \coloneqq\max_{i=1,\dots,k} L(\phi^i)$.
        \begin{itemize}
            \item[$(b.4.1)$] If $d_j \coloneqq L - L(\phi^j)> 0$, we set
        \[
            W(\phi^{\Id}_{\dimout \phi^j, d_j} \odot \phi^j) = ((A_1^j, b_1^j), \dots, (A_{L}^j, b_{L}^j))
        \]
        and by $(a.4)$,
        \begin{equation}\label{1 bound lemma bounds para concat}
            \begin{split}
                M_L(\phi^{\Id}_{\dimout\phi^j,d_j}\odot\phi^j) = M_{d_j}(\phi^{\Id}_{\dimout\phi^j,d_j}) =& \begin{cases}
                    2 \dimout\phi^j & \txt{ if } d_j > 1\\
                    \dimout\phi^j & \txt{ if } d_j = 1
                \end{cases}\\
                \leq& 2\dimout \phi^j
            \end{split}
        \end{equation}

        \item[$(b.4.2)$] When $d_j = 0$, we set
        \[
            W(\phi^j) = ((A_1^j,b_1^j),\dots,(A_L^j,b_L^j))
        \]
        which, naturally, satisfies
        \begin{equation}\label{2 bound lemma bounds para concat}
            M_L(\phi^j) = \Norm{A_L^j}_0 + \Norm{b_L^j}_0.
        \end{equation}
        \end{itemize}
        
        It is then immediate by \eqref{1 bound lemma bounds para concat} and \eqref{2 bound lemma bounds para concat} that
        \[
            \Norm{A_L^j}_0 + \Norm{b_L^j}_0 \leq \max\{2\dimout\phi^j, M_{L(\phi^i)}(\phi^i)\}
        \]
        and by definition of the parallelization, we conclude
        \begin{align*}
            M_L(\phi^1,\dots,\phi^k) =& \sum_{i=1}^k \Norm{A_L^i}_0 + \Norm{b_L^i}_0\\
            \leq& \sum_{i = 1}^k \max\{2\dimout\phi^i, M_{L(\phi^i)}(\phi^i)\}.
        \end{align*}
    \end{itemize}
\end{proof}

\section{Additional Definitions of Neural Network in the Literature.} \label{sec: comments and other results chap 0}

    There are other approaches that are close to the ideas of classic NNs.

    \begin{enumerate}
    \item A Convolutional Neural Networks (CNN) is a NN giving the matrices a more restrictive structure which we do not consider in this thesis for a reason of length. We take the definition given in \cite{ZHOU2020787}. Fix some $s \in \N$ and $w = (w_0, w_1, \dots, w_s) \in \R^{s + 1}$. The discrete convolution of $w$ and another vector $v = (v_1, \dots, v_d) \in \R^d$ is defined as
    \[
    w * v = \left ( \sum_{k = 1}^d w_{1 - k} v_k, \sum_{k = 1}^d w_{2 - k} v_k, \dots, \sum_{k = 1}^d w_{s + d - k} v_k \right ) \in \R^{s + d}
    \]
    where $w_k = 0$ if $k \not \in \set{0, 1, \dots, s}$.
    Then, if we define the Toeplitz matrix
    \[
    \tau = \begin{bmatrix*} w_0 & 0 & \cdots & 0 & 0 \\
    w_1 & w_0 & \cdots & 0 & 0 \\
    \vdots & \vdots & \ddots & \vdots & \vdots \\
    w_s & w_{s - 1} & \cdots & w_0 & 0 \\
    0 & w_s & \cdots & w_1 & w_0 \\
    \vdots & \vdots & \ddots & \vdots & \vdots \\
    0 & 0 & \cdots & w_s & w_{s - 1} \\
    0 & 0 & \cdots & 0 & w_s
    \end{bmatrix*} \in \R^{(d + s) \times d}
    \]
    then $\tau v = w * v$.
    We also can define the downsampling matrix: given $d, m \in \N$, we define
    \[
    D = (d_j \delta_{i, j})_{\substack{i = 1, \dots, \Floor{d / m} \\ j = 1, \dots, d}}
    \]
    where $d_j = 0$ if $j / m \not \in \N$ and $d_j = 1$ if $j / m \in \N$. A CNN is then a NN where the affine functions and activation functions $(\alpha, T)$ are either $Tx = \tau x + b$ or $\alpha = \Id$ and $Tx = Dx$ where $\tau$ is a Toeplitz matrix, $b$ is a real vector and $D$ is a downsampling matrix.
    
    \item Deep Operator Nets, or DeepONets for short, are operators between spaces of functions and are introduced in \cite{lu2019deeponet}. As in a NN, DeepONets are characterized by some parameters: for a fixed set $X$, for every function $u: X \longrightarrow \R$ and $y \in \R^d$, a DeepONets operator is defined as
    \[
    G(u)(y) = \inner{(\inner{c^1, \sigma(\xi^1 \textbf{u} + \theta^1)}, \dots, \inner{c^p, \sigma(\xi^p \textbf{u} + \theta^p))}, \sigma(W y + \tau)}
    \]
    with $c^1, \dots, c^p, \theta^1, \dots, \theta^p \in \R^n$, $\xi^1, \dots, \xi^p \in \R^{n \times m}$, $ \textbf{u} = (u(x_1), \dots, u(x_m)) \in \R^m$, $W \in \R^{p \times d}$, $\tau \in \R^p, x_1, \dots, x_m \in X$ and $\sigma: \R \longrightarrow \R$ is an activation function applied to every component.

    \item Neural Operator are presented in \cite{kovachki2023neural}. Imitating the layers of a classic NN, the layer of a Neural Operator is defined for every function $u: \R^d \longrightarrow \R $ as
    \[
    u \longmapsto \varrho \left ( W(\; \cdot \;)u + \int K(\; \cdot \;, y, u(\; \cdot \;), u(y)) u(y) dy + b \right )
    \]
    where the kernel $K: \R^d \times \R^d \times \R \times \R$ is a parameter, $W, b:\R^d \longrightarrow \R$ are functions and $\varrho$ is an activation function, giving the Neural Operators its learning capability. A Neural Operator is then a composition of these layers. 
    \end{enumerate}
    However, due to space, we are not considering such structures.
\parindent=0em
\chapter[Deep Neural Networks and PDEs]{Deep Neural Networks for the numerical resolution of PDEs} \label{chap: Mult NN}

The paper \cite{kutyniok2022theoretical} upper bounds on the number of parameters required to approximate the solution of elliptic PDEs with Neural Networks. Even though the paper \cite{kutyniok2022theoretical} considers parametric problems, the main points appear without that extension and thus we restraint ourselves to elliptic equations. The idea behind is to build a Neural Network with a controlled size to apply the Galerkin method. Its construction is based on the approximation of the square function by a Neural Network, which leads to an approximation for the product of two scalars, the product of two matrices and then the inverse of matrices.
The idea of said construction was brought by \cite{kutyniok2022theoretical}, but here we improved the architectural bounds of the Neural Networks for multiplying scalars (see Remark \ref{better mult}) and inverting matrices, as well as shorten the proof of their properties (see Remark \ref{better inversion NN}).

\section{Galerkin Method.}

The Galerkin method is a numerical approach to solve some PDEs. It relies on the fact that some PDEs can be reduced to find some element $u$ in an infinite-dimensional separable Hilbert space $H$ which satisfies that
\begin{align}\label{pde}
    b(u,v) = f(v), \txt{ for all } v \in H,
\end{align}
where $b$ is a bilinear form and $f \in H^*$, where $H^*$ denotes the dual space of $H$. Given the problem \eqref{pde}, one can choose from a Hilbert base $\seq{\varphi_k}{k\in\N}\subset H$ a discretization, a finite collection of elements $\seq{\varphi_{k_i}}{i=1}^d$ from said base, hoping that solving the problem on the finite-dimensional space $V = \Span \lr{\{}{\varphi_{k_1},\ldots,\varphi_{k_d}}{\}}$ gives us a reasonable approximation of the actual solution as explained in Chapter 3 of \cite{Quarteroni2016}. The finite-dimensional case is relatively easy to solve since it only involves linear algebra:
\begin{align*}
b(u, v) = f(v) \txt{ for all } v\in V &\iff b(u, \varphi_{k_i}) = f(\varphi_{k_i}) \txt{ for } i=1,\ldots, d\\
&\iff B \mu = F,
\end{align*}
where 
\begin{align*}
    B =&(b(\varphi_{k_i}, \varphi_{k_j}))_{i,j}\in\R^{d\times d},\\
    u =& \sum\limits_{i=1}^d \mu_i \varphi_{k_i},\\
    \mu =& (\mu_1,\ldots, \mu_d)\in\R^d,\\
    \txt{and } F =& (f(\varphi_{k_1}),\ldots,f(\varphi_{k_d}))\in\R^d.
\end{align*}
Therefore, the finite-dimensional problem is solved by finding $\mu$.

We guarantee the existence and uniqueness of solutions when the bilinear form is symmetric, continuous and coercive, that is, when it exists some constants $C_{\txt{cont}}, C_{\txt{coer}}>0$ such that, for every $u, v\in H$,
\begin{align*}
    b(u,v) = b(v,u), && \Abs{b(u,v)}\leq C_{\txt{cont}}\Norm{u}_H\Norm{v}_H && \txt{and} && C_{\txt{coer}} \leq \frac{b(u,u)}{\Norm{u}_H^2}.
\end{align*}
Indeed, the Lax-Milgram's Theorem ensures it, see Theorem 1 in Section 6.2.1 of \cite{evans2022partial}. These conditions also tells us that $B$ is positive definite, symmetric and invertible on a explicit way.

\begin{proposition}\label{B invertible}
    Let $b:H\times H\longrightarrow \R$ be a symmetric, continuous and coecive bilinear form on a Hilbert space $H$. If $B = (b(\varphi_i,\varphi_j))_{i,j=1}^d\in\R^{d\times d}$ where $\varphi_1,\dots,\varphi_d$ are orthonormal, then $B$ is invertible and for $\alpha>0$ sufficiently small,
    \begin{equation}\label{B-1}
        B^{-1} = \alpha \sum_{k=0}^\infty \left(\Id_d-\alpha B\right)^k.
    \end{equation}
\end{proposition}

\begin{remark}
    As seen later on this chapter (Theorem \ref{matrix inversion NN}), even the statement is true for small $\alpha$, we prefer that is as big as possible. 
\end{remark}

The proof of this proposition relies on the spectral matrix norm.
\begin{deff}
    Given a matrix $A\in\R^{d\times d}$, its 2 norm is
    \[
    \Norm{A}_2 \coloneqq \sup_{x\neq 0}\frac{\Abs{Ax}_2}{\Abs{x}_2}.
    \]
\end{deff}
As it can be seen in Proposition 27 and Theorem 127 of \cite{layton2014numerical}, if $A\in\R^{d\times d}$ is a symmetric matrix, its eigenvalues are real, it is diagonalizable and its spectral norm is in fact related to its spectral decomposition: 
\begin{equation}\label{norm 2}
    \Norm{A}_2 = \max\{\Abs{\lambda}\;|\; \lambda \txt{ eigenvalue of } A\}.
\end{equation}

\begin{lemma}\label{Neumann}
    Given a matrix $A\in\R^{d\times d}$ such that $\Norm{A}_2<1$ then
    \begin{equation}\label{eq Neumann}
        (\Eye_d - A)^{-1} = \sum_{k=0}^\infty A^k.
    \end{equation}
    and for any $N\in\N$,
    \begin{equation}\label{neumann error}
        \Norm{(\Eye_d-A)^{-1} - \sum_{k=0}^N A^k}_2 \leq \frac{\Norm{A}_2^{N+1}}{1-\Norm{A}_2}.
    \end{equation}
\end{lemma}
\begin{proof}
    The convergence of this series is guaranteed by the completeness of $(\R^{d\times d},\Norm{\;\cdot\;}_2)$: let $N,M\in\N$ with $N<M$,
    \[
        \Norm{\sum_{k=0}^NA^k-\sum_{k=0}^M A^k}_2 \leq \sum_{k = N+1}^M \Norm{A}_2^k = \Norm{A}_2^{N+1}\frac{1-\Norm{A}_2^{M-N}}{1-\Norm{A}_2} \xrightarrow[N,M\rightarrow \infty]{} 0.
    \]
    Since $A^k\rightarrow 0\in\R^{d\times d}$ as $k\rightarrow\infty$ and
    \[
    (\Eye_d - A)\sum_{k=0}^N A^k = \sum_{k=0}^N A^k(\Eye_d - A) = \Eye_d - A^{N+1}
    \]
    where $\Id_d \coloneqq (\delta_{i,j})_{i,j=1}^d\in\R^{d\times d}$, by the continuity of the matrix multiplication, we deduce \eqref{eq Neumann}. 
    
    Given $N\in\N$, using the triangular inequality and the submultiplicativity of the spectral norm, we get
    \[
    \begin{split}
        \Norm{(\Eye_d-A)^{-1} - \sum_{k=0}^N A^k}_2 &= \Norm{\sum_{k=N+1}^\infty A^k} = \Norm{A^{N+1}\sum_{k=0}^\infty A^k}_2\\
        &\leq \Norm{A}^{N+1}_2\sum_{k=0}^\infty \Norm{A}^k_2=\frac{\Norm{A}^{N+1}_2}{1-\Norm{A}_2}
    \end{split}
    \]
    which proves \eqref{neumann error}.
\end{proof}

\begin{proof}[Proof of Proposition \ref{B invertible}]
    If we find $\alpha\in\R$ such that $\Norm{\Eye_d-\alpha B}_2<1$ then we can apply Lemma \ref{Neumann}:
    \[
    B^{-1} = \alpha (\Eye_d -(\Eye_d-\alpha B))^{-1} = \alpha \sum_{k=0}^\infty(\Eye_d-\alpha B)^k.
    \]
    By hypothesis, $B$ is symmetric and therefore $\Id_d - \alpha B$ too for any $\alpha\in\R$ so, we can use \eqref{norm 2} and the fact that its eigenvalues are real. Let $\lambda\in\bb{R}$ be an eigenvalue of $B$ and $x = (x_1,\dots,x_d)\in\bb{R}^d$ its associated eigenvector such that $\Norm{x}_2 = 1$ . It is easy to check that
    \[
        \lambda = x^TBx = \sum_{i,j=1}^d x_ix_j b(\varphi_i,\varphi_j) = b\left(\sum_{i=1}^d x_i \varphi_i,\sum_{i=1}^d x_i \varphi_i\right)\in[\Ccoer,\Ccont],
    \]
    where $\Ccont$ and $\Ccoer$ are the corresponding continuous and coercive constant of the bilinear form $b$. So, for any eigenvalue $\lambda\in[\Ccoer,\Ccont]$ of $B$ and $\alpha\in (0,1/\Ccoer)$, the eigenvalue $1-\alpha\lambda$ of $\Eye_d-\alpha B$ is between $0$ and $1$ implying that $\Norm{\Eye_d-\alpha B}_2<1$.
\end{proof}
Proposition \ref{B invertible} motivates us to define the set
\begin{equation}\label{def invertible set}
    \mathcal{I}_d(\alpha, \delta) = \lr{\{}{A\in\R^{d\times d} \;\middle |\; \Norm{\Eye_d-\alpha A}_2 \leq \delta}\}.
\end{equation}
Even thought all the matrix in $\mathcal{I}_d(\alpha, \delta)$ are invertible if $\delta \in [0, 1)$, not all invertible matrices are in $\bigcup\limits_{\alpha \in \R} \bigcup\limits_{\delta \in [0, 1)} \mathcal{I}_d(\alpha, \delta)$. For example $\begin{bmatrix*}
    0 & 1 \\ 1 & 0
\end{bmatrix*}$:
\[
\Norm{\begin{bmatrix*} -\alpha & 1 \\ 1 & -\alpha \end{bmatrix*}}_2 = 1 + \Abs{\alpha} \geq 1.
\]

\begin{remark}\label{relation delta alpha}
    If we restrain ourselves to a symmetric matrix $B \in \R^{d \times d}$ with real positive eigenvalues, like in Propopsition \ref{B invertible}, the relation between $\alpha$ and the spectral norm of $\Eye_d - \alpha B$ is clear. Indeed, if $\lambda$ and $\Lambda$ are the smallest and biggest eigenvalue of $B$ respectively then, using \eqref{norm 2},
    \[
    \Norm{\Eye_d - \alpha B}_2 = 1 - \alpha \lambda
    \]
    for all $\alpha \in [0, 1 / \Lambda)$. This relation gives us an interval to work. As we later will see, when approximating the inverse of such a matrix, it is convenient to choose $\alpha$ as big as possible.
\end{remark}

If we can approximate $B$, we will see on the next sections that we can approximate $B^{-1}$ using a Neural Network and consequently getting a Galerkin solution of the PDE.

\section{Neural Network construction.}
\subsection{Product Neural Network.}

We can now begin with the approximation results.
\begin{proposition}[Square NN] \label{sq NN}
    If we call
    \begin{align*}
        g(x) = \min\{2x,2-2x\}, && g^m = \overbrace{g\circ\ldots\circ g}^m&& \text{and} && f_m(x) = x - \sum_{k=1}^m \frac{g^k(x)}{2^{2k}},
    \end{align*}
    then
    \[
    \sup_{x\in[0,1]}\Abs{x^2-f_m(x)} = 2^{-2(m+1)}.
    \]
\end{proposition}
The functions $g^m$ are commonly named sawtooth, see Propostion \ref{plot gm}.
We detail and split the proof of Proposition 2 of \cite{YAROTSKY2017103} into two lemmas:

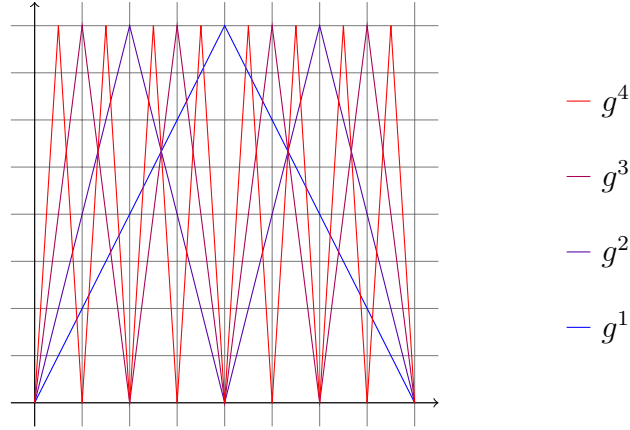
\begin{figure}
    \centering
    \begin{tikzpicture}[scale = 5]
        \draw[step = .125, color = gray] (-.0625, -.0625) grid (1.0625, 1.0625);
        \draw[->] (-.0625, 0) -- (1.0625, 0);
        \draw[->] (0, -.0625) -- (0, 1.0625);
        \foreach \n [count = \c, evaluate = \c as \shade using 100*(\c-1)/3, evaluate = \n as \neval using 2^\n] in {0,...,3}
            {
            \draw[color = red!\shade!blue] (0, 0) 
                \foreach \j in {1,...,\neval}
                    {--++ (.5*2^-\n, 1) --++ (.5*2^-\n, -1)};
            \draw[color = red!\shade!blue] (1.4, .2 * \c) --++ (0.0625, 0);
            }
        \foreach \n in {1,...,4}
            {
            \node[right] at (1.4625, .2 * \n) {$g^\n$};
            }
    \end{tikzpicture}
    \label{plot gm}
    \caption{Plot of $g^m$ for $m = 1,2,3,4$.}
\end{figure}

\begin{lemma}\label{sq NN lemma 1}
    Let
    \begin{align*}
        g(x) = \min\{2x, 2-2x\}    &&  \txt{ and } && g^m = \overbrace{g\circ\ldots\circ g}^m
    \end{align*}
    for $x\in[0,1]$ and $m\in\N$. Then
    \[
    g^m(x) = g(2^{m-1}x-\Floor{2^{m-1}x}).
    \]
\end{lemma}
We recall that for any $x\in\R$, $\Floor{x}\in\mathbb{Z}$ denotes the only integer satisfying 
\[\Floor{x}\leq x < \Floor{x}+1.\]
\begin{proof}
    We prove by induction on $m$. The case $m = 1$ is immediate. Suppose that the case $m\in\N$ holds, which implies, for any $x\in[0,1]$,
    \begin{equation}\label{ind step sq NN lemma 1}
        g^{m+1}(x) = g(g^m(x)) = g(g(2^{m-1}x-\Floor{2^{m-1}x})).
    \end{equation}
    We divide the induction step on three subcases.
    \begin{itemize}
        \item If $2^{m-1}x-\Floor{2^{m-1}x}\in[0,1/2)$, then
        \begin{equation}\label{floor entre 0 y .5}
            \begin{split}
                &\Floor{2^{m-1}x} \leq 2^{m-1}x < \Floor{2^{m-1}x}+\frac{1}{2}\\
                \iff &2\Floor{2^{m-1}x}\leq 2^mx < 2\Floor{2^{m-1}x}+1\\
                \iff & 2\Floor{2^{m-1}x} = \Floor{2^mx}
            \end{split}
        \end{equation}
        and
        \begin{equation*}
            g(2^{m-1}x-\Floor{2^{m-1}x}) = 2^mx - 2\Floor{2^{m-1}x} = 2^mx-\Floor{2^mx}.
        \end{equation*}
        Thus, using \eqref{ind step sq NN lemma 1}, we prove the induction step:
        \[
        g^{m+1}(x) = g(2^mx-\Floor{2^{m}x}).
        \]
        \item If $2^{m-1}x-\Floor{2^{m-1}x}\in[1/2,1)$, then
        \begin{equation}\label{floor entre .5 y 1}
            \begin{split}
                &\Floor{2^{m-1}x}+\frac{1}{2} \leq 2^{m-1}x< \Floor{2^{m-1}x}+1\\
                &\iff 2\Floor{2^{m-1}x}\leq 2^mx-1<2\Floor{2^{m-1}x}+1\\
                &\iff 2\Floor{2^{m-1}x} = \Floor{2^mx-1} = \Floor{2^{m}x}-1
            \end{split}
        \end{equation}
        and
        \begin{equation*}
            g(2^{m-1}x-\Floor{2^{m-1}x}) = 2 - (2^mx-2\Floor{2^{m-1}x}) = 1-(2^mx-\Floor{2^mx}).
        \end{equation*}
        Thus, using \eqref{ind step sq NN lemma 1}, we get
        \[
        g^{m+1}(x) = g(1 -\left(2^mx - \Floor{2^{m}x}\right)).
        \]
        It is sufficient to use the fact that $g(t) = g(1-t)$ for all $t\in[0,1]$ to conclude the inductive step.
        \item If $x = 1$, the identity is immediate.
    \end{itemize}
\end{proof}

\begin{lemma}\label{sq NN lemma 2}
    Let $g:[0,1]\longrightarrow\R$ be the function defined at Lemma \ref{1 bound lemma bounds para concat}, $m\in\N$ and
    \begin{equation}\label{linear interpolation x2}
    h_m(x) = \frac{2\Floor{2^mx}+1}{2^m}\left(x-\frac{\Floor{2^mx}}{2^m}\right)+\left(\frac{\Floor{2^mx}}{2^m}\right)^2.
    \end{equation}
    Then
    \begin{equation*} \label{eq: difference of linear interpolation}
    h_m(x)-h_{m+1}(x) = \frac{g^{m+1}(x)}{2^{2(m+1)}}.
    \end{equation*}
\end{lemma}

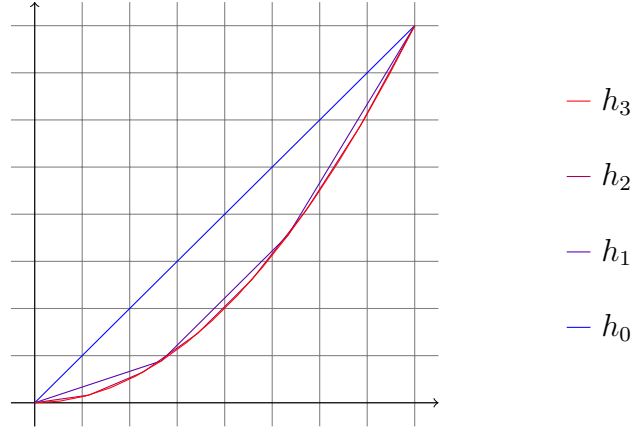
\begin{figure}
    \centering
    \begin{tikzpicture}[scale = 5]
        \draw[color = gray, step = .125] (-.0625, -.0625) grid (1.0625, 1.0625);
        \draw[->] (-.0625, 0) -- (1.0625, 0);
        \draw[->] (0, -.0625) -- (0, 1.0625);
        \foreach \n [count = \c, evaluate = \c as \shade using 100*(\c-1)/3, evaluate = \n] in {2^1, 2^..., 2^4}
            {
            \draw[color = red!\shade!blue, domain = 0:1, samples = \n] plot (\x, \x^2);
            \draw[color = red!\shade!blue] (1.4, .2 * \c) --++ (.0625, 0);
            }
        \foreach \n in {0,...,3}
            {
            \node[right] at (1.4625, .2 * \n + .2) {$h_\n$};
            }
    \end{tikzpicture}
    \label{plot hm}
    \caption{Plot of $h_m$ for $m = 0, 1, 2, 3$.}
\end{figure}

\begin{remark}
    The function $h_m$ is the linear interpolation of the square function on the dyadic intervals in $[0, 1]$ as it is seen in \ref{plot hm}. Indeed, given $k\in\{0,\dots,2^{-m}-1\}$, the map $x \mapsto \Floor{2^m x}$ is constant in the dyadic interval $[k2^{-m},(k+1)2^{-m})$, implying that $h_m$ is a piecewise one degree polynomial, and since
    \[
    \lim_{x\rightarrow \frac{k}{2^m}^-} h_m(x) = \lim_{x\rightarrow \frac{k}{2^m}^+} h_m(x) = \left(\frac{k}{2^m}\right)^2,
    \]
    it is continuous.
\end{remark}

\begin{proof}
    As a consequence of Lemma \ref{sq NN lemma 1}, we know that:
    \begin{equation}\label{explicit gm+1}
        \frac{g^{m+1}(x)}{2^{2(m+1)}} = \begin{cases}
        \frac{1}{2^{m+1}}\left(x-\frac{\Floor{2^mx}}{2^m}\right) & \txt{ if } 2^mx-\Floor{2^mx}\in[0,1/2]\\
        \frac{1}{2^{m+1}}\left(\frac{\Floor{2^mx}+1}{2^m}-x\right)& \txt{ if } 2^mx-\Floor{2^mx}\in[1/2,1]
    \end{cases}.
    \end{equation}
    As on the Lemma \ref{sq NN lemma 1}, we split the proof on three cases.
    \begin{itemize}
        \item If $2^mx-\Floor{2^mx}\in[0,1/2)$, by \eqref{floor entre 0 y .5},
        \begin{equation*}
            \Floor{2^{m+1}x}=2\Floor{2^mx}.
        \end{equation*}
        Thus, by the definition of $h_{m+1}$ \eqref{linear interpolation x2}, we get
        \begin{equation*}
            \begin{split}
                h_{m+1}(x) =& \frac{2\Floor{2^{m+1}x}+1}{2^{m+1}}\left(x-\frac{\Floor{2^{m+1}x}}{2^{m+1}}\right)+\left(\frac{\Floor{2^{m+1}x}}{2^{m+1}}\right)^2\\
                =& \frac{4\Floor{2^mx}+1}{2^{m+1}}\left(x-\frac{\Floor{2^mx}}{2^m}\right)+\left(\frac{\Floor{2^mx}}{2^m}\right)^2\\
                =& \frac{2\Floor{2^mx}+1}{2^m}\left(x-\frac{\Floor{2^mx}}{2^m}\right)+\left(\frac{\Floor{2^mx}}{2^m}\right)^2-\frac{1}{2^{m+1}}\left(x-\frac{\Floor{2^mx}}{2^m}\right)\\
                =&h_m(x)-\frac{g^{m+1}(x)}{2^{2(m+1)}}
            \end{split}
        \end{equation*}
        just as we wanted, where at the last equality we used the definition of $h_m$ given by \eqref{linear interpolation x2} and the explicit expression of $g^{m+1}$ given by \eqref{explicit gm+1}.
        \item If $2^mx-\Floor{2^mx}\in[1/2,1)$, by \eqref{floor entre .5 y 1},
        \begin{equation*}
            \Floor{2^{m+1}x}=2\Floor{2^mx}+1.
        \end{equation*}
        Thus, by the definition of $h_{m+1}$ \eqref{linear interpolation x2}, we get
        \begin{equation*}
            \begin{split}
                h_{m+1}(x) =& \frac{2\Floor{2^{m+1}x}+1}{2^{m+1}}\left(x-\frac{\Floor{2^{m+1}x}}{2^{m+1}}\right)+\left(\frac{\Floor{2^{m+1}x}}{2^{m+1}}\right)^2\\
                =& \left(\frac{2\Floor{2^mx}+1}{2^m}+\frac{1}{2^{m+1}}\right)\left(x-\left(\frac{\Floor{2^mx}}{2^m}+\frac{1}{2^{m+1}}\right)\right)+\left(\frac{\Floor{2^mx}}{2^m}+\frac{1}{2^{m+1}}\right)^2\\
                =& \frac{2\Floor{2^mx}+1}{2^m}\left(x-\frac{\Floor{2^mx}}{2^m}\right)-\frac{1}{2^{m+1}}\frac{2\Floor{2^mx}+1}{2^m}\\
                &+\frac{1}{2^{m+1}}\left(x-\frac{2\Floor{2^mx}+1}{2^{m+1}}\right) + \left(\frac{\Floor{2^mx}}{2^m}\right)^2+2\frac{\Floor{2^mx}}{2^m}\frac{1}{2^{m+1}}+\left(\frac{1}{2^{m+1}}\right)^2\\
                =&h_m(x) + \frac{1}{2^{m+1}}\left(x-\frac{2\Floor{2^mx}+1}{2^{m+1}}\right)-\frac{1}{2^{2(m+1)}}\\
                =&h_m(x) - \frac{g^{m+1}(x)}{2^{2(m+1)}}
            \end{split}
        \end{equation*}
        just as we wanted, where at the last equality we used the definition of $h_m$ given by \eqref{linear interpolation x2} and the explicit expression of $g^{m+1}$ given by \eqref{explicit gm+1}.
        \item The case $x = 1$ is obvious.
    \end{itemize}
\end{proof}
\begin{proof}[Proof of Proposition \ref{sq NN}]
    By the Lemma \ref{sq NN lemma 2}, the linear interpolation of the square function $h$ satisfies
    \begin{align*}
    h_m(x)-h_0(x)=&\sum_{k=0}^{m-1}h_{k+1}(x)-h_k(x)=-\sum_{k=0}^{m-1}\frac{g^{k+1}(x)}{2^{2(k+1)}}
    \end{align*}
    which implies that
    \[
    h_m(x)= x- \sum_{k=1}^m\frac{g^k(x)}{2^{2k}} = f_m(x).
    \]
    It is easy to check that
    \[
    \max_{x\in[0,1]}\Abs{h_m(x)-x^2}=\max_{k=0,\ldots,2^{m-1}-1}\lr{(}{\max_{x\in[\frac{k}{2^m},\frac{k+1}{2^m}]}\Abs{h_m(x)-x^2}}{)}=2^{-2(m+1)}
    \]
    concluding the proof.
\end{proof}
We can write $f_m$ as a strict $\ReLU$-NN because
\[
\min\{2 x, 2 - 2 x\} = 2 x -\max\{0, 4 x - 2\} = 2 \ReLU(x) - 4 \ReLU( x - 1/2)
\]
for every $x\in[0,1]$.

For $k\in\N$ and $m> 2$, we define
\begin{align*}
\alpha &\coloneqq \begin{bmatrix*}
    1\\
    1\\
    1\\
    1
\end{bmatrix*},\quad b \coloneqq \begin{bmatrix*}
    0\\
    -1/2\\
    -1\\
    0
\end{bmatrix*},\\
A_k &\coloneqq \begin{bmatrix*}
    2 & -4 & 2 & 0\\
    2 & -4 & 2 & 0\\
    2 & -4 & 2 & 0\\
    -2^{-2k+3} & 2^{-2k+4} & -2^{-2k+3} & 1
\end{bmatrix*},\\
\omega^m &\coloneqq \begin{bmatrix*}
    -2^{-2m+3} & 2^{-2m+4} & -2^{-2m+3} & 1
\end{bmatrix*}
\end{align*}
and a strict $\ReLU$-NN with
\begin{equation}\label{sq NN def}
W(\phi^m_{\txt{sq}}) = ((\alpha, b), (A_2, b), \ldots, (A_{m-1}, b), (\omega^m, 0))
\end{equation}
then, as proved on the Lemma 6.1 in \cite{Elbr_chter_2021}, $f_{m-1} = R(\phi^m_{\txt{sq}})$. After some calculations, we get
\begin{equation}\label{sq bounds}
\begin{split}
    \dimin(\phi^m_{\txt{sq}}) =& \dimout(\phi^m_{\txt{sq}}) = 1,\\
    M_1(\phi^m_{\txt{sq}}) =& 6,\quad M_{L(\phi^m_{\txt{sq}})}(\phi^m_{\txt{sq}}) = 4,\\
    M_\ell(\phi^m_{\txt{sq}}) =& 15 \quad \txt{ for } \ell = 2,\ldots,m-1,\\
    M(\phi^m_{\txt{sq}}) =& 10 + 15(m-2),\\
    L(\phi^m_{\txt{sq}}) =& m.
\end{split}
\end{equation}
Consequently, if we want that $\Abs{x^2 - R(\phi^m_{\txt{sq}})(x)}<\varepsilon<1$ for some $x\in[0,1]$, it is sufficient to choose $m = \lr{\lceil}{\frac{1}{2}\log_2\lr{(}{\frac{1}{\varepsilon}}{)}}{\rceil}$.

It follows from the equation
\begin{equation}\label{trapecio}
    xy = M^2\left(\left(\frac{x+y}{2M}\right)^2-\left(\frac{x-y}{2M}\right)^2\right) \;\;\;\; \forall M> 0
\end{equation}
the following Corollary.
\begin{corollary} \label{scalar mult NN}
    Keeping the same notation as the previous Proposition \ref{sq NN}, for every $M>0$ then
    \[
    \sup_{x,y\in[-M,M]}\Abs{xy - M^2\lr{(}{f_m\lr{(}{\frac{\Abs{x+y}}{2M}}{)}-f_m\lr{(}{\frac{\Abs{x-y}}{2M}}{)}}{)}} \leq M^2 2^{-2(m+1)+1}.
    \]
\end{corollary}
\begin{remark}\label{better mult}
    In the Proposition 3.7 of the paper \cite{kutyniok2022theoretical}, they use the formula valid for all $M>0$,
    \[
    xy = 2M^2\left(\left(\frac{x+y}{2M}\right)^2- \left(\frac{x}{2M}\right)^2-\left(\frac{y}{2M}\right)^2\right)
    \]
    instead of \eqref{trapecio}. Note that this identity requires squaring three terms instead of two, conducing to a bulkier NN.
\end{remark}
It is easy to transition to NN terms.
\begin{corollary}[Scalar multiplication NN]\label{scalar mult NN def and bound}
    For every $\varepsilon\in(0,M^2)$, setting \\$C = \max\{1,M\}$ and $m = \Floor{\frac{1}{2}\log_2\lr{(}{\frac{2C^2}{\varepsilon}}{)}}$, the strict $\ReLU$-NN with weights
    \begin{multline*}
        \phi^{m,M}_\times \coloneqq W_{\ReLU}^{-1}\lr{(}{\lr{(}{
        \begin{bmatrix*}
        M^2 & -M^2
        \end{bmatrix*}, 0
        }{)}}{)}\bullet P(\phi^m_{\txt{sq}},\phi^m_{\txt{sq}})\\
        \bullet W^{-1}_{\ReLU}\lr{(}{\lr{(}{\frac{1}{2M}\begin{bmatrix*}
        1 & 1\\
        -1 &-1\\
        1 & -1\\
        -1 & 1
        \end{bmatrix*}, 0}{)}, \left(\begin{bmatrix*}
        1 & 1 & 0 & 0 \\ 0 & 0 & 1 & 1
        \end{bmatrix*}, 0\right)}{)}
    \end{multline*}
approximates the product of any $x,y\in[-M,M]$ with an error of $\varepsilon$ and satisfies that
\begin{align*}
    \dim_{\txt{in}}(\phi_\times^{m,M}) =& 2, \quad \dim_{\txt{out}}(\phi_\times^{m,M}) = 1,\\
    M_1(\phi_\times^{m,M}) =& 8, \quad M_L(\phi_\times^{m,M}) = 8,\\
    M(\phi_\times^{m,M}) =& 30m-28,\\
    L(\phi^{m,M}_\times) =& m+1.
\end{align*}
\end{corollary}
\begin{proof}
    The result follows directly from the definition of $\phi_\times^{m,M}$ \eqref{sq NN def} and its size bounds by \eqref{sq bounds} and Proposition \ref{sparse concatenation}.
\end{proof}

\subsection{Matrix Product Neural Network.}
For the next step, we define, for every matrix $A=(a_{i,j})_{i,j}\in\R^{k\times l}$ and every vector
\[
v = (v_{1,1},\ldots,v_{k,1}, v_{1, 2},\ldots, v_{kl}) \in\R^{kl},
\]
\[
\vect(A) \coloneqq (a_{1,1},\ldots, a_{k, 1}, a_{1,2},\ldots, a_{k,l})\in\R^{kl}\txt{ and } \mat_{k,l}(v) \coloneqq (v_{i,j})_{i,j}\in\R^{k\times l}.
\]
which helps following the dimensions of the input and keep the applied field closer where the matrix-vector multiplication is optimize, as in the renowned Python packages for Deep Learning: TensorFlow \cite{abadi2016tensorflow} and PyTorch \cite{paszke2019pytorch}. 
When there is no room for error, we will omit the subscripts in $\mat$.
For every $M>0$, we also define
\[
K_{d,n,l}(M) \coloneqq \lr{\{}{(A,B)\in\R^{d\times n}\times\R^{n\times l}\st \Norm{A}_2,\Norm{B}_2 \leq M}{\}}
\]
\begin{proposition}[Matrix multiplication NN]\label{Matrix multiplication NN}
    For every $\varepsilon, M> 0$, there is a strict $\ReLU$-NN $\Pi_{d,n,l}^{\varepsilon,M}$ such that
    \begin{equation}\label{multiplication proximity}
    \max_{(A,B)\in K_{d,n,l}(M)}\Norm{AB - \mat\lr{(}{R(\Pi_{d,n,l}^{\varepsilon,M})(\vect(A),\vect(B))}{)}}_2\leq \varepsilon
    \end{equation}
    with
    \begin{align*}
        \dimin(\Pi^{\varepsilon,M}_{d,n,l}) =& n(d+l), \quad \dimout(\Pi^{\varepsilon,M}_{d,n,l}) = dl,\\
        M_1(\Pi^{\varepsilon,M}_{d,n,l}) \leq& 8dln, \quad M_L(\Pi^{\varepsilon,M}_{d,n,l}) \leq 8dln,\\
        M(\Pi^{\varepsilon.M}_{d,n,l})\leq& dln\left(30\lr{\lfloor}{\frac{1}{2}\log_2\lr{(}{\frac{2n\sqrt{dl}C^2}{\varepsilon}}{)}}{\rfloor}-28\right),\\
        L(\Pi^{\varepsilon,M}_{d,n,l})=& \lr{\lfloor}{\frac{1}{2}\log_2\lr{(}{\frac{2n\sqrt{dl}C^2}{\varepsilon}}{)}}{\rfloor}+1.
    \end{align*}
    where $C=\max\{1,M\}$.
\end{proposition}
The proof of this proposition from \cite{kutyniok2022theoretical} is based on the parallelization of all the scalar products involved on the usual matrix multiplication.
\begin{proof}
    By Corollary \ref{scalar mult NN def and bound}, we know that there is a NN $\phi^{m,M}_\times$ which approximates the product of two scalars bounded by $M$. Choosing $m = \lr{\lfloor}{\frac{1}{2}\log_2\lr{(}{\frac{2n\sqrt{dl}C^2}{\varepsilon}}{)}}{\rfloor}$, we get that, for every $A=(a_{i,j})_{i,j}\in\R^{d\times n}, B=(b_{i,j})_{i,j}\in\R^{n\times l}$ such that $\Norm{A}_2,\Norm{B}_2\leq C$
    \[
    \Abs{a_{i,j}b_{j,k} - R(\phi^{m,M}_\times)(a_{i,j}, b_{j,k})}\leq\frac{\varepsilon}{n\sqrt{dl}}.
    \]
    since $\Abs{a_{i,j}},\Abs{b_{j,k}}\leq C$. Calling
    \[
    D_{i,j,k}\in\R^{2\times n(d + l)} \txt{ such that } D_{i,j,k}(\vect(A),\vect(B)) = (a_{i,j}, b_{j,k}),
    \]
    let us define the NN
    \[
    \varphi_{i,j,k}^{m,M} = \phi^{m,M}_\times \bullet W^{-1}_{\ReLU}((D_{i,j,k},0)).
    \]
    and the scalar product of the $i$-th row by the $k$-th column:
    \[
    \psi^{m,M}_{i,k} = W^{-1}_{\ReLU}((1_n, 0))\bullet P(\varphi^{m,M}_{i,1,k},\ldots,\varphi^{m,M}_{i,n,k})\bullet W^{-1}_{\ReLU}\lr{(}
    {\lr{(}
    {\Eye_{n(d+l)}^{n},0}
    {)}}
    {)}
    \]
    where $1_n = (1,\ldots, 1)\in\R^n$ and
    \[
    \Eye^n_{n(d+l)}=\begin{bmatrix*}
        \Eye_{n(d+l)}\\
        \vdots\\
        \Eye_{n(d+l)}
    \end{bmatrix*}\in\R^{n^2(d+l)\times n(d+l)}.
    \]
    Finally, our aimed NN is then
    \[
    \Pi^{m,M}_{d,n,l}= P(\psi^{m,M}_{1,1},\ldots,\psi^{m,M}_{1, l}, \psi^{m,M}_{2, 1},\ldots,\psi^{m,M}_{d, l})\bullet W^{-1}_{\ReLU} \lr{(}{\lr{(}{\Eye_{n(d+l)}^{dl},0}{)}}{)}.
    \]
    which, keeping track of the definition of $\psi_{i,k}^{m,M}$, using the size bounds of Corollary \ref{scalar mult NN def and bound} and the norm 0 properties in Lemma \ref{norm 0 prop}, we get that
    \begin{align*}
        \dim_{\txt{in}}(\Pi^{m,M}_{d,n,l}) =& n(d+l), \quad \dim_{\txt{out}}(\Pi^{m,M}_{d,n,l}) = dl,\\
        M_1(\Pi^{m,M}_{d,n,l}) \leq& 8dln, \quad M_{L(\Pi^{m,M}_{d,n,l})}(\Pi^{m,M}_{d,n,l}) \leq 8dln,\\
        M(\Pi^{m.M}_{d,n,l})\leq& dln(30m-28),\\
        L(\Pi^{m,M}_{d,n,l})=& m+1.
    \end{align*}
    The proof ends using that for every $C = (c_{i,j})_{i,j}\in\R^{d\times l}$, we get
    \[
    \Norm{C}_2 \leq \sqrt{dl}\max_{i,j}\Abs{c_{i,j}}.
    \]
    Indeed, if we call $c_k = (c_{1,k},\dots,c_{d,k})$ for $k = 1,\dots,l$ and let $x = (x_1,\dots,x_l)\in\R^l$,
    \begin{equation*}
        \begin{split}
            \Abs{C x}_2 = \Abs{\sum_{k=1}^l x_k c_k}_2
            \leq \sum_{k=1}^l \Abs{x_k}\Abs{c_k}_2
            \leq \sqrt{d}\max_{i,j}\Abs{c_{i,j}} \Abs{x}_1
            \leq \sqrt{dl}\max_{i,j}\Abs{c_{i,j}} \Abs{x}_2
        \end{split}
    \end{equation*}
    where we used that for any $n\in\N$ and $v\in\R^n$,
    \begin{align*}
        \Abs{v}_2 \leq \sqrt{n}\Abs{v}_\infty, && \Abs{v}_1 \leq \sqrt{n} \Abs{v}_2
    \end{align*}
    and that, for any $k\in\{1,\dots, l\}$,
    \[
    \Abs{c_k}_\infty \leq \max_{i,j}\Abs{c_{i,j}}.
    \]
\end{proof}

\subsection{Matrix Inversion Neural Network.}

As already mentioned in Proposition \ref{B invertible}, our matrix of interest $B\in\R^{d\times d}$ is in $\mathcal{I}_d(\alpha, \delta)$ for some $\alpha > 0$ and $\delta \in [0,1)$ so we rely on the so called Neumann series to calculate the inverse of $B$. To this end, we use the following identity valid for any matrix $A\in\R^{d\times d}$ and $N \in \N$:
\begin{equation}\label{iterative neumann}
    \begin{split}
        \sum_{k=0}^{2^{N+1}-1} A^k =& \sum_{k=0}^{2^N-1}A^k + \sum_{k=2^N}^{2^{N+1}-1}A^k \\
        =&\sum_{k=0}^{2^N-1}A^k + A^{2^N}\sum_{k=2^N}^{2^{N+1}-1}A^{k-2^N}\\
        =& \sum_{k=0}^{2^N-1}A^k+A^{2^N}\sum_{k=0}^{2^N-1}A^k \\
        =& \left(\Eye_d+A^{2^N}\right)\sum_{k=0}^{2^N-1}A^k
    \end{split}
\end{equation}
Consequently, by induction, we also get:
\begin{equation}\label{product neumann}
    \sum_{k=0}^{2^{N+1}-1} A^k = \prod_{k=0}^N\left(\Eye_d + A^{2^k}\right).
\end{equation}
The identity \eqref{iterative neumann} is more operationally efficient than the direct formula: if we already evaluated $\sum\limits_{k=0}^{2^{N}-1}A^k$ and $A^{2^{N-1}}$, we only need two matrix product, one to evaluate $A^{2^N}$ and the other one to get $\sum\limits_{k=0}^{2^{N+1}-1}A^k$.
\begin{theorem}[Matrix inversion NN]\label{matrix inversion NN}
    For every $\varepsilon \in (0,1/4)$ and $\alpha > 0$, there is a strict $\ReLU$-NN $\Upsilon^{\varepsilon,\alpha}_{d,\delta}$ such that for every $\delta \in [0, 1)$ and $B \in\cali{I}_d(\alpha, \delta)$
    \[
    \Norm{B^{-1} - \mat(R(\Upsilon^{\varepsilon,\alpha}_{d,\delta})(\vect(B)))}_2 \leq \varepsilon.
    \]
    satisfying
    \begin{align*}
        \dim_{\txt{in}}(\Upsilon^{\varepsilon,\alpha}_{d,\delta}) =&\dim_{\txt{out}}(\Upsilon^{\varepsilon,\alpha}_{d,\delta})= d^2\\
        M(\Upsilon^{\varepsilon,\alpha}_{d,\delta})\leq& n(\varepsilon/\alpha,\delta)(60d^3(N(\varepsilon/\alpha,\delta)-1)+2d^2)+d^3(12N(\varepsilon/\alpha,\delta)-2)+4d^2+2d,\\
        L(\Sigma_{d,N}^\varepsilon) =& N(\varepsilon/\alpha,\delta)(n(\varepsilon/\alpha,\delta)+2)-2
    \end{align*}
    where
    \begin{align*}
        N(\varepsilon,\delta)&\coloneqq \Ceil{\log_2\left( \max \left \{ \frac{\log_2((1-\delta)\varepsilon)}{\log_2(\delta)}, 2 \right \}\right)},\\
        n(\varepsilon,\delta)&\coloneqq2^{N(\varepsilon,\delta)-1} + 1 + \Floor{\frac{1}{2}\log_2\left(\frac{d^2}{\varepsilon}\right)}.
    \end{align*}
\end{theorem}

Recall that we defined $\mathcal{I}_d(\alpha, \delta)$ in $\eqref{def invertible set}$ as the set
\[
\mathcal{I}_d(\alpha, \delta) = \left \{ A \in \R^{d \times d} \; \middle | \; \Norm{\Eye_d - \alpha A}_2 \leq \delta \right \}.
\]

\begin{remark}\label{better inversion NN}
    The result and its proof are an improvement of the Theorem 3.8 in \cite{kutyniok2022theoretical} when $\alpha = 1$: we give more precise bounds and we shorten the proof by skipping the construction of intermediate NNs.
\end{remark}

\begin{remark}
    As said already mention on Remark \ref{relation delta alpha}, for a symmetric matrix with positive eigenvalues $B \in \R^{d \times d}$, $N(\varepsilon, \Norm{\Eye_d - \alpha B}_2)$ explodes as $\alpha \rightarrow 0$ in Theorem \ref{matrix inversion NN}. We know that, if $\alpha$ is small enough,
    \[
    \Norm{\Eye_d - \alpha B}_2 = 1 - \alpha \lambda_{\min}
    \]
    where $\lambda_{\min}$ is the smallest eigenvalue of $B$ and, when $\alpha$ small, 
    \begin{align*}
        N(\varepsilon / \alpha, 1 - \alpha \lambda_{\min}) &\geq \log_2 \left ( \frac{\log_2(\lambda_{\min} \varepsilon)}{\log_2(1 - \alpha \lambda_{\min})} \right )\\
        \txt{ and } \quad n(\varepsilon / \alpha, 1 - \alpha \lambda_{\min}) & \geq \frac{1}{2} \left ( \frac{\log_2(\lambda_{\min} \varepsilon)}{\log_2(1 - \alpha \lambda_{\min})} + \log_2 \left ( d \sqrt{\frac{\alpha}{\varepsilon}} \right ) \right )
    \end{align*}
    implying the explosion of the bounds of Theorem \ref{matrix inversion NN}.
\end{remark}

\begin{proof}
    \underline{Step 1:} Neumann series approximating NN of any matrix.

    In this step, we do two things: construct a NN and show it approximates the Neumann series.

    \underline{Step 1.1:} Notation and construction of Neumann series NN.

    Let us fix $\varepsilon\in(0,\frac{1}{4})$ and a matrix $A\in\R^{d\times d}$ such that $\Norm{A}_2<1$ and let us call
    \[
    S_N = \sum_{k=0}^{2^N-1}A^k.
    \]

    To approximate $S_N$ using equations \eqref{iterative neumann} or \eqref{product neumann}, we encounter a problem bounding the spectral norm of the multiplied matrices. We apply recurrently the NN of Proposition \ref{Matrix multiplication NN} and for that we must control the spectral norm of the inputs and the output. A way to achieve it is by ensuring that one of the inputs of the product NN has a spectral norm strictly less than 1. That is why the NNs we are building through this proof approximates the powers of the matrices halved using the formula
    \begin{equation}\label{product neumann normalizado}
        \begin{split}
            S_{N+1} =& 3^{-1} \cdot 2^{\sum_{k=1}^N2^k + 3}\prod_{k=1}^N\left[\left(\frac{A}{2}\right)^{2^k}+2^{-2^k}\Eye_d\right] \frac{3}{8} S_1 \\
            =& 3^{-1} \cdot 2^{2^{N+1} + 1} \prod_{k = 1}^N \left [ \left ( \frac{A}{2} \right )^{2^k} + 2^{-2^k} \Eye_d \right ] \frac{3}{8} S_1
        \end{split}
    \end{equation}
    that follows from equation \eqref{product neumann}. This gives us a representation of $S_{N+1}$ in terms of products of matrices with spectral norm less than a half. The factor $\frac{3}{8}$ help us when working on the special case $N = 2$.

    Approximating $S_1$ with NNs is as easy as defining
    \begin{align*}
        \Sigma_{d, 1}^{\varepsilon} \coloneqq \left(\left(\Eye_{d^2}, \vect(\Eye_d)\right)\right)
    \end{align*}
    which satisfies trivially that
    \begin{align*}
        \mat(R(\Sigma_{d,1}^{\varepsilon})(\vect(A))) = A + \Eye_d = S_1.
    \end{align*}
    With this case solved, we define the building blocks of the final NN:
    \begin{multline}\label{varphi def}
        \varphi_{d}^\varepsilon \coloneqq P\left(\Pi_{d,d,d}^{\varepsilon,1} \bullet W^{-1}_{\ReLU}\left(\left(\frac{1}{2}\begin{bmatrix*}
        \Eye_{d^2}\\ \Eye_{d^2}
    \end{bmatrix*},0\right)\right), W^{-1}_{\ReLU} \left(\left(\frac{3}{8}\Eye_{d^2}, \frac{3}{8}\vect(\Eye_d)\right)\right)\bullet 
    \phi^{\Id}_{d^2,L(\Pi_{d,d,d}^{\varepsilon,1})} \right) \\
     \bullet W^{-1}_{\ReLU}\left(\left( \begin{bmatrix*}
        \Eye_{d^2}\\ \Eye_{d^2}
    \end{bmatrix*}, 0\right)\right)
    \end{multline}
    and
    \begin{equation}\label{psi def}
        \psi_{d,k}^\varepsilon \coloneqq P(\Pi_{d,d,d}^{\varepsilon/4,1}, \Pi_{d,d,d}^{\varepsilon/4,1}) \bullet W^{-1}_{\ReLU} \left ( \left ( \begin{bmatrix*}
        \Eye_{d^2} & 0\\
        \Eye_{d^2} & 0\\
        \Eye_{d^2} & 0\\
        0 & \Eye_{d^2}
    \end{bmatrix*}, \begin{bmatrix*}
        0\\0\\ 2^{-2^k}\vect\left(\Eye_d\right)\\0
    \end{bmatrix*}\right)\right)
    \end{equation}
    where $k\in\N$.
    We define therefore the following NN for $N \in \N$:
    \begin{equation}\label{Phi def}
        \Phi_{d,N}^{\varepsilon}\coloneqq \psi_{d,N-1}^{\varepsilon}\odot\psi_{d,N-2}^{\varepsilon}\odot\ldots\odot\psi_{d,2}^{\varepsilon}\odot\psi_{d,1}^{\varepsilon}\odot\varphi^\varepsilon_d.    
    \end{equation}
    When $N = 1$, we assume that $\Phi_{d, N}^\varepsilon = \varphi_d^\varepsilon$. We claim that it approximates $3 \cdot 2^{-(2^{N}+1)}S_N$. To prove that, we define for convenience
    \[
    \begin{bmatrix*}
        A_{2^N}\\
        \sigma_N
    \end{bmatrix*} \coloneqq \mat\left(R(\Phi_{d,N}^\varepsilon)(\vect(A))\right) \txt{ and } \overline{S}_N \coloneqq 3 \cdot 2^{-(2^{N}+1)}S_N \txt{ for all } N\in\N.
    \]
    
    \begin{figure}
        \centering
        \begin{tikzpicture}[scale = 1]
            \node at (.7,0) (A) {$A$};
            \node at (3, .5) (Pi1) [draw, minimum height = 20, minimum width = 20, outer ysep = -3] {$\Pi$};
            \node at (3, -.5) (Id1) [draw, minimum height = 20, minimum width = 20] {$\phi^{\Id}$};
            \node at (Pi1 -| 5.6,0) (A2) {$(\frac{A}{2})^2$};
            \node at (Id1 -| 5.6,0) (S1) {$\frac{3}{8}S_1$};
            \draw[|-] (A) [rounded corners] -- (1.5,0)  -- (1.5, 0 |- Pi1.north west) to node [above = -2 pt, font = \tiny] {$\cdot 1/2$} (Pi1.north west);
            \draw[|-] (A) [rounded corners] -- (1.5,0)  -- (1.5, 0 |- Pi1.south west) to node [above = -2 pt, font = \tiny] {$\cdot 1/2$} (Pi1.south west);
            \draw[|-] (A) [rounded corners] -- (1.5,0) -- (1.5,0 |- Id1) to node [above = -2 pt, font = \tiny] {$+\Id$} (Id1);
            \draw[] (Pi1) --++ (1.5,0);
            \draw[] (Id1) --++ (1.5,0) node [midway, font = \tiny, above = -2 pt] {$\cdot 3/8$} node [pos = 1, name = finId1] {};
            \draw (1.4, 2) rectangle (finId1 |- 3,-1.05);
            \draw[->] (finId1|- Pi1) -- (A2);
            \draw[->] (finId1|- Id1) -- (S1);
            \node at (3,1.5) (varphi) {$\varphi$};
            \node at (8.3,.5) (Pi2) [draw, minimum height = 20, minimum width = 20, outer ysep = -3] {$\Pi$};
            \node at (8.3,-.5) (Pi3) [draw, minimum height = 20, minimum width = 20, outer ysep = -3] {$\Pi$};
            \draw[|-] (A2) [rounded corners] -- (6.8,.5) -- (6.8,0 |- Pi2.north west) -- (Pi2.north west);
            \draw[|-] (A2) [rounded corners] -- (6.8,.5) -- (6.8,0 |- Pi2.south west) -- (Pi2.south west);
            \draw[|-] (A2) [rounded corners] -- (6.8,.5) -- (6.8,0 |- Pi3.north west) -- (Pi3.north west) node [pos = .5, above = -2 pt, font = \tiny] {$+2^{-2}\Id$};
            \draw[|-] (S1) [rounded corners] -- (6.8,-.5) -- (6.8,0 |- Pi3.south west) -- (Pi3.south west);
            \draw[] (Pi2) --++ (.5,0);
            \draw[] (Pi3) --++ (.5,0) node [pos = 1, name = finPi3] {};
            \draw (6.7,2) rectangle (finPi3 |- 8.6,-1.05);
            \path (6.7,1.5) -- (finPi3 |- 6.7,1.5) node [midway, name = psi] {$\psi_{\cdot,1}$};
            \node at (9.75,.5) (A4) {$\left(\frac{A}{2}\right)^{2^2}$};
            \node at (9.75,-.5) (S2) {$\frac{3}{2^5}S_2$};
            \draw[->] (finPi3|- Pi3) -- (S2);
            \draw[->] (finPi3|- Pi2) -- (A4);
            \draw[|-] (S2) -- (10.5,-.5);
            \draw[dashed] (10.5,-.5) -- (11, -.5);
            \draw[|-] (A4) -- (10.5,.5);
            \draw[dashed] (10.5,.5) -- (11,.5);

            \path (A2) --++ (0,-4) node [pos = 1] (Ak) {$\left(\frac{A}{2}\right)^{2^k}$};
            \path (S1) --++ (0,-4) node [pos = 1] (Sk) {$\frac{1}{2^k}S_k$};
            \path (Pi2) --++ (0,-4) node [pos = 1,draw, minimum height = 20, minimum width = 20, outer ysep = -3] (Pi4) {$\Pi$};
            \path (Pi3) --++ (0,-4) node [pos=1,draw, minimum height = 20, minimum width = 20, outer ysep = -3] (Pi5) {$\Pi$};
            \draw[|-] (Ak) [rounded corners] -- (6.8,0 |- Ak) -- (6.8,0 |- Pi4.north west) -- (Pi4.north west);
            \draw[|-] (Ak) [rounded corners] -- (6.8,0 |- Ak) -- (6.8,0 |- Pi4.south west) -- (Pi4.south west);
            \draw[|-] (Ak) [rounded corners] -- (6.8,0 |- Ak) -- (6.8,0 |- Pi5.north west) -- (Pi5.north west) node [pos = .5, above = -2 pt, font = \tiny] {$+2^{-2^k}\Id$};
            \draw[|-] (Sk) [rounded corners] -- (6.8,0|- Sk) -- (6.8,0 |- Pi5.south west) -- (Pi5.south west);
            \draw[] (Pi4) --++ (.5,0);
            \draw[] (Pi5) --++ (.5,0) node [pos = 1, name = finPi5] {};
            \draw (6.7, -2) rectangle (finPi5 |- 8.6,-5.05);
            \path (psi) --++ (0,-4) node[pos = 1] {$\psi_{\cdot, k}$};
            \draw[->,dashed] (finId1 |- Ak) -- (Ak);
            \draw[->,dashed] (finId1 |- Sk) -- (Sk);
            \node at (10, 0|- Pi4) (Ak+1) {$\left(\frac{A}{2}\right)^{2^{k+1}}$};
            \draw[->] (finPi3 |- Pi4) -- (Ak+1);
            \node[right] at (9.21, 0 |- Pi5) (Sk+1) {$\frac{3}{2^{2^{k + 1}+1}}S_{k+1}$};
            \draw[->] (finPi3 |- Pi5) -- (Sk+1);

            \draw[dotted] (1.1, 3) -- (11.4, 3) -- (11.4, -1.45) -- (9, -1.45) -- (9, -5.45) -- (1.1, -5.45) -- cycle;
            \node at (6.25, 2.5) {$\Phi_{\cdot,k+1}$};
        \end{tikzpicture}
        \caption{NN scheme of $\Phi_{\cdot,k+1}$.}
        \label{Phi scheme}
    \end{figure}
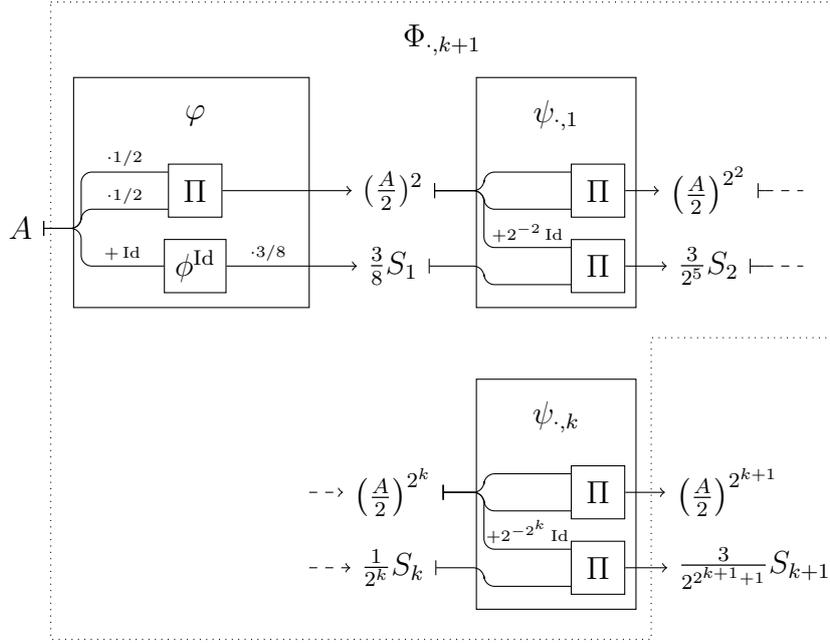

    As it can be seen in the definition of the building blocks of $\Phi_{d, N}^\varepsilon$ or at the scheme \ref{Phi scheme}, the value of $A_{2^N}$ does not depend on $\sigma_k$ for $k = 1, \dots, N$.

    \underline{Step 1.2:} We will prove by induction that
    \begin{align}
        \Norm{A_{2^N} - (A/2)^{2^N}}_2 &\leq \varepsilon \label{power proximity}\\
        \txt{ and } \Norm{A_{2^N}}_2 &\leq \frac{1}{2}. \label{norm bound power approximation}
    \end{align}
    When $N = 1$,
    \begin{align*}
        A_{2^N}^2 &= \mat \left ( R \left ( \Pi_{d, d, d}^{\varepsilon, 1} \bullet W^{-1}_{\ReLU} \left ( \frac{1}{2} \begin{bmatrix*} \Eye_{d^2} & \Eye_{d^2} \end{bmatrix*}, 0 \right ) \right )(\vect(A)) \right ) \\
        &= \mat \left ( R\left ( \Pi_{d, d, d}^{\varepsilon, 1} \left ( \frac{1}{2}\vect(A), \frac{1}{2} \vect(A) \right ) \right ) \right )
    \end{align*}
    and since $\Norm{A/2}_2 < 1$,
    \begin{equation}
        \Norm{A_{2} - (A/2)^2}_2 \leq \varepsilon \tag{1}
    \end{equation}
    implied by the properties of $\Pi_{d, d, d}^{\varepsilon, 1}$, see Proposition \ref{Matrix multiplication NN}. It also follows that
    \begin{equation}
        \Norm{A_2}_2 \leq \varepsilon + \Norm{(A/2)^2}_2 \leq \frac{1}{4} + \frac{1}{4} = \frac{1}{2}. \tag{2}
    \end{equation}
    With (1) and (2), we have proven the base case of the induction.

    Now, suppose that \eqref{power proximity} and \eqref{norm bound power approximation} are satisfied for $N \in \N$.

    We know by the definition of $A_{2^{N+1}}$
    \[
        A_{2^{N+1}} = \mat \left ( R(\Pi_{d, d, d}^{\varepsilon / 4, 1}) \left ( \vect \left ( A_{2^N} \right ), \left ( A_{2^N} \right ) \right ) \right )
    \]
    By the induction hypothesis, $\Norm{A_{2^N}}_2 \leq \frac{1}{2} < 1$, so we can apply \eqref{multiplication proximity}:
    \begin{equation}
        \begin{split}
            \Norm{A_{2^{N+1}}-(A/2)^{2^{N+1}}}_2 \leq& \Norm{A_{2^{N+1}}-\left(A_{2^N}\right)^2}_2+\Norm{\left(A_{2^N}\right)^2-(A/2)^{2^{N+1}}}_2\\
            \leq& \frac{\varepsilon}{4}+\Norm{A_{2^N}-(A/2)^{2^N}}_2\Norm{A_{2^N}+(A/2)^{2^N}}_2\\
            \leq& \frac{\varepsilon}{4} + \varepsilon \left(\frac{1}{2}+\frac{1}{2^{2^N}}\right)\leq\frac{\varepsilon}{4}+\varepsilon\left(\frac{1}{2}+\frac{1}{4}\right)=\varepsilon.
        \end{split} \tag{1}
    \end{equation}
    Note how we used the fact that we are approximating $(A/2)^{2^N}$ instead of $A^{2^N}$ to bound the spectral norm of $A_{2^N}$ by a half unlocking the approximation property of $\Pi_{d, d, d}^{\varepsilon / 4, 1}$ and also at the last line in (1) to prove \eqref{power proximity}.
    
    We get as before, using the triangular inequality,
    \begin{equation}
    \Norm{\widetilde{A/2}^{2^{N+1}}}_2\leq \varepsilon + \Norm{(A/2)^{2^{N+1}}}_2\leq\frac{1}{4}+\frac{1}{2^{2^{N+1}}}\leq\frac{1}{2}. \tag{2}
    \end{equation}
    With (1) and (2), we proved the inductive step.

    \underline{Step 1.3:} Using step 1.2, we will prove by induction that
    \begin{align}
        \Norm{\sigma_N-\overline{S}_N}_2 & \leq \varepsilon \label{neumann proximity}\\
        \txt{and }\Norm{\sigma_N}_2 & \leq 1. \label{norm bound neumann}
    \end{align}

    When $N=1$, since
    \begin{align*}
        \sigma_1 = \mat\left(\frac{3}{8}(\vect(A)+\vect(\Id_d))\right) =\frac{3}{8}(A+\Eye_d),
    \end{align*}
    we get that $\Norm{\sigma_1}_2\leq \frac{3}{4}$ and
    \[
    \Norm{\sigma_1 - \frac{3}{2^{2^1+1}}S_1}_2=\Norm{\sigma_1 - \frac{3}{8}S_1}_2 = 0< \varepsilon.
    \]
    which proves \eqref{neumann proximity} and \eqref{norm bound neumann} for $N = 1$.

    When $N = 2$, since
    \begin{align*}
        \sigma_2 &= \mat \left ( R(\Pi_{d, d, d}^{\varepsilon / 4, 1}) \left ( \vect \left ( A_{2} + 2^{-2} \Eye_{d} \right ), \sigma_1 \right ) \right ) \\
        &= \mat \left ( R(\Pi_{d, d, d}^{\varepsilon / 4, 1}) \left ( \vect \left ( A_{2} + 2^{-2} \Eye_{d} \right ), \frac{3}{8}S_1 \right ) \right )
    \end{align*}
    and $\Norm{\frac{3}{8} S_1}_2 \leq \frac{3}{4}$, we get
    \begin{equation}
        \Norm{\sigma_2 - \overline{S}_2}_2 \leq \frac{\varepsilon}{4} + \Norm{(A_2 + 2^{-2} \Eye_d) \frac{3}{8}S_1 - ((A/2)^2 + 2^{-2} \Eye_d) \frac{3}{8}S_1}_2 \leq \frac{\varepsilon}{4} + \varepsilon \Norm{\frac{3}{8}S_1}_2 \leq \varepsilon \tag{1}
    \end{equation}
    where we used \eqref{power proximity}. Note how the factor $\frac{3}{8}$ allowed us to use the approximation properties of $\Pi_{d, d, d}^{\varepsilon / 4, 1}$ \eqref{multiplication proximity} and prove \eqref{neumann proximity} for $N = 2$ in (1). By the triangular inequality,
    \begin{equation}
        \Norm{\sigma_2}_2 \leq \varepsilon + \Norm{\overline{S}_2}_2 \leq \frac{1}{4} + \left ( \frac{1}{4} + \frac{1}{4} \right ) \frac{3}{4} \leq 1. \tag{2}
    \end{equation}
    With (1) and (2), we proved \eqref{neumann proximity} and \eqref{norm bound neumann} for $N = 2$.
    
    Now, suppose that \eqref{neumann proximity} and \eqref{norm bound neumann} hold for $N \in \N$, $N \geq 2$.

    By definition of $\sigma_{N+1}$,
    \[\sigma_{N+1} = \mat \left ( R \left ( \Pi_{d,d,d}^{\varepsilon/4,1} \right ) \left ( \vect \left ( A_{2^N} + 2^{-2^N} \Eye_d \right ), \vect (\sigma_N) \right ) \right )
    \]
    so
    \[
    \Norm{\sigma_{N+1}}_2 \leq \frac{\varepsilon}{4} + \Norm{\left ( A_{2^N} + 2^{-2^N} \Eye_d \right ) \sigma_N}_2 \leq \frac{1}{4} + \frac{1}{2} + 2^{-2^N} \leq \frac{1}{4} + \frac{1}{2} + \frac{1}{4} = 1
    \]
    because $\Norm{\sigma_N}_2\leq 1$ by induction hypothesis and \eqref{norm bound power approximation}. For \eqref{neumann proximity}, 
    \begin{align*}
        \Norm{\sigma_{N+1}-\overline{S}_{N + 1}}_2 \leq & \Norm{\sigma_{N+1} - \left ( A_{2^N}+2^{-2^N}\Eye_d\right)\sigma_N}_2
        \\
        & + \Norm{\left ( A_{2^N} + 2^{-2^N} \Eye_d \right ) \sigma_N-\overline{S}_{N+1}}_2\\
        \leq & \frac{\varepsilon}{4} + \Norm{\left ( A_{2^N} + 2^{-2^N} \Eye_d \right ) \sigma_N - \left ( \left ( A / 2 \right )^{2^N} + 2^{-2^N}\Eye_d \right ) \overline{S}_N}_2\\
        \leq & \frac{\varepsilon}{4} + \Norm{\left ( A_{2^N} + 2^{-2^N} \Eye_d \right ) \sigma_N - \left ( A_{2^N} + 2^{-2^N} \Eye_d \right ) \overline{S}_N}_2\\
        & + \Norm{\left ( A_{2^N} + 2^{-2^N} \Eye_d \right ) \overline{S}_N - \left( \left ( A / 2 \right )^{2^N} + 2^{-2^N} \Eye_d \right ) \overline{S}_N}_2\\
        \leq & \frac{\varepsilon}{4} + \Norm{A_{2^N} + 2^{-2^N} \Eye_d}_2 \varepsilon + \Norm{\overline{S}_N}_2 \varepsilon \leq \frac{\varepsilon}{4} + \frac{\varepsilon}{2} + \frac{3}{8} \varepsilon \leq \varepsilon
    \end{align*}
    where we used the induction hypothesis \eqref{neumann proximity}, \eqref{power proximity} and that, when $N \geq 2$,
    \begin{align*}
        \Norm{\overline{S}_N}_2 = & \Norm{\prod_{k = 1}^{N - 1} \left [ \left ( \frac{A}{2} \right )^{2^k} + 2^{-2^k} \Eye_d \right ] \left ( \frac{3}{8} S_1 \right )}_2 \\
        \leq & \prod_{k = 1}^{N - 1} \Norm{\left ( A / 2 \right )^{2^k} + 2^{-2^k} \Eye_d}_2 \Norm{\frac{3}{8}S_1}_2\\
        \leq& \frac{3}{4} \prod_{k = 1}^{N - 1} \frac{2}{2^{2^k}} = \frac{3}{4}\frac{2^{N - 1}}{2^{\sum_{k = 1}^{N - 1} 2^k}} = 3\frac{2^{N - 1}}{2^{2^N}} \leq \frac{3}{8}
    \end{align*}
    finishing the induction proof.

    Hence, if we call $C(N)\coloneqq 2^{2^{N} + 1} / 3$ and
    \begin{equation}\label{psi2 def}
        \widetilde{\psi}_{d, N}^\varepsilon \coloneqq W^{-1}_{\ReLU}((C(N) \Eye_{d^2}, 0)) \bullet \Pi_{d, d, d}^{\varepsilon / 4, 1} \bullet W^{-1}_{\ReLU} \left ( \left ( \Eye_{2 d^2}, \begin{bmatrix*} 2^{-2^{N - 1}} \vect (\Eye_d) \\ 0 \end{bmatrix*} \right ) \right ),
    \end{equation}
    the target NN is then
    \begin{equation}\label{eq: Sigma def}
        \Sigma_{d,N}^\varepsilon \coloneqq \widetilde{\psi}_{d, N}^{\varepsilon / C(N)}  \odot \Phi_{d,N-1}^{\varepsilon/C(N)}
    \end{equation}
    approximates $\sum\limits_{k=0}^{2^N-1}A^k$ for any matrix $A$ with norm 1 or less. 
    
    For the next step, we call
    \[
    m(\varepsilon)\coloneqq \lr{\lfloor}{\frac{1}{2}\log_2\lr{(}{\frac{2d^2}{\varepsilon}}{)}}{\rfloor}.
    \]
    which $m(\varepsilon) \geq 3$ for $\varepsilon < 1/4$.

    \underline{Step 2:} bounds of number of layers of $\Sigma_{d, N}^\varepsilon$.

    To do so, we work first with the building blocks of $\Sigma_{d, N}^\varepsilon$.

    By the definitions of $\varphi_d^\varepsilon$ \eqref{varphi def} and Definition \ref{concatenation def}, the Proposition \ref{bounds para concat} and the number of layers of $\Pi_{d, d, d}^{\varepsilon, 1}$ in Proposition \ref{Matrix multiplication NN},
    \begin{equation}\label{eq: number layers varphi}
    L(\varphi_d^\varepsilon) = L(\Pi_{d, d, d}^{\varepsilon, 1}) = m(\varepsilon) + 1.
    \end{equation}
    The same strategy can be used to check that
    \begin{equation}\label{eq: number layers psi}
    L(\psi_{d, k}^\varepsilon) = L(\Pi_{d, d, d}^{\varepsilon / 4, 1}) = m(\varepsilon / 4) + 1,
    \end{equation}
    \begin{equation}\label{eq: number layers psi2}
        L(\widetilde{\psi}_{d, N}^\varepsilon) = L(\Pi_{d, d, d}^{\varepsilon / 4}) = m(\varepsilon / 4) + 1
    \end{equation}
    which follows from its definitions \eqref{psi def} and \eqref{psi2 def}.
    Using then \eqref{eq: number layers varphi} and \eqref{eq: number layers psi} plus the definition of $\Phi_{d, N - 1}^\varepsilon$ \eqref{Phi def}, the Proposition \ref{bounds para concat} tells us
    \begin{equation}\label{eq: number layers Phi}
        L(\Phi_{d, N - 1}^\varepsilon) = (N - 2)(m(\varepsilon / 4) + 1) + m(\varepsilon).
    \end{equation}
    By the definition of $\Sigma_{d, N}^\varepsilon$ \eqref{eq: Sigma def}, the number of layers of $\widetilde{\psi}_{d, N - 1}^\varepsilon$ \eqref{psi2 def} and the last Equation \eqref{eq: number layers Phi}, we conclude that
    \begin{equation}\label{eq: number layers Sigma}
        L(\Sigma_{d, N}^\varepsilon) = (N - 1)(m(\varepsilon / 4) + 1) + m(\varepsilon) = N (m(\varepsilon) + 2) - 2
    \end{equation}
    where we used $m(\varepsilon / 4) = m(\varepsilon) + 1$.

    \underline{Step 3:} bounds of number of weights of $\Sigma_{d,N}^\varepsilon$.

    To do so, we work first with the building blocks of $\Sigma_{d,N}^\varepsilon$.

    Using Propositions \ref{bounds para concat}, \ref{Matrix multiplication NN}, the properties of the Lemma \ref{norm 0 prop} and the definition of $\varphi_d^\varepsilon$ \eqref{varphi def}, we obtain at the first layer:
    \begin{equation}\label{bounds varphi M1}
        \begin{split}
            M_1(\varphi^\varepsilon_d) \leq & M_1 \lr{(}{\Pi^{\varepsilon, 1}_{d,d,d} \bullet W^{-1}_{\ReLU} \lr{(}{\lr{(}{\frac{1}{2} \begin{bmatrix*} \Eye_{d^2} \\ \Eye_{d^2} \end{bmatrix*}, 0}{)}}{)}}{)}\\
            & + M_1 \left ( W^{-1}_{\ReLU} \left ( \left ( \frac{1}{2} \Eye_{d^2}, \frac{1}{2} \vect ( \Eye_d ) \right ) \right ) \bullet \phi^{\Id}_{d^2,L(\Pi_{d,d,d}^{\varepsilon,1})}\right)\\
            \leq& M_1 ( \Pi_{d,d,d}^{\varepsilon,1}) + M_1 (\phi^{\Id}_{d^2,L(\Pi_{d,d,d}^{\varepsilon,1})})\leq 2d^2+8d^3,
        \end{split}
    \end{equation}
    at the last layer:
    \begin{equation}\label{bounds varphi ML}
        \begin{split}
            M_{L(\varphi_d^\varepsilon)}(\varphi_d^\varepsilon) =& M_{L(\Pi_{d,d,d}^{\varepsilon,1})} \left( \Pi_{d,d,d}^{\varepsilon,1} \bullet W^{-1}_{\ReLU} \lr{(}{\lr{(}{\frac{1}{2}\begin{bmatrix*}
                \Eye_{d^2}\\ \Eye_{d^2}
            \end{bmatrix*},0}{)}}{)} \right)\\
            & + M_{L(\Pi_{d,d,d}^{\varepsilon,1})} \left ( W^{-1}_{\ReLU} \left ( \left ( \frac{1}{2} \Eye_{d^2}, \frac{1}{2} \vect(\Eye_d) \right ) \right ) \bullet \phi^{\Id}_{d^2,L(\Pi_{d,d,d}^{\varepsilon,1})}\right)\\
            \leq& 8d^3+2d^2+d,
        \end{split}
    \end{equation}
    and in general:
    \begin{equation}\label{bounds varphi M}
        \begin{split}
            M(\varphi_d^\varepsilon)=&M\lr{(}{\Pi^{\varepsilon,1}_{d,d,d}\bullet W^{-1}_{\ReLU}\lr{(}{\lr{(}{\frac{1}{2}\begin{bmatrix*}
                \Eye_{d^2}\\ \Eye_{d^2}
            \end{bmatrix*},0}{)}}{)}}{)}\\
            &+M\left( W^{-1}_{\ReLU} \left(\left(\frac{1}{2}\Eye_{d^2},\frac{1}{2}\vect(\Eye_d)\right)\right)\bullet\phi^{\Id}_{d^2,L(\Pi_{d,d,d}^{\varepsilon,1})}\right)\\
            \leq& d^3(30m(\varepsilon)-28)+2L(\Pi_{d,d,d}^{\varepsilon,1}))d^2+d\\
            =&d^3(30m(\varepsilon)-28)+2(m(\varepsilon)+1)d^2+d.
        \end{split}
    \end{equation}
    It is immediate to check that
    \begin{multline*}
    P(\Pi_{d,d,d}^{\varepsilon/4,1},\Pi_{d,d,d}^{\varepsilon/4,1}) \bullet W^{-1}_{\ReLU} \left(\left(\begin{bmatrix*}
        \Eye_{d^2} & 0\\
        \Eye_{d^2} & 0\\
        \Eye_{d^2} & 0\\
        0 & \Eye_{d^2}
    \end{bmatrix*}, \begin{bmatrix*}
        0\\0\\ 2^{-2^k}\vect\left( \Eye_d\right)\\0
    \end{bmatrix*}\right)\right)\\= P\left(\Pi_{d,d,d}^{\varepsilon/4,1}\bullet W^{-1}_{\ReLU} \left(\left(\begin{bmatrix*}
        \Eye_{d^2}&0\\ \Eye_{d^2}&0
    \end{bmatrix*},0\right)\right), \right . \\ \left . \Pi_{d,d,d}^{\varepsilon/4,1}\bullet W^{-1}_{\ReLU} \left(\left(\Eye_{2 d^2},\begin{bmatrix*}
        2^{-2^{-k}}\vect(\Eye_d)\\0
    \end{bmatrix*}\right)\right)\right)
    \end{multline*}
    which implies, if we set $\Pi^{\varepsilon/4,1}_{d,d,d}=((A_1,b_1),\ldots, (A_{L(\Pi^{\varepsilon/4,1}_{d,d,d})},b_{L(\Pi^{\varepsilon/4,1}_{d,d,d})}))$, by the definition of $\psi_{d,k}^\varepsilon$ \eqref{psi def} and Proposition \ref{Matrix multiplication NN},
    \begin{equation}\label{M1 bound psi}
        \begin{split}
            M_1(\psi_{d,k}^\varepsilon)=&\Norm{A_1\begin{bmatrix*}
        \Eye_{d^2}&0\\ \Eye_{d^2}&0
    \end{bmatrix*}}_0+\Norm{b_1}_0 \\&+ \Norm{A_1\begin{bmatrix*}
        \Eye_{d^2}&0\\0&\Id_{d^2}
    \end{bmatrix*}}_0 + \Norm{A_1\begin{bmatrix*}
        2^{-2^{-k}}\vect(\Eye_d)\\0
    \end{bmatrix*}+b_1}_0\\
    \leq& 2(\Norm{A_1}_0+\Norm{b_1}_0)+\Norm{A_1}_0\leq 3M_1(\Pi_{d,d,d}^{\varepsilon,1})\leq 24d^3.
        \end{split}
    \end{equation}
    If follows, again, from Proposition \ref{bounds para concat} and \ref{Matrix multiplication NN} and the definition of $\psi_{d,k}^\varepsilon$ \eqref{psi def} that the number of weights at the last layer is:
    \begin{equation}\label{psi bounds ML}
        \begin{split}
            M_{L(\psi_{d,k}^\varepsilon)}(\psi_{d,k}^\varepsilon) =& M_{L(\Pi^{\varepsilon/4,1}_{d,d,d})}\left(\Pi_{d,d,d}^{\varepsilon/4,1}\bullet W^{-1}_{\ReLU} \left(\left(\begin{bmatrix*}
                \Eye_{d^2}&0\\ \Eye_{d^2}&0
            \end{bmatrix*},0\right)\right)\right)\\
                & + M_{L(\Pi^{\varepsilon/4,1}_{d,d,d})} \left ( \Pi_{d,d,d}^{\varepsilon/4,1} \bullet W^{-1}_{\ReLU} \left(\left( \Eye_{2 d^2}, \begin{bmatrix*}
                2^{-2^{-k}}\vect(\Eye_d)\\0
            \end{bmatrix*}\right)\right)\right)\\
            =& 2M_{L(\Pi^{\varepsilon/4,1}_{d,d,d})}\left(\Pi_{d,d,d}^{\varepsilon/4,1}\right)\leq 16d^3
        \end{split}
    \end{equation}
    and at the whole network:
    \begin{equation}\label{psi bounds M}
        \begin{split}
            M(\psi_{d,k}^\varepsilon) =&M\left(\Pi_{d,d,d}^{\varepsilon/4,1}\bullet W^{-1}_{\ReLU} \left(\left(\begin{bmatrix*}
            \Eye_{d^2}&0\\ \Eye_{d^2}&0
            \end{bmatrix*},0\right)\right)\right)\\
            &+M \left ( \Pi_{d,d,d}^{\varepsilon / 4,1} \bullet  W^{-1}_{\ReLU} \left ( \left ( \Eye_{2 d^2}, \begin{bmatrix*} 2^{-2^{-k}} \vect(\Eye_d) \\ 0 \end{bmatrix*} \right ) \right ) \right ) \\
            \leq& 2 M \left ( \Pi_{d, d, d}^{\varepsilon / 4, 1} \right ) + M_1 \left ( \Pi_{d, d, d}^{\varepsilon / 4, 1} \right )\\
            \leq& 2 d^3 (30 m(\varepsilon / 4) - 28) + 8 d^3 = 2 d^3 (30 m(\varepsilon / 4) - 24).
        \end{split}
    \end{equation}
    We get just as in \eqref{psi bounds ML}, \eqref{psi bounds M} that the number of weights at the first layer is
    \begin{equation}\label{psi2 bounds M1}
        M_1(\widetilde{\psi}_{d,N}^\varepsilon) \leq 2M_1(\Pi_{d,d,d}^{\varepsilon/4,1})\leq 16d^3,
    \end{equation}
    at the last layer
    \begin{equation}\label{psi2 bounds ML}
        M_{L(\widetilde{\psi}_{d,N}^\varepsilon)} \leq M_{L(\Pi_{d,d,d}^{\varepsilon/4,1})}(\Pi_{d,d,d}^{\varepsilon/4,1}) \leq 8d^3,
    \end{equation}
    and at the whole network
    \begin{equation}\label{psi2 bounds M}
        M(\widetilde{\psi}_{d,N}^\varepsilon) \leq M(\Pi_{d,d,d}^{\varepsilon/4,1})+M_1(\Pi_{d,d,d}^{\varepsilon/4,1})\leq d^3(30m(\varepsilon/4)-20).
    \end{equation}
    It is then routine using the Proposition \ref{bounds para concat}, the definition of $\psi_{d, N-1}^\varepsilon$ \eqref{Phi def} and the combination of the previous equations \eqref{bounds varphi M1}--\eqref{psi2 bounds M} to check that the number of weights at last layers is
    \begin{equation}\label{phi bounds M1}
        M_{L(\Phi_{d,N-1}^\varepsilon)}(\Phi_{d,N-1}^\varepsilon)= M_{L(\psi_{d,N-2}^\varepsilon)}(\psi_{d,N-2}^\varepsilon)\leq16d^3
    \end{equation}
    and at the whole network
    \begin{equation}\label{phi bounds M}
        \begin{split}
            M(\Phi_{d,N-1}^\varepsilon) \leq& 2d^3(30m(\varepsilon/4)-24)(N-2)+d^3(30m(\varepsilon)-28)\\&+2d^2(m(\varepsilon)+1)+2d.
        \end{split}
    \end{equation}
    Therefore, knowing that $m(\varepsilon/4) = m(\varepsilon)+1$ and recalling \eqref{psi2 bounds M1}--\eqref{psi2 bounds M} together with the Proposition \ref{bounds para concat}, the weights are
    \begin{equation}\label{Sigma bounds M}
        \begin{split}
            M(\widetilde{\psi}_{d,N-1}^\varepsilon\odot\Phi_{d,N-1}^\varepsilon)\leq& M(\widetilde{\psi}_{d,N-1}^\varepsilon)+M(\Phi_{d,N-1}^\varepsilon)+M_1(\widetilde{\psi}_{d,N-1}^\varepsilon)+M_{L(\Phi_{d,N-1}^\varepsilon)}(\Phi_{d,N-1}^\varepsilon)\\
            \leq& d^3(30m(\varepsilon/4)-20)+2d^3(30m(\varepsilon/4)-24)(N-2)\\
            &+d^3(30m(\varepsilon)-28)+2d^2(m(\varepsilon)+1)+2d+16d^3+16d^3\\
            =&m(\varepsilon)(60d^3(N-1)+2d^2)+d^3(12N-10)+2d^2+2d,
        \end{split}
    \end{equation}
    and number of layers is
    \begin{equation}\label{Sigma bounds L}
        L(\widetilde{\psi}_{d,N-1}^\varepsilon\odot\Phi_{d,N-1}^\varepsilon)= (N-1)(m(\varepsilon/4)+1)+m(\varepsilon)=N(m(\varepsilon)+2)-2.
    \end{equation}
    Finally, we conclude that, using the above Equations \eqref{Sigma bounds M}--\eqref{Sigma bounds L} and \eqref{M1 bound psi}, for $N\geq 2$, the input and output dimensions are 
    \begin{equation}\label{Sigma1 bouonds dims}
        \dim_{\txt{in}}(\Sigma_{d,N}^\varepsilon) =\dim_{\txt{out}}(\Sigma_{d,N}^\varepsilon)= d^2,
    \end{equation}
    the number of weights at the first layer
    \begin{equation}\label{Sigma1 bounds M1}
        M_1(\Sigma_{d,N}^\varepsilon) = M_1(\varphi_d^{\varepsilon/C(N)})\leq 2d^2+8d^3,
    \end{equation}
    and the number of weights at the whole network
    \begin{equation}\label{Sigma1 bounds M}
        M(\Sigma_{d,N}^\varepsilon) \leq m(\varepsilon/C(N))(60d^3(N-1)+2d^2)+d^3(12N-10)+2d^2+2d,
    \end{equation}
    where $C(N)=2^{2^N+1}$ and
    \[
    m(\varepsilon/C(N))=\Floor{\frac{1}{2}\log_2\left(\frac{2C(N)d^2}{\varepsilon}\right)} = 2^{N-1} + 1 + \Floor{\frac{1}{2}\log_2\left(\frac{d^2}{\varepsilon}\right)}.
    \]

    \underline{Step 4:} From the Neumann series to matrix inversion.
    
    Let $\delta \in [0, 1)$, $B\in\mathcal{I}_d(\alpha, \delta)$ and $A\coloneqq \Eye_d - \alpha B$. It can be proved that if \\ $N\geq \log_2\left(\max\left\{\log_\delta\left((1-\delta)\varepsilon\right), 2\right\}\right)$ then
    \[
    \Norm{(\Eye_d - A)^{-1}-\sum_{k=0}^{2^N-1} A^k}_2\leq\frac{\delta^{2^N}}{1-\delta}\leq\varepsilon 
    \]
    as seen in \eqref{neumann error}. We define then
    \begin{align*}
        N(\varepsilon, \delta) & \coloneqq \Ceil{\log_2 \left ( \max \left \{ \frac{\log_2((1 - \delta) \varepsilon)}{\log_2(\delta)}, 2 \right \} \right ) }, \\
        n(\varepsilon, \delta) & \coloneqq 2^{N(\varepsilon, \delta) - 1} + 1 + \Floor{\frac{1}{2}\log_2 \left ( \frac{d^2}{\varepsilon} \right )}
    \end{align*}
    and the target NN $\Upsilon^{\varepsilon,\alpha}_{d,\delta}$:
    \[
    \Upsilon^{\varepsilon,\alpha}_{d,\delta}\coloneqq W^{-1}_{\ReLU}  ((\alpha\Eye_{d^2},0)) \bullet \Sigma_{d,N(\epsilon,\delta)}^{\varepsilon/\alpha}\bullet W^{-1}_{\ReLU} ((-\alpha\Eye_{d^2}, \vect(\Eye_d))).
    \]
    We conclude as always (applying the Proposition \ref{bounds para concat}, \eqref{eq: number layers Sigma} and \eqref{Sigma1 bouonds dims}--\eqref{Sigma1 bounds M}) that the dimensions of input and output are
    \begin{equation*}
        \dim_{\txt{in}}(\Upsilon^{\varepsilon,\alpha}_{d,\delta}) = \dim_{\txt{out}}(\Upsilon^{\varepsilon,\alpha}_{d,\delta})= d^2,
    \end{equation*}
    the number of weights is
    \begin{equation*}
        \begin{split}
            M(\Upsilon^{\varepsilon,\alpha}_{d,\delta})\leq& M(\Sigma_{d,N(\varepsilon/\alpha	,\delta)}^{\varepsilon/\alpha})+M_1(\Sigma_{d,N(\varepsilon/\alpha,\delta)}^{\varepsilon/\alpha})\\
            \leq& n(\varepsilon/\alpha,\delta)(60d^3(N(\varepsilon/\alpha,\delta)-1)+2d^2)+d^3(12N(\varepsilon/\alpha,\delta)-2)+4d^2+2d,
        \end{split}
    \end{equation*}
    and the number of layers are
    \begin{equation*}
        L(\Sigma_{d,N}^\varepsilon) = N(\varepsilon/\alpha,\delta)(n(\varepsilon/\alpha,\delta)+2)-2.
    \end{equation*}
\end{proof}

\section{Other Ways of Solving PDEs.}

As discussed though out all this chapter, Theorem \ref{matrix inversion NN} implies that NN can approximate solutions of elliptic PDEs. We have done this by using the Galerkin method but there are other possible strategies.

Within those strategies, instead of approximating individual solutions or families of solutions, we can approximate the resolver function with DeepONets or Neural Operators as described in \ref{sec: comments and other results chap 0}. In \cite{lu2019deeponet}, a density theorem for DeepONets states that continuous operators can be approximated by DeepONets, and since some resolvers are continuous, this justifies the use of this architecture. For Neural Operators, density results for a variety of operators between spaces of function are given in \cite{kovachki2021universal} which again shows the potential of this kind operators.

Back to NN, the most common approach to solve PDEs numerically is using PINNs. Let us consider a PDE
\[
\begin{cases}
    F(u) = 0 & \txt{ in } \Omega \\
    G(u) = 0 & \txt{ in } \partial \Omega
\end{cases},
\]
where $F$ and $G$ are differential operators, the method consist of choosing a $\varrho$-NN $\phi$ whose parameters minimize
\[
\int_\Omega \Abs{F(R(\phi))}^2 + \int_{\partial \Omega} \Abs{G(R(\phi))}^2
\]
while $\varrho$ is chosen so $R(\phi)$ has enough regularity for $F$ and $G$. This idea is tested in \cite{raissi2019physics} for classical one dimension PDEs (like Burger's and Schrodinger's equations).
\parindent=0em
\chapter{Approximation Space of Neural Networks}
\label{chap: space NN}

In this chapter, we use as a guide \cite{gribonval2022approximation} to introduce the approximation space of NN. These spaces consist of all the functions in $L^p$ that can be approximated at fixed rate when allowing bigger NNs. By showing the density of NNs in $L^p$ based on \cite{leshno1993multilayer} (see Remark \ref{remark: density theorem changes}), we get that these spaces is a quasi-normed space. Another proven property is the embedding of the Besov spaces in the approximations spaces, in particular, which the smoother a function is in the Besov sense, the faster it can be approximated.

\section{Introductory Notions.}

\begin{deff}[Quasi and Semi Norms]\label{def: quasi and semi norm}
    Let $X$ be a (real) vector space and
    \[
    \func{\Norm{\; \cdot \;}}{X}{[0, \infty)}{x}{\Norm{x}}
    \]
    We say $\Norm{\; \cdot \;}$ is a \emph{quasi-norm} if
    \begin{enumerate}
        \item $\Norm{x} = 0$ implies $x = 0$
        \item $\Norm{\lambda x} = \Abs{\lambda} \Norm{x}$ for all $x \in X$ and $\lambda \in \R$
        \item there exists $C > 1$ such that
        \[
        \Norm{x + y} \leq C ( \Norm{x} + \Norm{y} )
        \]
        for all $x, y \in X$.
    \end{enumerate}
    If such a function exists, we say that $(X, \Normm)$ (or $X$ simply) is a \emph{quasi-normed space}.
    
    We say $\Norm{\; \cdot \;}$ is a \emph{seminorm} if it satisfies b). and is subadditive, that is, satisfies the triangular inequality
    \begin{enumerate}
        \item[d).]  $\Norm{x + y} \leq \Norm{x} + \Norm{y}$ for all $x, y \in X$.
    \end{enumerate}
    We say $\Norm{\; \cdot \;}$ is a \emph{quasi-seminorm} if satisfies b). and c).
\end{deff}

Note that if a map $\Norm{\; \cdot \;}: X \longrightarrow [0, \infty)$ satisfies \textit{a)., b)., c).} with $C \leq 1$, then $\Norm{\; \cdot \;}$ is a norm when $C = 1$ and $X = \set{0}$ otherwise. Also, a norm is specifically a seminorm, a quasi-norm, and a quasi-seminorm. A quasi-norm induces a distance and therefore a topology, by considering the power of an equivalent quasi-norm (see Theorem 1.1 of Chapter 2 in \cite{devore1993constructive}).

\begin{deff}[Quasi-Banach Spaces]\label{def: quasi-Banach}
    Let $(X, \Normm)$ be a quasi-normed space. We say $(X, \Normm)$ (or $X$ simply) is a quasi-Banach space if $X$ with the induced topology of the quasi-norm $\Normm$ is a complete space.
\end{deff}

In this chapter we study the capacity of NN to approximate functions on $L^p(\Omega)^k$. We recall that, given a Lebesgue-measurable set $\Omega \subset \R^{d}$ and $p \in (0, \infty)$, 
\begin{multline*}
L^p(\Omega)^k \coloneqq \left \{ (f_1, \dots, f_k): \Omega \longrightarrow \R^k \; \middle | \; f_i \txt{ is Lebesgue measurable and } \phantom{\int} \right . \\ \left . \int_\Omega \Abs{f_i(x)}^p dx < \infty \; i = 1, \dots, k \right \} / \txt{a.e}
\end{multline*}
\begin{multline*}
    L^\infty(\Omega)^k \coloneqq \left \{ (f_1, \dots, f_k): \Omega \longrightarrow \R^k \; \middle | \; f_i \txt{ is Lebesgue measurable and } \right . \\ \left . \exists M > 0: \Abs{f_i} < M \txt{ a.e for } i = 1, \dots, k \right \} / \txt{a.e.}
\end{multline*}
where a.e. is the equivalence relation defined as follows:

given two measurable functions $f, g: \Omega \longrightarrow \R^k$, we denote $f = g$ a.e if the set \\ $\{x \in \Omega \; | \; f(x) \neq g(x) \}$ has Lebesgue measure 0.

These spaces are accompanied by the maps
\begin{enumerate}
    \item for $p \in (0, \infty)$,
    \[
    \func{\Norm{\; \cdot \;}_{L^p(\Omega)^k}}{L^p(\Omega)^k}{\R}{(f_1, \dots, f_k)}{\left ( \sum_{i = 1}^k \int_\Omega \Abs{f_i(x)}^p dx \right )^\frac{1}{p}}
    \]

    \item for $p = \infty$,
    \[
    \func{\Normm_{L^\infty(\Omega)^k}}{L^\infty(\Omega)^k}{\R}{(f_1, \dots, f_k)}{\max_{i = 1, \dots, k} \essup_{x \in \Omega} \Abs{f_i(x)}}
    \]
    where
    \[
    \essup_{x \in \Omega} f(x) = \inf \set{ M \in \R \st \Abs{\set{ x \in \Omega \st f(x) \leq M}}= 0 },
    \]
\end{enumerate}

which are quasi-norms when $p \in (0, 1)$ and norms if $p \in [1, \infty]$ (see \cite{folland1999real} for a proof).

To keep notation simple, we will use, if there is no room for confusion, 
\[
\Norm{f}_p = \Norm{f}_{L^p(\Omega)^k}.
\]
for any function $f: \Omega \longrightarrow \R^k$.

The next definition allows us to study the relations between different spaces.

\begin{deff}[Embbeding]\label{def: embedding}
    Let $(X, \Normm_X)$ and $(Y, \Normm_Y)$ be two quasi-normed spaces. We say that $X$ is embedded in $Y$, and we write $X \hookrightarrow Y$, if there exists a linear, injective and continuous map $J: X \longrightarrow Y$.
\end{deff}

\begin{proposition}
    Let $T:X \longrightarrow Y$ be linear function between quasi-normed spaces $(X, \Normm_X)$ and $(Y, \Normm_Y)$. Then $T$ is continuous if and only if there is a constant $C > 0$ such that
    \[
    \Norm{Tx}_Y \leq C \Norm{x}_X
    \]
    for every $x \in X$.
\end{proposition}

\begin{remark}
    Maintaining notation of the above Definition \ref{def: embedding}, we can identify $X$ and $J(X)$ where $J(X)$ is equipped with the induced quasi-norm $\Norm{\; \cdot \;}_J = \Norm{\; \cdot \;}_X \circ J^{-1}$. We can then assume, without loosing generality, that $X \subset Y$ and $J = \Eye|_X$ where $\Eye: Y \longrightarrow Y$ is the identity.
\end{remark}

For the propose of establishing an order relationships quasinorms, we introduce the following relations between two functions $f, g: X \longrightarrow \R$ where $X$ is arbitrary:
\begin{enumerate}
    \item[$\bullet$] $f \sleq g$ if there exists $C > 0$ such that $f(x) \leq C g(x)$ for all $x \in X$.
    \item[$\bullet$] $f \sgeq g$ if there exists $C > 0$ such that $C f(x) \geq g(x)$ for all $x \in X$.
    \item[$\bullet$] $f \approx g$ if there exists $C, c > 0$ such that $c f(x) \leq g(x) \leq C f(x)$ for all $x \in X$.
\end{enumerate}

\section{Approximation Spaces.}

Let $X$ be a quasi-Banach space with $\Norm{\; \cdot \;}_X$ its quasi-norm. Given $\Omega \subset X$, we define
\[
E(f, \Omega)_X \coloneqq \inf_{g \in \Omega} \Norm{f - g}_X
\]

\begin{deff}[Approximation Space]
    Let $X$ be a quasi-Banach space and $\Sigma = (\Sigma_n)_{n = 0}^\infty$ a family of subsets $\Sigma_n \subset X$ for $n = 0, 1, \dots$ Consider, for $\alpha	> 0$ and $q \in (0, \infty]$, the map defined for every $f \in X$ as 
    \[
    \Norm{f}_{A^\alpha_q(X, \Sigma)} \coloneqq \begin{dcases}
        \left ( \sum_{n = 1}^\infty \left ( n^\alpha E(f, \Sigma_{n - 1})_X \right )^q \frac{1}{n} \right )^\frac{1}{q} & \txt{ if } q \in (0, \infty)\\
        \sup_{n \in \N} n^\alpha E(f, \Sigma_{n - 1})_X & \txt{ if } q = \infty
    \end{dcases}.
    \]
    Then, we call the \emph{approximation class} to the set
    \[
    A^\alpha_q(X, \Sigma) \coloneqq \left \{ f \in X \; \middle | \; \Norm{f}_{A^\alpha_q(X, \Sigma)} < \infty \right \}.
    \]
\end{deff}

\begin{proposition}
    Let $X$ be a quasi-Banach space and $\Sigma = (\Sigma_n)_{n = 0}^\infty$ a family of subsets $\Sigma_n \subset X$ for all $n \in \N \cup \set{0}$ such that
    \begin{enumerate}[label={(P\arabic*)}]
        \item \label{enu: Sigma0 0}
            $\Sigma_0 = \{0\}$;

        \item \label{enu: Increasing Sigma}
            $\Sigma_{n - 1} \subset \Sigma_{n}$ for $n \in \N$;

        \item \label{enu: Scalar closeness}
            $a \cdot \Sigma_n = \Sigma_n$ for all $a \in \R \setminus \{0\}$ and $n \in \N_0$;

        \item \label{enu: Additivity closeness}
            There is a fixed constant $c \in \N$ with $\Sigma_n + \Sigma_n \subset \Sigma_{cn}$
            for $n = 0, 1, \dots$;

        \item \label{enu: Sigmainfty density}
            $\Sigma_{\infty} := \bigcup\limits_{j = 0}^\infty \Sigma_j$ is dense in $X$.
    \end{enumerate}
    Then, $(A^\alpha_q(X, \Sigma), \Normm_{A^\alpha_q(X, \Sigma)})$ is quasi-Banch space satisfying $A^\alpha_q(X, \Sigma) \hookrightarrow X$ for all $\alpha > 0$ and $q \in (0, \infty]$.
\end{proposition}

Even if $q \in [1, \infty]$, $X$ is a Banach space and $\Sigma$ satisfies \ref{enu: Sigma0 0}--\ref{enu: Sigmainfty density}, there is no guarantee that $(A^\alpha_q(X, \Sigma), \Normm_{A^\alpha_q(X, \Sigma)})$ is a space.

\begin{ex}
    Consider the Banach space $(\R^d, \Abs{\; \cdot \;}_p)$ for $d \geq 2$ and $p \in [1, \infty)$. We want to find an approximation class of this normed space that is not itself a normed space. Let $q \in [1, \infty)$, $\alpha > 0$,
    \[
    \Sigma_n = \set{ x \in \R^d \st \Norm{x}_0 \leq n }
    \]
    for all $n \in \N \cup \set{0}$ and $\Sigma = (\Sigma_n)_{n = 0}^\infty$.
    \ref{enu: Sigma0 0}--\ref{enu: Scalar closeness} are trivially satisfied. It is also immediate to check that $\Sigma_n = \R^d$ for $n \geq d$ and that if $x = (x_1, \dots, x_d) \in \R^d$ with $\Abs{x_{i_1}} \leq \ldots \leq \Abs{x_{i_d}}$ then
    \begin{equation}\label{eq: example error}
    E(x, \Sigma_n)_p = \inf_{y \in \Sigma_n} \Abs{x - y}_p = \left ( \sum_{k = 1}^{d - n} \Abs{x_{i_k}}^p \right )^\frac{1}{p}
    \end{equation}
    where an empty sum must be interpreted as 0. Moreover, from \eqref{eq: example error}, we conclude that $A^\alpha_q(\R^d, \Sigma) = \R^d$ since for every $x \in \R^d$
    \[
    \Norm{x}_{A^\alpha_q(\R^d, \Sigma)} = \left ( \sum_{n = 1}^\infty (n^\alpha E(x, \Sigma_{n - 1})_p)^q \frac{1}{n} \right )^\frac{1}{q} = \left ( \sum_{n = 1}^d (n^\alpha E(x, \Sigma_{n - 1})_p)^q \frac{1}{n} \right )^\frac{1}{q} < \infty.
    \]
    
    Consider $e_1 = (1, 0, \dots, 0), e_2 = (0, 1, 0, \dots, 0) \in \R^d$. Using \eqref{eq: example error}, we get
    \[
    E(e_1, \Sigma_n)_p = E(e_2, \Sigma_n)_p = \delta_{n, 0},
    \]
    and
    \[
    E(e_1 + e_2, \Sigma_n)_p = \begin{dcases}
        2^\frac{1}{p} & \txt{ if } n = 0 \\
        1 & \txt{ if } n = 1 \\
        0 & \txt{ if } n \geq 2
    \end{dcases}
    \]
    implying
    \[
    \Norm{e_1}_{A^\alpha_q(\R^d, \Sigma)} = \Norm{e_2}_{A^\alpha_q(\R^d, \Sigma)} = 1
    \]
    and
    \[
    \Norm{e_1 + e_2}_{A^\alpha_q(\R^d, \Sigma)} = \left ( 2^\frac{q}{p} + 2^{\alpha q - 1} \right )^\frac{1}{q}.
    \]
    If we define
    \[
    f(p) = \left ( 2^\frac{q}{p} + 2^{\alpha q - 1} \right )^\frac{1}{q},
    \]
    we get that $f(1) > 2$ and that $f$ is decreasing for every $q > 0$ and $\alpha > 0$. It can be seen that $2 \in f([1, \infty))$ if and only if
    \[
    \alpha < \beta(q) \coloneqq \frac{1 + \log_2(2^q - 1)}{q},
    \]
    concretely
    \[
    f(r(\alpha, q)) = f \left ( \frac{q}{\log_2(2^q - 2^{\alpha q - 1})} \right ) = 2.
    \]
    Therefore, the triangle inequality is not satisfied for $p \geq 1$ and $\alpha \geq \beta(q)$ or $p \in [1, r(\alpha, q))$ and $\alpha < \beta(q)$:
    \[
    \Norm{e_1}_{A^\alpha_q(\R^d, \Sigma)} + \Norm{e_2}_{A^\alpha_q(\R^d, \Sigma)} = 2 < f(p) = \Norm{e_1 + e_2}_{A^\alpha_q(\R^d, \Sigma)}.
    \]
    The property \ref{enu: Sigmainfty density} also holds because $\Sigma_{d + 1} = \R^d$.
    
    The same counterexample also works for $q = \infty$ and $\max \set{\alpha, \frac{1}{p}} > 1$.
\end{ex}

\section{Besov Spaces.}

\subsection{Main Definitions.} \label{subsection: main definition Besov Spaces}

To introduce the Besov spaces, we need to define some operators.

Given $h \in \R^d$, the translation operator $T_h$ is defined for any function $f$ with domain in $\R^d$ as $T_h(f) = f(\; \cdot \; + h)$. We denote by $\Eye$ the identity operator. For $\Omega \subset \R^d$ and $f: \Omega \longrightarrow \R^k$, we denote
\[
\Delta_h^r(f, \Omega)(x) \coloneqq \begin{cases}
    (T_h - \Eye)^r(f)(x) & \txt{ if } x, x + h, \dots, x + rh \in \Omega \\
    0 & \txt{ otherwise} 
\end{cases}
\]
where $r \in \N$, $h \in \R^d$, $x \in \R^d$ and
\[
(T_h - \Eye)^r = \overbrace{(T_h - \Eye) \circ \dots \circ (T_h - \Eye)}^{r \txt{ times}}.
\]

\begin{deff} \label{def: modulus of smoothness}
    Let $\Omega \subset \R^d$ a measurable set, $f: \Omega \longrightarrow \R^k$ a function and $r \in \N$. The modulus of smoothness of $f$ is
    \[
    \omega_r(f, \Omega)_p(t) \coloneqq \sup_{\Abs{h}_2 \leq t} \Norm{\Delta_h^r(f, \Omega)}_{L^p(\Omega, \R^k)}.
    \]
    For $q \in (0, \infty]$,
    \[
        \Abs{f}_{B^\alpha_{p, q}(\Omega)} \coloneqq \begin{dcases}
        \left ( \int_0^1 \left ( \frac{\omega_{\Ceil{\alpha}}(f, \Omega)_p(t)}{t^\alpha} \right )^q \frac{dt}{t} \right )^\frac{1}{q} & \txt{ if } q \in (0, \infty) \\
        \essup_{t \in (0, 1)} t^{-\alpha} \omega_{\Ceil{\alpha}}(f, \Omega)(t) & \txt{ if } q = \infty
        \end{dcases}
    \]
    which is a seminorm when $p,q \in [1, \infty]$ and a quasi-seminorm otherwise.
\end{deff}
\begin{deff}\label{def: Besov}
    For $\Omega \subset \R^d$, $\alpha > 0$, $p, q \in (0, \infty]$, the Besov space is defined as
    \[
    B_{p, q}^\alpha(\Omega) = \left \{ f \in L^p(\Omega) \; \middle | \; \Abs{f}_{B^\alpha_{p, q}(\Omega)} < \infty \right \}
    \]
    where the map
    \[
    \func{\Norm{\; \cdot \;}_{B^\alpha_{p, q}(\Omega)}}{B_{p, q}^\alpha(\Omega)}{\R}{f}{\Norm{f}_{L^p(\Omega)} + \Abs{f}_{B^\alpha_{p, q}(\Omega)}}
    \]
    is a norm when $p, q \in [1, \infty]$ and a quasi-norm otherwise.
\end{deff}

these spaces are closely related to the well-known Sobolev spaces:

\begin{proposition}
    Let $\Omega \subset \R^d$ and $m \in \N$. Then,
    \[
    B_{p, p}^m(\Omega) \hookrightarrow W^{m, p}(\Omega) \hookrightarrow B_{p, 2}^m(\Omega)
    \]
    if $p \in (0, 2]$ and
    \[
    B_{p, 2}^m(\Omega) \hookrightarrow W^{m, p}(\Omega) \hookrightarrow B_{p, p}^m(\Omega)
    \]
    if $p \in [2, \infty)$
\end{proposition}

The statement of this result can be found on \cite{adams2003sobolev} at paragraph 7.33. Even though this book works with a different definition of the Besov space given in Definition \ref{def: Besov}, by Theorem 7.47 of the same book, they are equivalent when $p \in (0, \infty)$ and $q \in [1, \infty)$.

\subsection{Equivalent Norms for the Besov Space.}

The above norm, even if it gives us an intuitive idea of its relation with \textit{smooth} functions, can be hard to work on.

\begin{proposition} \label{prop: Besov norm equivalence discretization}
    For every $\alpha > 0$ and $p, q\in (0, \infty]$, the Besov norm is equivalent to an infinite Riemann sum:
    \[
    \Normm_{B_{p, q}^\alpha(\Omega)} \approx \Normm_{L^p(\Omega)} + \begin{dcases}
        \left ( \sum_{k = 1}^\infty \left [ 2^{\alpha k} \omega_r(\; \cdot \;, \Omega)_p(2^{-k}) \right ]^q \right )^\frac{1}{q} & \txt{ if } q \in (0, \infty) \\
        \sup_{k \in \N} 2^{\alpha k} \omega_r(\; \cdot \;, \Omega)_p(2^{-k}) & \txt{ if } q = \infty
    \end{dcases}.
    \]
\end{proposition}
\begin{proof}
    Let $f \in B^\alpha_{p, q}(\Omega)$. We start with $q \in (0, \infty)$.
    
    For $\sleq$,
    \[
    \begin{split}
        \int_0^\frac{1}{2} \left ( \frac{\omega_{\Ceil{\alpha}}(f, \Omega)_p(t)}{t^\alpha} \right )^q \frac{dt}{t} &= \sum_{k = 1}^\infty \int_{2^{-(k + 1)}}^{2^{-k}} \left ( \frac{\omega_{\Ceil{\alpha}}(f, \Omega)_p(t)}{t^\alpha} \right )^q \frac{dt}{t} \\
        &\leq \sum_{k = 1}^\infty \int_{2^{-(k + 1)}}^{2^{-k}} \left ( 2^{\alpha (k + 1)}\omega_{\Ceil{\alpha}}(f, \Omega)_p(2^{-k}) \right )^q 2^{k + 1} dt \\
        &= \sum_{k = 1}^\infty \left ( 2^{\alpha (k + 1)}\omega_{\Ceil{\alpha}}(f, \Omega)_p(2^{-k}) \right )^q \\
        &= 2^{\alpha q} \sum_{k = 1}^\infty \left ( 2^{\alpha k}\omega_{\Ceil{\alpha}}(f, \Omega)_p(2^{-k}) \right )^q
    \end{split}
    \]
    where, at the first inequality, we used that $\omega_r(f, \Omega)_r$ is an increasing function.
    We finish using that $\omega_{R}(\; \cdot \;, \Omega) \sleq \Normm_p$:
    \[
    \begin{split}
        \int_\frac{1}{2}^1 \left ( \frac{\omega_{\Ceil{\alpha}}(f, \Omega)_p(t)}{t^\alpha} \right )^q \frac{dt}{t} &\sleq \int_\frac{1}{2}^1 \left ( \frac{\Norm{f}_p}{t^\alpha} \right )^q \frac{dt}{t} \\
        &\leq \Norm{f}_p^q \frac{2^{\alpha q} - 1}{\alpha q}.
    \end{split}
    \]

    For $\sgeq$,
    \[
    \begin{split}
        \int_0^1 \left ( \frac{\omega_{\Ceil{\alpha}}(f, \Omega)_p(t)}{t^\alpha} \right )^q \frac{dt}{t} &= \sum_{k = 0}^\infty \int_{2^{-(k + 1)}}^{2^{-k}} \left ( \frac{\omega_{\Ceil{\alpha}}(f, \Omega)_p(t)}{t^\alpha} \right )^q \frac{dt}{t} \\
        &\geq \sum_{k = 0}^\infty \int_{2^{-(k + 1)}}^{2^{-k}} \left ( 2^{\alpha k}\omega_{\Ceil{\alpha}}(f, \Omega)_p(2^{-(k + 1)}) \right )^q 2^{k} dt \\
        &= \sum_{k = 0}^\infty 2^{-1} \left ( 2^{\alpha k}\omega_{\Ceil{\alpha}}(f, \Omega)_p(2^{-(k + 1)}) \right )^q \\
        &= 2^{-(\alpha q + 1)} \sum_{k = 1}^\infty \left ( 2^{\alpha k}\omega_{\Ceil{\alpha}}(f, \Omega)_p(2^{-k}) \right )^q
    \end{split}
    \]

    The case $q = \infty$ can be proved in an analogous way.
\end{proof}

Besov spaces are heavily related to B-splines.
\begin{deff} \label{eq: def B spline}
For every $x \in \R$, we define
\begin{equation}
\begin{split}
    \beta^{(0)} & \coloneqq \chi_{(0, 1]} \\
    \beta^{(r)}(x) &\coloneqq \frac{1}{r!} \sum_{k = 0}^{r + 1} \binom{r + 1}{k} (-1)^k \varrho_r(x - k) \quad \forall r \in \N \cup \set{0}
\end{split}
\end{equation}
where we recall that $\varrho_r(x) = \ReLU(x)^r$. Moreover, for $x = (x_1, \dots, x_d) \in \R^d$, $j = (j_1, \dots, j_d) \in \Z^d$, $r \in \N$ and $k \in \N \cup \set{0}$, we define
\begin{equation}
\beta^{(r, d)}_{k, j}(x) \coloneqq \beta^{(r - 1)}(2^{k} x_1 - j_1) \cdots \beta^{(r - 1)}(2^{k} x_d - j_d).
\end{equation}
\end{deff}
The B-splines are constructed by repeat convolution:

\begin{proposition} \label{prop: B spline properties convolution}
    Consider $\beta^{(r)}$ as above. Then, for all $r \in \N \cup \set{0}$,
    \begin{equation} \label{eq: B spline is a convolution}
        \beta^{(r + 1)} = \beta^{(r)} * \chi_{[0, 1]}
    \end{equation}
\end{proposition}
\begin{proof}
    Suppose first that $r = 0$. Then, we know
    \[
    \beta^{(1)}(x) = \begin{cases}
        0 & \txt{ if } x \leq 0 \\
        x & \txt{ if } x \in [0, 1] \\
        2 - x & \txt{ if } x \in [1, 2] \\
        0 & \txt{ if } x \geq 2
    \end{cases}
    \]
    and, for $x \in [0, 2]$
    \begin{align*}
        \chi_{(0, 1]} * \chi_{[0, 1]}(x) &= \int_0^1 \chi_{[0, 1]}(x - y) dy \\
        &= \int_{\max\set{0, x - 1}}^{\min\set{0, x}} dy\\
        &= \min\set{0, x} - \max\set{0, x - 1} \\
        &= \begin{cases}
            x & \txt{ if } x \in [0, 1] \\
            2 - x & \txt{ if } x \in [1, 2]
        \end{cases}.
    \end{align*}
    Trivially $\chi_{(0, 1]} * \chi_{[0, 1]} = 0$ outside of $[0, 2]$, and so $\beta^{(1)} = \chi_{(0, 1]} * \chi_{[0, 1]}$.
    
    Let $k \in \set{0, \dots, r + 1}$, $r \geq 1$ and $x \geq k$. Then, it follows
    \[
    \begin{split}
        \varrho_r(\; \cdot \; - k) * \chi_{[0, 1]}(x) &= \int_0^1 \max \set{x - y - k, 0} dy \\
        &= \int_0^{\min \set{x - k, 1}} (x - y - k)^r dy\\
        &= \frac{1}{r + 1} \left ( (x - k)^{r + 1} - (x - \min \set{x - k, 1} - k)^{r + 1} \right ) \\
        &= \frac{1}{r + 1} \left ( (x - k)^{r + 1} - \max \set{x - (k + 1), 0}^{r + 1} \right ) \\
        &= \frac{1}{r + 1} \left ( \varrho_{r + 1}(x - k) - \varrho_{r + 1}(x - (k + 1)) \right ).
    \end{split}
    \]
    Moreover, if $x < k$,
    \[
    \varrho_r( \; \cdot \; - k) * \chi_{[0, 1]} = \int_0^1 \max \set{x - y - k, 0} dy = 0 = \frac{1}{r + 1} (\varrho_{r + 1}(x - k) - \varrho_{r + 1}(x - (k + 1)).
    \]
    We can conclude:
    \[
    \begin{split}
        \beta^{(r)} * \chi_{[0, 1]}(x) &= \frac{1}{r!} \sum_{k = 0}^{r + 1} \binom{r + 1}{k} (-1)^k \varrho_r(\; \cdot \; - k) * \chi_{[0, 1]}(x) \\
        &= \frac{1}{(r + 1)!} \sum_{k = 0}^{r + 1} \binom{r + 1}{k} (-1)^k (\varrho_{r + 1}(x - k) - \varrho_{r + 1}(x - (k + 1))) \\
        &= \frac{1}{(r + 1)!} \sum_{k = 0}^{r + 1} \binom{r + 1}{k} (-1)^k \varrho_{r + 1}(x - k)\\
        &+ \frac{1}{(r + 1)!} \sum_{k = 0}^{r + 1} \binom{r + 1}{k} (-1)^{k + 1} \varrho_{r + 1}(x - (k + 1)) \\
        &= \frac{1}{(r + 1)!} \sum_{k = 1}^{r + 1} \left ( \binom{r + 1}{k} + \binom{r + 1}{k - 1} \right ) (-1)^k \varrho_{r + 1}(x - k) \\
        &+ (-1)^{r + 2} \varrho_{r + 1}(x - (r + 2)) + \varrho_{r + 1}(x) \\
        &= \frac{1}{(r + 1)!} \sum_{k = 0}^{r + 2} \binom{r + 2}{k} (-1)^k \varrho_{r + 1}(x - k) = \beta^{(r + 1)}(x).
    \end{split}
    \]
\end{proof}

Those piecewise polynomials are a partition of unity:
\begin{proposition} \label{prop: B spline properties partition of unity}
    Consider the B-splines $\beta_{k, j}^{(r, d)}$ as in Definition \ref{eq: def B spline}. Then
    \begin{equation} \label{eq: B splines support}
        \supp \beta_{k, j}^{(r, d)} \subset 2^{-k} ( [0, r]^d + j)
    \end{equation}
    and
    \begin{equation} \label{eq: B spline partition of unity}
        \sum_{j \in \Z^d} \beta_{k, j}^{(r, d)}(x) = 1 \qquad \forall x \in \R^d.
    \end{equation}
\end{proposition}
\begin{proof}
    Let us start proving \eqref{eq: B splines support} by induction. 

    To prove this result, we only need to show this property for $\beta^{(r)}$. The case $r = 0$ is obvious. Suppose that
    \[
    \supp\beta^{(r)} \subset [0, r + 1]
    \]
    for $r \in \N \cup \set{0}$. Then, we get by Proposition \ref{prop: B spline properties convolution} that if $x \not \in (0, r + 2)$,
    \[
    \beta^{(r + 1)}(x) = \beta^{(r)} * \chi_{[0, 1]}(x) = \int_0^1 \beta^{(r)}(x - y) dy = 0.
    \]
    Therefore, we conclude \eqref{eq: B splines support}.

    Let us finish with \eqref{eq: B spline partition of unity}.

    We can suppose that $d = k = 1$. We again prove this result using induction over $r \in \N$. When $r = 1$, \eqref{eq: B spline partition of unity} follows immediately. Suppose for some $r \in \N$ that
    \[
    \sum_{j \in \Z} \beta_{1, j}^{(r, 1)}(x) = \sum_{j \in \Z} \beta^{(r - 1)}(x + j)
    \]
    for all $x \in \R$. Then, \eqref{eq: B spline partition of unity} follows using Proposition \ref{prop: B spline properties convolution}:
    \begin{align*}
        \sum_{j \in \Z} \beta_{1, j}^{(r + 1, 1)}(x) &= \sum_{j \in \Z} \beta^{(r)}(x + j) \\
        &= \sum_{j \in \Z} \beta^{(r - 1)} * \chi_{[0, 1]}(x + j) \\
        &= \chi_{[0, 1]} * \left ( \sum_{j \in \Z} \beta^{(r - 1)}(\; \cdot \; + j) \right )(x) \\
        &= (\chi_{[0, 1]} * 1) (x) = 1.
    \end{align*}
\end{proof}
\begin{proof}[Alternative Proof of Proposition \ref{prop: B spline properties partition of unity}]
    Let us start proving \eqref{eq: B splines support}.
    
    To prove this result, we only need to show this property for $\beta^{(r)}$. It is obvious that $\beta^{(r)} = 0$ in $(- \infty, 0]$. To prove that $\beta^{(r)} = 0$ in $[r + 1, \infty)$, consider the following function
    \[
    f(x, y) \coloneqq e^{x y} (1 - e^{- y})^{r + 1}
    \]
    which, by the binomial formula, is the same as
    \[
    f(x, y) = \sum_{k = 0}^{r + 1} \binom{r + 1}{k} (-1)^k e^{y(x - k)}.
    \]
    Then, we get that
    \begin{equation} \label{eq: f derivative as a sum}
        \frac{\partial^r}{\partial y^r} f(x, y) = \sum_{k = 0}^{r + 1} \binom{r + 1}{k} (-1)^k (x - k)^r e^{y(x - k)}
    \end{equation}
    and so
    \[
    \frac{\partial^r f}{\partial y^r}(x, 0) = r! \beta^{(r)}(x)
    \]
    for all $x \geq r + 1$. Therefore, we only need to proof that $\frac{\partial^r f}{\partial y^r}(x, 0) = 0$ for all $x \geq r +1$. Since $(1 - e^{-y}) = y + \mathcal{O}(y^2)$ by Taylor, we get
    \[
    f(x, y) = e^{x y} y^{r + 1} + \mathcal{O}(y^{r + 2}).
    \]
    implying
    \[
    \frac{\partial^r f}{\partial y^r}(x, y) = (r + 1)! e^{x y} y + \mathcal{O}(y^2).
    \]
    and thus \eqref{eq: B splines support}.

    Let us finish with \eqref{eq: B spline partition of unity}.
    
    We can suppose that $d = k = 1$. By \eqref{eq: B splines support}, we get that for all $x \in [0, 1]$
    \begin{align*}
    B_{r + 1}(x) &\coloneqq \sum_{j \in \Z} \beta^{(r + 1, 1)}_{1, j}(x) = \sum_{j = 0}^r \beta^{(r)}(x + j) \\
    &= \sum_{j = 0}^r \frac{1}{r!} \sum_{k = 0}^{r+ 1} \binom{r + 1}{k} (-1)^k \varrho_r(x + j - k) \\
    &= \frac{1}{r!} \sum_{k = 0}^{r + 1} \binom{r + 1}{k} (-1)^k \sum_{j = k}^r (x + j - k)^r.
    \end{align*}
    Consider then the function
    \[
    g(x, y) = \frac{1}{r!} \sum_{k = 0}^{r + 1} \binom{r + 1}{k} (-1)^k \sum_{j = k}^r e^{(x + j - k)y}
    \]
    since it satisfies that
    \[
    \frac{\partial^r g}{\partial y^r}(x, 0) = B_{r + 1}(x).
    \]
    Our aim is then to show that $\frac{\partial^r g}{\partial y^r}(x, 0) = 1$. For this propose, we rewrite $g$ for $y \neq 0$:
    \begin{align*}
    g(x, y) &= \frac{1}{r!} \sum_{k = 0}^{r + 1} \binom{r + 1}{k} (-1)^k e^{(x - k) y} \sum_{j = k}^r (e^y)^j \\
    &= \frac{1}{r!} \sum_{k = 0}^{r + 1} \binom{r + 1}{k}(-1)^k e^{(x - k)y} e^{k y} \frac{1 - (e^y)^{r - k + 1}}{1 -e^y} \\
    &= \frac{e^{x y}}{r!(1 -e^y)} \sum_{k = 0}^{r + 1} \binom{r + 1}{k} (-1)^k (1 - (e^y)^{r - k + 1}) \\
    &= \frac{e^{x y}}{r!(1 -e^y)} \left ( \sum_{k = 0}^{r + 1} \binom{r + 1}{k} (-1)^k - \sum_{k = 0}^{r + 1} \binom{r + 1}{k} (-1)^k (e^y)^{r - k + 1} \right ) \\
    &= \frac{e^{x y}}{r!(1 -e^y)} \left ( (1 - 1)^{r + 1} - (e^y - 1)^{r + 1} \right ) \\
    &= \frac{e^{x y}}{r!} (e^y - 1)^r.
    \end{align*}
    Since both sides of the equation are continuous, we deduce that for all $x, y \in \R$
    \[
    g(x, y) = \frac{e^{x y}}{r!} (e^y - 1)^r
    \]
    and using again that $(e^y - 1) = y + \mathcal{O}(y^2)$,
    \[
    \frac{\partial^r}{\partial y^r} g(x, y) = \frac{\partial^r}{\partial y^r} \left ( \frac{1}{r!} e^{x y} y^r + \mathcal{O}(y^{r + 2}) \right ) = e^{x y} + \mathcal{O}(y).
    \]
    which tells us $\frac{\partial^r g}{\partial y^r}(x, 0) = 1$.
\end{proof}
We introduce the notation
\[
x^\mu = x_1^{\mu_1} \dots x_d^{\mu_d}
\]
for all $x = (x_1, \dots, x_d) \in \R^d$, $\mu = (\mu_1, \dots, \mu_d) \in \N \cup \set{ 0 }$ and
\[
\mathcal{P}^{(r, d)} \coloneqq \set{ x \longmapsto \sum_{ \substack{\mu \in \N \cup \set{0} \\ \Abs{\mu}_1 < r} } a_\mu x^\mu \st a_\mu \in \R }
\]
where $r \in \N \subset \set{0}$ is the maximum degree and $d \in \N$ the number of variables. Let us see how Besov spaces are related to the functions locally approximated by polynomials:

\begin{proposition} \label{prop: Besov norm equivalence on the unit cube}
    Let $r \in \N$ and $\lambda = \min \set{ r + \frac{1}{p} - 1, r }$ and $\Omega = (0, 1)^d$. Then, for every $\alpha \in (0, \lambda)$ and $p, q \in (0, \infty]$, we have
    \[
    \Normm_{B^\alpha_{p, q}(\Omega)} \approx \Normm_p + \begin{dcases}
        \left ( \sum_{k = 0}^\infty \left [ 2^{\alpha k} \inf_{P \in \Sigma^{(r, d)}_k} \Norm{\; \cdot \; - P}_{L^p(\Omega)} \right ]^q \right )^\frac{1}{q} & \txt{ if } q \in (0, \infty) \\
        \sup_{k \in \N \cup \set{ 0 }} 2^{\alpha k} \inf_{P \in \Sigma^{(r, d)}_k} \Norm{\; \cdot \; - P}_{L^p(\Omega)} & \txt{ if } q = \infty
    \end{dcases}
    \]
    where, for all $r, d, k \in \N$,
    \[
    \Sigma^{(r, d)}_k = \Span \set{ \beta_{k, j}^{(r, d)}|_\Omega}.
    \]
\end{proposition}
Due to space constraints, the following proof omits important details that can be consulted in \cite{devore1988interpolation}.
\begin{proof}[Sketch of the Proof]
    Fix $r \in \N$, $p, q \in (0, \infty]$ and $\alpha > 0$. For convenience, we call
    \[
    s_k^r(f) \coloneqq \inf_{P \in \Sigma_k^{(r, d)}} \Norm{f - P}_{L^p(\Omega)}.
    \]
    
    \underline{Step 1:} $\omega_r(\; \cdot \;, \Omega)(2^{-k}) \sgeq s_k^r$ where the implicit constant does not depend on $k \in \N$.

    We introduce the notation
    \begin{align*}
        D_k &\coloneqq \set{ (0, 2^{-k})^d + j \subset (0, 1)^d \st j \in \Z^d } \\
        \Pi_k &\coloneqq \set{ f: \R \longrightarrow \R \st \forall I \in D_k, \exists P \in \mathcal{P}^{(r, d)}: f|_I = P|_I } \\
        \txt{and } \Lambda_k &\coloneqq \set{ j \in \Z^d \st \supp \beta^{(r, d)}_{k, j} \cup Q \neq \emptyset }
    \end{align*}
    We can find the dual of any $\beta_{k, j}^{(r, d)}$, that is, for every $r, k \in \N$ and $j \in \Z^d$, there is a function $\alpha_{k, j}^{(r, d)}: \Sigma_k^{(r, d)} \longrightarrow \R$ such that $\alpha_{k, j}^{(r, d)}(\beta_{k, i}^{(r, d)}) = \delta_{i, j}$. As shown in \cite{de1973spline}, there are some coefficients $\lambda_\mu \in \R$ and for any values $\xi_j \in \supp \beta_{k, j}$ such that
    \[
    \alpha_{k, j}^{(r, d)}(f) = \sum_{\substack{\mu \in (\N \cup \set{0})^d \\ \Abs{\mu}_1 < r}} \lambda_\mu D^\mu f(\xi_j) \quad \txt{ for all } f \in \Sigma_k^{(r, d)}.
    \]
    With this definition of $\alpha_{j, k}^{(r, d)}$, we can define a embedding between $\Pi_k$ and $\Sigma_k^{(r, d)}$:
    \[
    \func{Q_k}{\Pi_k}{\Sigma_k^{(r, d)}}{f}{\sum_{j \in \Lambda_k} \alpha_{k, j}^{(r, d)}(f) \beta_{k, j}^{(r, d)}}.
    \]
    It can be proved that there exists $c_1, c_2 > 0$ such that for every $S \in \Pi_k$ and $I \in D_k$
    \begin{align} \label{eq: Qk properties}
        \Norm{Q_k(S)}_{L^p(I)} \leq \Norm{S}_{L^p(\widetilde{I})} && \txt{ and } && \Norm{S - Q_k(S)}_{L^p(I)} \leq c_2 \inf_{P \in \mathcal{P}^{(r, d)}}\Norm{S - P}_{L^p(\widetilde{I})}
    \end{align}
    where
    \[
    \widetilde{I} = \bigcup \set{2^{-k}([0, r]^d + j) \st j \in \Z^d \txt{ and } 2^{-k}([0, r]^d + j) \cap I \neq \emptyset}.
    \]
    We fix a constant $A \geq 1$ from now on. For any function $f \in L^p(\Omega)$, we can always find $P \in \mathcal{P}^{(r, d)}$ such that $\Norm{f - P}_{L^p(\Omega)} \leq A E(f, \mathcal{P}^{(r, d)})_{L^p(\Omega)}$. And so, for a fixed $f \in L^p(\Omega)$ we define $S_k(f) \in \Pi_k$ such that
    for every $I \in D_k$, there is $P \in \mathcal{P}^{(r, d)}$ satisfying $S_k(f)|_I = P|_I$ and $\Norm{f - P}_{L^p(\Omega)} \leq A E(f, \mathcal{P}^{(r, d)})_{L^p(I)}$. It follows from properties \eqref{eq: Qk properties} that $S_k(f)$ satisfies
    \[
    \Norm{Q_k(S)}_{L^p(I)} \leq \Norm{S}_{L^p(\widetilde{I})}
    \]
    \[
    s_k^r(f) \leq \Norm{f - Q_k(S_k(f))}_{L^p(\Omega)} \sleq \omega_r(f, \Omega)_p(2^{-k})
    \]

    \underline{Step 2:} for every $k \in \N$, $\mu \leq \min \set{1, p}$ and $\lambda = \min \set{r, r - 1 + 1/p }$
    \[
        \omega_{r}(\; \cdot \;, \Omega)_p(2^{-k}) \sleq 2^{-k \lambda} \left ( 2^{-\lambda} \Normm_{L^p(\Omega)}^\mu + \sum_{j = 0}^k [2^{j \lambda} s_j^r(f)]^\mu \right )^\frac{1}{\mu} 
    \]
    where the implicit constant does not depend on $k$.

    For $f \in L^p(\Omega)$, we find $U_j \in \Sigma_j^{(r, d)}$ such that $\Norm{f - U_k} = s_k^r(f)$ and we call $u_j = U_j - U_{j - 1}$ for $j = 0, \dots, k$ and $u_{-1} = 0$. We get for $h \in \R^d$
    \begin{equation} \label{eq: rewrite of diference}
        (T_h - \Eye)^r f = (T_h - \Eye)^r (f - U_k) + \sum_{j = 0}^k (T_h - \Eye)^r u_j
    \end{equation}
    where $T_h$ and $\Eye$ are as introduced in Subsection \ref{subsection: main definition Besov Spaces}. When $\Abs{h}_2 \leq r^{-1} 2^{-k}$, it can be seen that there is a constant $C = C(d, r) > 0$
    \[
    \Norm{(T_h - \Eye)^r \beta^{r, d}_{j, i}}_{L^p(\Omega(r h))}^p \leq C (\Abs{h}_2 2^j)^{\lambda p} 2^{- j d}
    \]
    where $\Omega(r h) = (\Omega - r h) \cap \Omega$ and so, denoting $s_{-1}^r(f) = \Norm{f}_p$,
    \begin{align*}
        \Norm{(T_h - \Eye)^r u_j}_{L^p(\Omega(r h))} &\leq C_1 \left ( \sum_{i \in \Lambda_j} \Abs{\alpha^{(r, d)}_{j, i}(u_j)}^p \Norm{(T_h - \Eye)^r \beta^{r, d}_{j, i}}_{L^p(\Omega(r h))}^p \right )^\frac{1}{p} \\
        &\leq C_2 (\Abs{h}_2 2^j)^\lambda \left ( \sum_{i \in \Lambda_j} \Abs{\alpha^{(r, d)}_{j, i}(u_j)}^p 2^{- j d} \right )^\frac{1}{p} \\
        &\leq C_3 (\Abs{h}_2 2^j)^\lambda \Norm{u_j}_{L^p(\Omega)} = C_3 (\Abs{h}_2 2^j)^\lambda \Norm{f - U_{j - 1} - (f - U_j)}_{L^p(\Omega)} \\
        &\leq C_4 (\Abs{h}_2 2^j)^\lambda (s_j^r(f) + s_{j - 1}^r(f)).
    \end{align*}
    In the second inequality we used that
    \[
    \left ( \sum_{i \in \Lambda_j} \Abs{\alpha_{j , i}^{(r, d)}(S)}^p 2^{-j d} \right )^\frac{1}{p} \approx \Norm{S}_{L^p(\Omega)}
    \]
    being the implicit constants independent of $j \in \N$. Thus, we get from \eqref{eq: rewrite of diference}
    \[
    \Norm{\Delta^r_h f}_{L^p(\Omega)} \leq C \left ( s_k^r(f)^\mu + \Abs{h}_2^{\mu \lambda} \sum_{j = -1}^k [2^{j \lambda} s_j^r(f)]^\mu \right )^\frac{1}{\mu}
    \]
    where we used $\Norm{(T_h - \Eye)^r \; \cdot \;}_{L^p(\Omega(r h))} \sleq \Normm_{L^p(\Omega)}$. We conclude
    \[
    \omega_r(f, \Omega)_p(2^{-k}) \leq \omega_r(f, \Omega)_p(r^{-1} 2^{-k}) \leq C r^{-\lambda} 2^{-k \lambda} \left (r^\lambda 2^{k \lambda} s_k^r(f)^\mu + \sum_{j = -1}^k [2^{j, \lambda} s_j^r(f)]^\mu \right )
    \]
    just as we wanted.

    \underline{Step 3:} $\Norm{(2^{k \alpha} \omega_r(\; \cdot \;, \Omega)_p(2^{-k}))_{k = 0}^\infty}_{\ell^q} \sleq \Norm{(2^{k \alpha} s_k^r(\; \cdot \;))_{k = - 1}^\infty}_{\ell^q}$ for $\alpha < \lambda$.

    This is a consequence of Step 2 and a discrete Hardy inequality. Indeed, if there are $c > 0, \lambda > \alpha, q \geq \mu$ and $(a_k)_{k = 0}^\infty, (b_k)_{k = 0}^\infty \subset \R$ such that
    \[
    \Abs{b_k} \leq c 2^{- k \lambda} \left ( \sum_{j = 0}^k [2^{j \lambda} \Abs{a_j}]^\mu \right )^\frac{1}{\mu}
    \]
    then there is a constant $C = C(\lambda, \mu, c) > 0$ satisfying
    \[
    \Norm{(2^{\alpha k} b_k)_{k = 0}^\infty}_{\ell^q} \leq C \Norm{(2^{\alpha k} a_k)_{k = 0}^\infty}_{\ell^q}.
    \]

    \underline{Step 4}: Conclusion.

    Step 1 and Step 3 implies the equivalence between the norms using Proposition \ref{prop: Besov norm equivalence discretization} because for all $r, r' \in \N$ such that $r, r' \geq \alpha$,
    \[
    \Norm{(2^{k \alpha} w_{r'}(\; \cdot \;, \Omega)_p(2^{-k})_{k = 0}^\infty}_{\ell^q} \approx \Norm{(2^{k \alpha} w_{r}(\; \cdot \;, \Omega)_p(2^{-k})_{k = 0}^\infty}_{\ell^q},
    \]
    see Theorem 10.1 of Chapter 2 in \cite{devore1993constructive}.
\end{proof}

\section{Neural Network Approximation Space.}

\subsection{Neural Network Approximation Space is a Quasi-normed Space.}

We want to study the approximation capabilities of NN in $L^p(\Omega, \R^k)$ for $p \in (0, \infty)$ and $C_0(\Omega, \R^k)$ when $\Omega \subset \R^d$. Concretely, how the conectivity, the number of neurons and the number of layers influence it (see Definition \ref{NN parameters}).

\begin{deff}\label{def: approx families and spaces}
    Let $p \in (0, \infty)$, $d, k \in \N$ and $\Omega \subset \R^d$ a measurable set. For $n \in \N$ and $X$ a subset of $L_{\loc}^1(\Omega)^k$, we call
    \begin{align*}
        \Mtt_n(X, \varrho) &\coloneqq \set{ R(\phi) \in X \st \phi \txt{ is a } \varrho\txt{-NN}, M(\phi) \leq n}, \\
        \Mtt_0(X, \varrho) &\coloneqq \set{ 0 }
    \end{align*}
    Then, we define for $q \in (1, \infty]$ the approximation classes
    \begin{equation*}
        M_{p, q}^\alpha(\Omega, \varrho)^k \coloneqq A_q^\alpha(L^p(\Omega)^k, (\Mtt_n(L^p(\Omega)^k, \varrho))_{n = 0}^\infty).
    \end{equation*}
\end{deff}

\begin{remark}
    In the paper \cite{gribonval2022approximation}, the authors consider also the spaces
    \[
    \begin{split}
    C_{p, q}^\alpha(\Omega, \varrho, \LL)^k &\coloneqq A_q^\alpha(L^p(\Omega)^k, (\Ctt_n(L^p(\Omega)^k, \varrho, \LL))_{n = 0}^\infty) \\
    N_{p, q}^\alpha(\Omega, \varrho, \LL)^k &\coloneqq A_q^\alpha(L^p(\Omega)^k (\Ntt_n(L^p(\Omega)^k, \varrho, \LL))_{n = 0}^\infty), \\
    \txt{where } \Ctt_n(X, \varrho, \LL) &\coloneqq \set{ R(\phi) \in X \st \phi \txt{ is a } \varrho\txt{-NN}, C(\phi) \leq n, L(\phi) \leq \LL(n)} \\
    \txt{and } \Ntt_n(X, \varrho, \LL) &\coloneqq \set{ R(\phi) \in X \st \phi \txt{ is a } \varrho\txt{-NN}, N(\phi) \leq n, L(\phi) \leq \LL(n) }
    \end{split}
    \]
    for any non-decreasing function $\LL: \N \longrightarrow \N \cup \set{\infty}$ but, as they show, for the results we are concerned, only the supremum of the non-decreasing function matters and the differences between $\Mtt_{p, q}^\alpha(\Omega, \varrho)$, $\Ntt_{p, q}^\alpha(\Omega, \varrho, \LL)^k$ and $\Ctt_{p, q}^\alpha(\Omega, \varrho, \LL)^k$ are minor. Therefore, for a reason of space, we omit those cases. For the same reason, we do not develop the theory on the approximation spaces introduced in the same paper
    \[
    \begin{split}
        C_{\infty, q}^\alpha(\Omega, \varrho, \LL)^k & \coloneqq A_q^\alpha(C_0(\Omega)^k, (\Ctt_n(C_0(\Omega)^k, \varrho, \LL))_{n = 0}^\infty) \\
        N_{\infty, q}^\alpha(\Omega, \varrho, \LL)^k & \coloneqq A_q^\alpha(C_0(\Omega)^k, (\Ntt_n(C_0(\Omega)^k, \varrho, \LL))_{n = 0}^\infty)
    \end{split}
    \]
    where
    \[
    C_0(\Omega)^k \coloneqq \set{ f|_\Omega \st f \in C(\R; \R^k) \txt{ and } \lim_{\Abs{x}_2 \rightarrow \infty} f(x) = 0 }.
    \]
\end{remark}

For the next proposition, we use the following notation. We define for every pair of functions $f: \R^d \longrightarrow \R^k$, $g: \R^{d'} \longrightarrow \R^{k'}$ their tensor product:
    \begin{equation} \label{eq: tensor product}
    \func{f \otimes g}{\R^d \times \R^{d'}}{\R^k \times \R^{k'}}{(x, y)}{(f(x), g(y))}
    \end{equation}
    and if $k = k'$ their tensor sum
    \begin{equation} \label{eq: tensor sum}
    \func{f \oplus g}{\R^d \times \R^{d'}}{\R^k}{(x, y)}{f(x) + g(y)}.
    \end{equation}

\begin{proposition} \label{prop: NN satisfies P1 to P4}
    For every activation function $\varrho: \R \longrightarrow \R$ and a measurable set $\Omega \subset \R^d$, the family $(\Mtt_n(L^p(\Omega, \R^k), \varrho))_{n = 0}^\infty$ satisfies \ref{enu: Sigma0 0}--\ref{enu: Additivity closeness}.
\end{proposition}
\begin{proof}
    \ref{enu: Sigma0 0} and \ref{enu: Increasing Sigma} are trivially satisfied. 
    
    To prove \ref{enu: Scalar closeness}, let $\phi = ((T_1, \alpha_1), \dots, (T_L, \alpha_L))$ be a $\varrho$-NN and $a \in \R \setminus \set{ 0 }$ any non zero scalar. Then, $\phi' = ((T_1, \alpha_1), \dots, (a T_L, \alpha_L))$ is a $\varrho$-NN (since $a T_L$ is still an affine map and $\alpha_L = \Id_{\dimout \phi}$), and clearly satisfies
    \[
        M(\phi') = M(\phi).
    \]
    It is also immediate that $R(\phi') = a R(\phi)$, again, because $\alpha_L = \Id_{\dimout \phi}$. Therefore,
    \[
    a \Mtt_n(L^p(\Omega)^k, \varrho) \subset \Mtt_n(L^p(\Omega)^k, \varrho),
    \]
    For the other content, it suffices to note that $\frac{1}{a}R(\phi)$ is a $\varrho$-NN if $a \neq 0$. So, we proved that this family satisfies \ref{enu: Scalar closeness}.

    To prove \ref{enu: Additivity closeness}, consider two $\varrho$-NN
    \[
    \begin{split}
        \phi_1 &= ((T_1, \alpha_1), \dots, (T_L, \alpha_L)), \\
        \phi_2 &= ((S_1, \beta_1), \dots, (S_K, \beta_K))
    \end{split}
    \]
    such that $\phi_1, \phi_2 \in \Mtt_n(L^p(\Omega)^k, \varrho)$ and $L \geq K$. We split the proof in two cases.
    \begin{itemize}
    \item If $M(\phi_1) \geq L(\phi_1)$, we consider the $\varrho$-NN
        \[
        \phi = ((T_1 \otimes S_1, \alpha_1 \otimes \beta_1), \dots, (T_{L - 1} \otimes S_{L - 1}, \alpha_{L - 1} \otimes \beta_{L - 1}), (T_L \oplus S_L, \Id_k))
        \]
        with $S_\ell = \beta_\ell = \Id_{\dimout \phi_2}$ for $\ell \in \{K + 1, \dots, L\}$. This new $\varrho$-NN satisfies
        \begin{equation}
        \begin{split}
            M(\phi) &\leq M(\phi_1) + M(\phi_2) + (L - K) \dimout \phi_2 \\
            &\leq n + n + L k \leq n(k + 2)
        \end{split}
        \end{equation}
        because
        \[
        W(\phi) = \left ( \left ( \begin{bmatrix} A_1^1 & 0 \\ 0 & A_1^2 \end{bmatrix}, \begin{bmatrix} b_1^1 \\ b_1^2 \end{bmatrix} \right ), \dots, \left ( \begin{bmatrix} A_{L - 1}^1 & 0 \\ 0 & A_{L - 1}^2 \end{bmatrix} \right ), \left ( \begin{bmatrix} A_L^1 & A_L^2 \end{bmatrix}, \begin{bmatrix} b_L^1 + b_L^2 \end{bmatrix} \right ) \right )
        \]
        where
        \[
        \begin{split}
            W(\phi) &= ((A_1^1, b_1^1), \dots, (A_L^1, b_L^1)), \\
            W((S_1, \beta_1, \dots, S_L, \beta_{L})) &= ((A_1^2, b_1^2), \dots, (A_L^2, b_L^2)).
        \end{split}
        \]
        As we want, we get that $R(\phi) = R(\phi_1) + R(\phi_2)$ and so
        \[
        R(\phi_1) + R(\phi_2) = R(\phi) \in \Mtt_{n(k + 2)}(L^p(\Omega)^k, \varrho).
        \]
        
        \item If $M(\phi_1) < L(\phi_1)$ and
        \[
        W(\phi_1) = ((A_1, b_1), \dots, (A_L, b_L))
        \]
        where $A_\ell \in \R^{n_\ell \times n_{\ell - 1}}$ and $b \in \R^{n_\ell}$ for $\ell \in \{ 1, \dots, L \}$ then we get
        \[
        M(\phi_1) = \sum_{\ell = 1}^L \Norm{A_\ell}_0 + \Norm{b_\ell}_0 < \sum_{\ell = 1}^L 1 = L.
        \]
        This implies that, for some $\ell \in \{1, \dots, L\}$, $\Norm{A_\ell}_0 = \Norm{b_\ell} = 0$ and therefore $R(\phi)$ is constant. So, we consider the $\varrho$-NN
        \[
        \phi' = ((S_1, \beta_1), \dots, (S_K', \beta_K))
        \]
        with $S'_K(\; \cdot \;) = S_K(\; \cdot \;) + R(\phi_1)$. We get that $R(\phi') = R(\phi_1) + R(\phi_2)$ and $C(\phi') = C(\phi)$. Then
        \[
        R(\phi_1) + R(\phi_2) = R(\phi') \in \Mtt_{n}(L^p(\Omega)^k, \varrho).
        \]
    \end{itemize}
    Therefore, we proved that, for any $n \in \N$,
    \[
    \Mtt_n(L^p(\Omega)^k, \varrho) + \Mtt_n(L^p(\Omega)^k, \varrho) \subset \Mtt_{n(k+2)}(L^p(\Omega)^k, \varrho).
    \]
    This concludes the proof of Property \ref{enu: Additivity closeness}.
\end{proof}

\begin{theorem}[Density] \label{th: Density of NN}
    Let $\varrho: \R \longrightarrow \R$ be a function. Suppose that
    \begin{enumerate}
        \item \label{enu: almost continuous}
        there exists an open set $U \subset \R$ such that $\R \setminus U$ has Lebesgue measure 0 and $\varrho|_U$ is continuous;

        \item \label{enu: locally bounded}
        $\varrho$ is locally bounded, that is, for every bounded interval $I$, $\varrho|_I$ is bounded.
    \end{enumerate}
    Then, the subspace
    \[
    V_d \coloneqq \Span \set{ \varrho ( \inner{ a, \; \cdot \; } + b ) \st a \in \R^d, b \in \R }
    \]
    is dense in $C(\R^d) = C(\R^d; \R)$ if and only if there does not exist $p \in \bigcup\limits_{r \in \N} \mathcal{P}^{(r, 1)}$ such that $\varrho = p$ a.e.
\end{theorem}

Note that every element in $V_d$ is a strict $\varrho$-NN of two layers.

Here, a familly of functions $S \subset C(\R^d)$ is dense in $C(\R^d)$ if, for every compact $K \subset \R^d$, $S$ is dense in $(C(K), \Normm_{L^\infty(K)})$. We also use the following notation
\[
C^\infty_c(\Omega) \coloneqq \set{ \varphi \in C^\infty(\Omega) \st \supp \varphi \subset \Omega \txt{ is compact} }.
\]
We recall too that the Lebesgue measure of a set $E \subset \R^d$ is defined as
\[
\Abs{E} = \inf \set{ \sum_{k = 1}^\infty \Abs{R_k} \st R_k \subset \R^d \txt{ is a rectangle for } k \in \N \txt{ and } E \subset \bigcup_{k = 1}^\infty R_k}
\]
where $R \subset \R^d$ is a rectangle if there are $I_1, \dots, I_d$ intervals such that $R = I_1 \times \dots \times I_d$ and $\Abs{R} = (\sup I_1 - \inf I_1) \cdots (\sup I_d - \inf I_d)$. 

\begin{remark} \label{remark: density theorem changes}
The proof that follows is a completed and corrected version from \cite{leshno1993multilayer}. Concretely, Step 1.1 is added following \cite{lin1993fundamentality} as well as Steps 4.2 and 5.1 following \cite{folland1999real} respectively. Moreover, Step 3 corrects an argument given in \cite{leshno1993multilayer} where hypothesis \textit{a)} is mistaken as a discontinuity in a countable set.    
\end{remark}

\begin{proof}
    Suppose that $\varrho = p$ a.e with $p \in \mathcal{P}^{(r, 1)}$. Then, $V_d \subset \mathcal{P}^{(r, d)}$ which is a finite dimensional space, preventing the density of $V$ in the infinite dimensional space $C(\R^d)$.

    For the converse implication, suppose that $\varrho$ satisfies \ref{enu: almost continuous} and \ref{enu: locally bounded} without being a polynomial. We divide the proof on several steps.

    \underline{Step 1:} It suffice to show that $V_1$ is dense.

    \underline{Step 1.1:} The space
    \[
    W_d \coloneqq \set{ \sum_{k = 1}^n f_k( \inner{ a_k, \; \cdot \; } ) \st n \in \N, a_1, \dots, a_n \in \R^d, f_1, \dots, f_n \in C(\R)}
    \]
    is dense in $C(\R)$.

    To show this, using Weierstrass Theorem, we only need to show that $W_d$ contains every polynomial. Let $r \in \N \cup \set{ 0 }$ and consider the finite-dimensional space
    \[
    H^r_d \coloneqq \Span \set{ x \mapsto x^\mu \st \mu \in (\N \cup \set{ 0 })^d, \Abs{\mu}_1 = r } \subset C(\R).
    \]
    We estate that $H_d^r \subset W_d$. By the multinomial formula, the function
    \[
    p_a(x) \coloneqq \inner{ a, x }^r = (a_1 x_1 + \dots + a_d x_d)^r
    \]
    is in $H^r_d$ and $W_d$ for every $a = (a_1, \dots, a_d) \in \R^d$. Consider the dual space of $H^r_d$, $(H^r_d)^*$, which if $\Span \set{ p_a }_{a \in \R^d} \neq H^r_d$, has a non zero linear function $l$ mapping every $p_a$ to 0. Let us now identify the elements of $(H^r_d)^*$. We know that
    \[
    D^\lambda x^\mu \coloneqq \frac{\partial^r}{\partial x_1^{\lambda_1} \dots \partial x_d^{\lambda_d}} x^\mu = \begin{dcases}
        \mu_1! \dots \mu_d! & \txt{ if } \lambda = \mu \\
        0 & \txt{ otherwise}
    \end{dcases}
    \]
    for all $\lambda = (\lambda_1, \dots, \lambda_d), \mu = (\mu_1, \dots, \mu_d) \in (\N \cup \set{ 0 })^d $ such that $\Abs{\lambda}_1 = \Abs{\mu}_1 = r$. So, for every $l \in (H^r_d)^*$ there exists $q \in H^r_d$ such that $l(p) = q(D) p$ for all $p \in H^r_d$. Let us now achieve the contradiction. Hence,
    \[
    \begin{split}
        l(p_a) = q(D)p_a(x) &= q(D)(a_1 x_1 + \dots + a_d x_d)^r \\
        &= q(D) \sum_{ \substack{\mu \in (\N \cup \set{ 0 })^d \\ \Abs{\mu}_1 = r}} \frac{r!}{\mu_1! \dots \mu_d!} (a_1 x_1)^{\mu_1} \dots (a_d x_d)^{\mu_d} = r! q(a)
    \end{split}
    \]
    so we conclude that $l(p_a) = 0$ for all $a \in \R$ if and only if $q = 0$, or in other words, $l = 0$. It follows that, by Weierstrass' Theorem,
    \[
    C(K) = \overline{ \bigcup_{r = 0}^\infty H_d^r} \subset \overline{ W_d } \subset C(K).
    \]
    for any compact $K \subset \R^d$.

    \underline{Step 1.2:} $V_1$ is dense in $C(\R)$ implies $V_d$ dense in $C(\R^d)$.

    Suppose that $V_1$ is dense in $C(\R)$. Fix $f \in C(\R^d)$, $K \subset \R^d$ a compact and $\varepsilon > 0$. By Step 1.1, there are $g_1, \dots, g_n \in C(\R)$ and $v_1, \dots, v_n \in \R^d$ such that
    \[
    \Norm{f - \sum_{i = 1}^n g_i(\inner{v_i, \; \cdot \;})}_{L^\infty(K)} < \frac{\varepsilon}{2}.
    \]
    Since $\inner{ v_i, \; \cdot \;}$ is continuous function, all the sets $\inner{ v_i, K } \subset \R$ are compact and more specifically, bounded. Assume then that $\inner{v_i , K} \subset [a, b]$ for $i = 1, \dots, n$. By hypothesis, for every $i \in \set{ 1, \dots, n}$, there exists $m_i \in \N$, $a_{i, 1}, \dots, a_{i, m_i}$, $b_{i, 1}, \dots, b_{i, m_1}$, $\lambda_{i, 1}, \dots, \lambda_{i, m_i} \in \R$ such that
    \[
    \Norm{
    g_i - \sum_{j = 1}^{m_i} \lambda_{i, j} \varrho( a_{i, j} \; \cdot \; + b_{i, j} )
    }_{L^\infty([a, b])} < \frac{\varepsilon}{2 n}.
    \]
    We conclude observing that 
    \[
    \sum_{i = 1}^n \sum_{j = 1}^{m_i} \lambda_{i, j} \varrho(\inner{ a_{i, j} a_i, \; \cdot \;} + b_{i, j}) \in V_d
    \]
    and
    \begin{multline*}
        \Norm{f - \sum_{i = 1}^n \sum_{j = 1}^{m_i} \lambda_{i, j} \varrho(\inner{ a_{i, j} a_i, \; \cdot \;} + b_{i, j})}_{L^\infty(K)} \leq \Norm{f - \sum_{i = 1}^n g_i( \inner{ v_i, \; \cdot \; } ) }_{L^\infty(K)} \\ + \sum_{i = 1}^n \Norm{ g_i( \inner{ v_i, \; \cdot \;} ) - \sum_{j = 1}^{m_i} \lambda_{i, j} \varrho(\inner{ a_{i, j} a_i, \; \cdot \;} + b_{i, j}) }_{L^\infty(K)} < \frac{\varepsilon}{2} + n \frac{\varepsilon}{2 n} = \varepsilon.
    \end{multline*}

    \underline{Step 2:} If there exists $\varphi \in \overline{V}_1 \cap C^\infty(\R)$ that is not a polynomial then $V_1$ is dense.

    It is not hard to see that $\varphi(a \; \cdot \; + b) \in \overline{V}_1$ for all $a, b \in \R$. Since
    \[
    \overline{V}_1 \ni \frac{\varphi((a + h) \; \cdot \; + b) - \varphi(a \; \cdot \; + b)}{h} \xrightarrow[h \rightarrow 0]{}\frac{\partial}{\partial a}\varphi(a \; \cdot \; + b) \qquad \txt{ in } C(\R),
    \]
    we conclude that $\frac{\partial}{\partial a} \varphi(a \; \cdot \; + b) \in \overline{V}_1$. This argument works with $\frac{\partial}{\partial a} \varphi(a \; \cdot \; + b) \in C^\infty(\R)$ and so $\frac{\partial^k}{\partial a^k} \varphi(a \; \cdot \; + b) \in \overline{V}_1$. This tells us
    \begin{equation}
    \func{\frac{\partial^k}{\partial a^k} \varphi(a \; \cdot \; + b)}{\R}{\R}{x}{\varphi^{(k)}(a x + b) x^k} \in \overline{V}_1 \label{eq: partials in V1}
    \end{equation}
    for all $a, b \in \R$. Since $\varphi$ is not a polynomial, for every $k \in \N$, there is $b_k \in \R$ such that $\varphi^{(k)}(b_k) \neq 0$. We conclude choosing in \eqref{eq: partials in V1} $a = 0$ and $b = b_k$ that the polynomials are in $\overline{V}_1$ and therefore, $\overline{V}_1$ is dense in $C(\R)$.

    \underline{Step 3:} For every $\varphi \in C^\infty_c(\R)$, the convolution $\varphi * \varrho$ is in $\overline{V}_1$.

    Let $\varphi \in C^\infty(\R)$ be a function such that $\supp \varphi \subset [-R, R]$ for some $R > 0$ and such that $\varphi \neq 0$. Fix $\varepsilon > 0 $ and $r > 0$. Let $U \subset \R$ be the open set from the statement of this theorem and we call $M = \Norm{\varrho}_{L^\infty([-R-r, R+r])} \Norm{\varphi}_{L^\infty(\R)}$. Since $K \coloneqq [-R-r, R+r] \setminus U$ is a compact with Lebesgue measure 0, by definition of Lebesgue measure, there are $I_1, \dots, I_{n} \subset \R$ open intervals pairwise disjoint such that
    \begin{equation}\label{eq: discontinuity measure control}
        K \subset \bigcup_{I = 1}^{n} I_i \eqqcolon I \qquad \txt{ and } \qquad \Abs{I} = \sum_{i = 1}^n \Abs{I_i} < \frac{\varepsilon / 3}{2M}.
    \end{equation}
    We know that $J \coloneqq [-R-r, R+r] \setminus I$ is a compact and $\varrho$ is continuous in $J$, therefore, there exists $\delta > 0$ such that
    \begin{equation} \label{eq: uniform continuity}
    \Abs{\varrho(y) - \varrho(z)} < \frac{\varepsilon / 3}{\Norm{\varphi}_{L^1(\R)}}
    \end{equation}
    for every $y, z \in J$ such that $\Abs{y - z} < \delta$.
    We can find $J_1, \dots, J_m \subset [-R - r, R + r]$ pairwise disjoint intervals (not necessarily open nor closed) such that
    \begin{equation} \label{eq: uniform continuity in mini intervals}
        [-R - r, R + r] = \bigcup_{j = 1}^m J_j \qquad \txt{ and } \qquad \Abs{J_j} < \min \left \{ \frac{\varepsilon / 3}{4 n M}, \delta \right \}
    \end{equation}
    We observe that, for any $x \in \R$, 
    \begin{equation} \label{eq: discontinuity intersects}
    \Abs{ \bigcup_{x - J_j \cap I \neq \emptyset} x - J_j } \leq \sum_{i = 1}^n \left (2 \max_{j = 1, \dots, n} \Abs{J_j} + \Abs{I_i} \right ) < \frac{\varepsilon / 3}{2 M} + \Abs{I}
    \end{equation}
    which follows from the fact that $x - J_j$ and $I_i$ are intervals for $i = 1, \dots, n$ and $j = 1, \dots, m$. We conclude with the fact that
    \[
    \sum_{j = 1}^m \varrho(x - y_j) \int_{J_j} \varphi(y)dy \in V_1
    \]
    for fixed $y_j \in J_j$ for $j = 1, \dots, m$ and that, for any $x \in [-r, r]$,
    \begin{align*}
        \Abs{\varrho * \varphi(x) - \sum_{j = 1}^m \varrho(x - y_j) \int_{J_j} \varphi(y)dy} =& \left \lvert \sum_{j = 1}^m \int_{J_j} (\varrho(x - y) - \varrho(x - y_j) ) \varphi(y)dy \right \rvert \\
        \leq & \sum_{x - J_j \cap I \neq \emptyset} \int_{J_j} \Abs{\varrho(x - y) - \varrho(x - y_j)} \Abs{\varphi(y)}dy \\
        & + \sum_{x - J_j \cap I = \emptyset} \int_{J_j} \Abs{\varrho(x - y) - \varrho(x - y_j)} \Abs{\varphi(y)}dy \\
        \leq & 2M \Abs{\bigcup_{x - J_j \cap I \neq \emptyset} x - J_j} + \frac{\varepsilon / 3}{\Norm{\varphi}_{L^1(\R)}}\Norm{\varphi}_{L^1(\R)} \\
        \leq & 2M \left ( \frac{\varepsilon / 3}{2M} + \Abs{I} \right ) + \frac{\varepsilon}{3} \\
        = & \frac{2\varepsilon}{3} + 2M \Abs{I} < \varepsilon
    \end{align*}
    where we used in the second inequality the uniform continuity in $x - J_j$ \eqref{eq: uniform continuity} since $\Abs{J_j} < \delta$ and in the third and forth ones, the bounds on the measure of the different sets, given by \eqref{eq: discontinuity intersects} and \eqref{eq: discontinuity measure control} respectively.

    \underline{Step 4:} If for all $\varphi \in C^\infty_c(\R)$ $\varrho * \varphi$ is a polynomial, then $\varrho$ is a polynomial a.e.

    \underline{Step 4.1:} If for all $\varphi \in C_c^\infty(\R)$ $\varrho * \varphi$ is a polynomial, then their degree are bounded.

    For a compact interval $I$, consider the space $C^\infty(I)$ with the metric
    \[
    d(\varphi, \psi) \coloneqq \sum_{k = 0}^\infty 2^{-k} \min \left \{  \Norm{\varphi^{(k)} - \psi^{(k)}}_{L^\infty(I)}, 1 \right \}.
    \]
    This space $(C^\infty(I), d)$ is a complete metric space. The fact that $d$ satisfies the triangle inequality follows from the subadditive of $\min \set{ \Normm_{L^\infty(I)}, 1 }$ inherited of $\Normm_{L^\infty(I)}$ and the completeness of the space from the completeness of $(C(I), \Normm_{L^\infty(I)})$.

    We define the subspaces
    \[
    W_k \coloneqq \set{ \varphi \subset C^\infty(I) \st \deg ( \varphi * \varrho ) \leq k }
    \]
    for $k \in \N \cup \set{ 0 }$. By hypothesis, we know that
    \[
    \bigcup_{k = 0}^\infty W_k = C^\infty_c (I)
    \]
    The Baire's Category Theorem (see Theorem 5.3 from \cite{folland1999real}) states that a complete metric space is not a union of sets whose closure has empty interior. Therefore, it exists $r \in \N \cup \set{ 0 }$ such that $W_r$ has a non empty interior in the topology induced by $d$. Being $W_r$ a vector space, this means that $W_r = C^\infty(I)$.

    By translation, the same bound on the degree works for any other compact interval $J$ with same or less length. For a compact interval $J$ with more length that $I$, we choose a convenient partition of unity, that is, we can find compact intervals $I_1, \dots, I_n$ with same length as $I$ and $\varphi_k \in C^\infty(I_k)$ for $k = 1, \dots, n$ such that
    \begin{align*}
        I \subset \bigcup_{k = 1}^n I_k && \txt{ and } && 1 = \sum_{k = 1}^n \varphi_k.
    \end{align*}
    Therefore, for any $\varphi \in C^\infty(J)$,
    \[
    \deg(\varphi * \varrho) = \deg \left ( \sum_{k = 1}^n (\varphi \varphi_k) * \varrho \right ) \leq \max_{k = 1, \dots, n} \deg \left ( (\varphi \varphi_k) * \varrho \right ) \leq r
    \]
    since $\varphi \varphi_k \in C^\infty(I_k)$ and $\Abs{I_k} = \Abs{I}$.

    \underline{Step 4.2:} For all $\varphi \in C^\infty_c(\R)$ and $f: \R \longrightarrow \R$ a locally bounded function continuous on an open set $U$, then $t^{-1} \varphi(t^{-1} \; \cdot \;) * f$ converges uniformly in compacts contained in $U$ to $f \cdot \int \varphi$ as $t \rightarrow 0$.

    Let $\varphi \in C^\infty_c(E)$ with $E$ a compact and $K \subset U$ a compact. For a fixed $r < \dist(K, \partial U)$, we define the compact set
    \[
    F \coloneqq \set{ x + y \st x \in K \txt{ and } \Abs{y} \leq r } \subset U.
    \]
    Since $f$ is uniformly continuous in $F$, given $\varepsilon > 0$, there exists $\delta > 0$ such that
    \[
    \Abs{f(x) - f(y)} < \varepsilon / \Norm{\varphi}_1
    \]
    if $x, y \in F$ and $\Abs{x - y} < \delta$. Since $E$ is bounded, for $t > 0$ small enough,
    \[
    \set{ x - tz \st x \in K \txt{ and } z \in E} \subset F.
    \]
    If $\varphi_t = t^{-1} \varphi( t^{-1} \; \cdot \; )$, then, for all $x \in K$ and $t \in (0, \delta)$ small enough,
    \begin{align*}
        \Abs{ \varphi_t * f(x) - f(x) \int \varphi(y) dy } &= \Abs{ \int f(x - y) \varphi_t(y) dy - f(x) \int \varphi_t(y) dy } \\
        &= \Abs{ \int_E ( f(x - t z) - f(x) ) \varphi(z) dz } \\
        &\leq \frac{\varepsilon}{\Norm{\varphi}_1} \int_E \Abs{\varphi(z)} dz = \varepsilon.
    \end{align*}
    
    \underline{Step 4.3:}  If for all $\varphi \in C^\infty_c(\R)$ $\varrho * \varphi$ is a polynomial, then $\varrho$ is a polynomial a.e.

    Let $\varphi \in C^\infty_c(\R)$ be a function with integral 1, $\varphi_t = t^{-1} \varphi(t^{-1} \; \cdot \;)$ for $t > 0$ and $r \in \N \cup \set{0}$ the bound over the degree of the convolutions with $\varrho$, as found in Step 4.1. We know that, by hypothesis,
    \[
    \varrho * \varphi_t(x) = \sum_{k = 0}^r a_k(t) x^k.
    \]
    By Step 4.2, it convergences to $\varrho$ in a compact $K \subset U$ and so, $\varrho * \varphi_t(x) \longrightarrow \varrho(x)$ for all $x \in K$ when $t \rightarrow 0$.
    
    The space $(\mathcal{P}^{(r + 1, 1)}, \Normm_{L^\infty(K)})$ is a finite dimensional space and so, the norm $\Abs{\; \cdot \;}_\infty$ defined for ever polynomial of degree $r$ or less $P(x) = \sum\limits_{k = 0}^r a_k x^k$ as
    \[
    \Abs{P}_\infty = \max_{k = 0, \dots, r} \Abs{a_k}
    \]
    is equivalent to $\Norm{\; \cdot \;}_{L^\infty(K)}$. Therefore, the sequence $\set{(a_0(t), \dots, a_r(t))}_{t \in (0, 1)}$ has a convergent subsequence to some vector $(\alpha_0, \dots, \alpha_r) \in \R^{r + 1}$. By Step 4.2 and the uniqueness of the convergence, we know that, for all $x \in U$
    \[
    \varrho * \varphi_t(x) \xrightarrow[t \rightarrow 0]{} \varrho(x) = \sum_{k = 0}^r \alpha_k x^k.
    \]
    So, since $\R \setminus U$ has Lebesgue measure 0, $\varrho$ is polynomial almost everywhere.

    \underline{Step 5:} $V_d$ is dense in $C(\R)$.
    
    To prove that $V_d$ is dense, by Step 1 and Step 2, we only need to show that $\overline{V}_1 \cap C^\infty(\R)$ contains at least a non polynomial function. Since $\varrho$ is not a polynomial a.e., by Step 4, there exists $\varphi \in C^\infty_c(\R)$ such that $\varrho * \varphi$ is not a polynomial. By the fact that $\varrho \in L^1_{\loc}(\R)$ and $\varphi \in C^\infty_c(\R)$, we get that $\varrho * \varphi \in C^\infty(\R)$ (a minor modification of Proposition 8.10 from \cite{folland1999real}). By Step 3, we conclude that $\varrho * \varphi \in \overline{V}_1 \cap C^\infty(\R)$.
\end{proof}

\begin{corollary} \label{cor: NN satisfies P5}
    Let $\varrho: \R \longrightarrow \R$ be a function satisfying all the hypotheses of Theorem \ref{th: Density of NN}, $p \in (0, \infty)$, $k \in \N$ and $\Omega \subset \R^d$ a bounded set. Then, the set
    \[
    \set{R(\phi) \in L^p(\Omega) \st \phi \txt{ is a strict } \varrho-\txt{NN}}
    \]
    is dense in $L^p(\Omega)^k$.
\end{corollary}
\begin{remark}
    This corollary implies that
    \[
    \bigcup_{n \in \N} \Mtt_n(L^p(\Omega)^k, \varrho)
    \]
    is dense and so $(\Mtt_n(L^p(\Omega)^k, \varrho))_{n = 0}^\infty$ satisfies Property \ref{enu: Sigmainfty density}.
\end{remark}
\begin{proof}
    We can suppose that $k = 1$. Fix $f \in L^p(\Omega)$ and $\varepsilon > 0$. It is well known that $C^\infty_c(\R^d)$ is dense in $L^p(\R^d)$ (the proof of Proposition 8.17 in \cite{folland1999real} works as well for $p \in (0, 1)$). Let $\varphi \in C^\infty_c(\R^d)$ such that
    \[
    \Norm{f - \varphi}_{L^p(\R^d)} < \frac{\varepsilon}{2 C_p}
    \]
    and $\supp \varphi, \Omega \subset K$ where $K \subset \R^d$ is compact, $f = 0$ in $\R^d \setminus \Omega$ and $C_p = 2^{1/p}$ if $p \in (0, 1)$ and $C_p = 1$ if $p \geq 1$. Since $\varrho$ satisfies all conditions of Theorem \ref{th: Density of NN}, we know that there exists a $\varrho$-realization $g: \R^d \longrightarrow \R$ such that
    \[
    \Norm{\varphi - g}_{L^\infty(K)} < \frac{\varepsilon}{2 C_p \Abs{\Omega}}
    \]
    Then
    \[
    \Norm{f - g}_{L^p(\Omega)} \leq C_p \left ( \Norm{f - \varphi}_{L^p(\Omega)} + \Norm{\varphi - g}_{L^p(\Omega)} \right ) < C_p \left ( \frac{\varepsilon}{2C_p} + \frac{\varepsilon}{2 C_p} \right ) < \varepsilon.
    \]
\end{proof}

\subsection{Besov Space embedded in the Neural Network Approximation Space.}

We are now going to see that the Besov Space introduce in Definition \ref{def: Besov} is embedded into the approximation space of NNs.

\begin{lemma} \label{lemma: characteristic function almost varrho NN}
    Let $r, d \in \N$ and $p \in (0, \infty)$. Then there exists $c \in \N$ such that, for every $\varepsilon > 0$, there is $\varphi_\varepsilon \in \Mtt_c(L^p(\R^d), \varrho_r)$ such that
    \[
    \Norm{\chi_{[0, 1]^d} - \varphi_\varepsilon}_{L^p(\R^d)} < \varepsilon.
    \]
\end{lemma}
Note that $c$ is independent of $\varepsilon > 0$.
\begin{proof}
    Fix some $\varepsilon > 0$. To find such $\varphi_\varepsilon$, we construct various intermediate functions. Consider the function
    \[
    \func{\sigma}{\R}{\R}{x}{\frac{1}{r!} \sum_{k = 0}^r \binom{r}{k} (-1)^k \varrho_r(x - k)}
    \]
    which satisfies that $\sigma' = \beta^{(r - 1)}$ and $\sigma(x) = 0$ if $x \leq 0$, $\sigma(x) = 1$ for $x \geq r$. By \eqref{eq: B spline is a convolution}, $\sigma' \geq 0$ and so $\sigma$ increases, meaning that $0 \leq \sigma \leq 1$. It follows then that, for any $\delta \in (0, 1/2)$, the function
    \[
    \func{\psi_\delta}{\R}{\R}{x}{\sigma \left( r \frac{x}{\delta} \right ) - \sigma \left ( r \frac{x + \delta - 1}{\delta} \right )}
    \]
    is a $\varrho_r$-realization satisfing $0 \leq \psi_\delta \leq 1$, $\psi_\delta(x) = 0$ if $x \not \in [0, 1]$ and $\psi_\delta(x) = 1$ if $x \in [\delta, 1 - \delta]$. If $d = 1$, we pick $\varphi_{\varepsilon} = \psi_{\varepsilon / 2}$. For $d \geq 2$, the function
    \[
    \phi_\delta(x) = \sigma \left ( r \left ( \sum_{k = 1}^d \psi_\delta(x_k) - d + 1 \right ) \right )
    \]
    is almost the target $\varrho_r$-realization. If $x \not \in [0, 1]^d$, then $x_k \not \in [0, 1]$ for some $k \in \set{1, \dots, d}$ and so
    \[
    \sum_{k = 1}^d \psi_\delta(x_k) \leq d - 1
    \]
    implying that $\phi_\delta(x) = 0$. For $x \in [\delta, 1 - \delta]^d$, then
    \[
    \sum_{k = 1}^d \psi_\delta(x) = d
    \]
    and so $\phi_\delta(x) = 1$. Since $\Abs{[0, 1]^d \setminus [\delta, 1 - \delta]^d} \longrightarrow 0$ as $\delta \rightarrow 0^+$, we know that for every $\varepsilon > 0$ there exists $\delta(\varepsilon) > 0$ such that
    \[
    \Norm{\chi_{[0, 1]^d} - \phi_{\delta(\varepsilon)}}_{L^p(\R^d)} < \varepsilon
    \]
    and so $\varphi_\varepsilon = \phi_{\delta(\varepsilon)}$. The fact that $\varphi_\varepsilon$ is a $\varrho_r$-realization is obvious and that the value of $\varepsilon > 0$ does not affect on the bound of the number of weights too. So, there is $c \in \N$ such that $\varphi_\varepsilon \in \Mtt_c(L^p(\R^d), \varrho_r)$ for every $\varepsilon > 0$.
\end{proof}

This lemma helps us proving an intermediate embedding.

\begin{lemma} \label{lemma: intermediate approximation space}
    Let $r \in \N$, $t \in \N \cup \set{0}$, $\alpha, p \in (0, \infty)$, $q \in (0, \infty]$, $\Omega \subset \R^d$ a bounded set and
    \[
    B_n^t \coloneqq \set{ \sum_{k = 1}^n \lambda_k \beta^{(t)}_d(a_k \; \cdot \; + b_k) \st \lambda_k \in \R, a_k > 0, b_k \in \R^d \txt{ for } k = 1, \dots, n }
    \]
    where $\beta_d^{(t)}(x) = \beta^{(t)}(x_1) \dots \beta^{(t)}(x_d)$ being $\beta^{(t)}$ a B-spline introduce in Definition \ref{eq: def B spline}. Then,
    \[
    A^\alpha_q(L^p(\Omega), (B_n^t)_{n = 0}^\infty) \hookrightarrow M_{p,q}^\alpha(\Omega, \varrho_r)
    \]
    if $t = 0$ or $t = r \geq \min \set{d, 2}$.
\end{lemma}

\begin{proof}
    \underline{Step 1:} Both approximation spaces are quasi-normed spaces.

    This is obvious for $M^\alpha_{p, q}(\Omega, \varrho_r)$ by Proposition \ref{prop: NN satisfies P1 to P4}, Corollary \ref{cor: NN satisfies P5} and Theorem \ref{prop: NN satisfies P1 to P4}. If we call $B^t = (B^t_n)_{n = 0}^\infty$, $B^t$ satisfies trivially \ref{enu: Sigma0 0}--\ref{enu: Additivity closeness} for all $t \in \N \cup \set{0}$. To prove that $B^t$ satisfies \ref{enu: Sigmainfty density} for all $t \in \N \cup \set{0}$ too, we use the density of 
    \[
    \Span \set{\chi_E \st E \subset \R^d, \Abs{E} < \infty}
    \]
    in $L^p(\R^d)$ as it has been proved in Proposition 6.7 \cite{folland1999real} with arguments that also work for $p \in (0, 1)$. Let us prove this result.
    
    \underline{Step 1.1:} $\chi_R \in \overline{\bigcup\limits_{n = 0}^\infty B_n^t}$ when $R$ is a rectangle.
    
    Let $R \subset \R^d$ be a bounded rectangle, that is, there are $I_1, \dots, I_d$ bounded intervals such that $R = I_1 \times \dots \times I_d$. We call for every set $E \subset \R^d$
    \[
    \Lambda_k(E) \coloneqq \set{ j \in \Z^d \st 2^{-k}([0, 1]^d + j) \cup E \neq \emptyset }.
    \]
    It is then trivial, being $R$ a rectangle, that
    \[
    \Abs{\bigcup_{j \in \Lambda_k(R)} 2^{-k}([0, 1]^d + j)} \leq \prod_{k = 1}^d (\Abs{I_k} + 2 \Abs{2^{-k}[0, 1]^d}) = \prod_{k = 1}^d (\Abs{I_k} + 2^{-k + 1}).
    \]
    And so, using \eqref{eq: B splines support} and \eqref{eq: B spline partition of unity}, for every $\varepsilon > 0$ there is a sufficiently big $k \in \N$,
    \[
    \Norm{\chi_R - \sum_{j \in \Lambda_k(R)} \beta_{k, j}^{(r, d)}}_{L^p(\R^d)} \leq \left ( \prod_{k = 1}^d (\Abs{I_k} + 2^{-k + 1}) - \Abs{R} \right )^\frac{1}{p} < \varepsilon.
    \]

    \underline{Step 1.2:} $\chi_E \in \overline{\bigcup\limits_{n = 0}^\infty B_n^t}$ when $\Abs{E} < \infty$.
    
    Let $E \subset \R^d$ be a finite measure set and $\varepsilon > 0$. By definition, there are $R_1, \dots, R_n$ bounded rectangles such that
    \begin{align*}
        \Abs{E} \leq \sum_{i = 1}^n \Abs{R_k} + \frac{\varepsilon}{2C_p}, && E \subset \bigcup_{i = 1}^n R_k
    \end{align*}
    where $C_p = 2^\frac{1}{p}$ if $p \in (0, 1)$ and $C_p = 1$ if $p \geq 1$.
    We pick $k \in \N$ big enough so
    \[
    \Norm{\chi_{R_i} - \sum_{j \in \Lambda_k(R_i)} \beta_{k, j}^{(r, d)}}_{L^p(\R^d)} < \frac{\varepsilon}{2 n C_p^{n + 1}}
    \]
    for $i = 1, \dots, n$. Then,
    \[
    \begin{split}
        \Norm{\chi_E - \sum_{i = 1}^n \sum_{j \in \Lambda_k(R_i)} \beta^{(r, d)}_{k, j}}_{L^p(\R^d)} &\leq C_p \left ( \Norm{\chi_E - \sum_{i = 1}^n \chi_{R_i}}_{L^p(\R^d)} \right .\\ 
        &\left . + \Norm{\sum_{i = 1}^n \left ( \chi_{R_i} - \sum_{j \in \Lambda_k(R_i)} \beta_{k, j}^{(r, d)} \right )}_{L^p(\R^d)} \right ) \\
        &< \frac{\varepsilon}{2} + C_p^{n + 1} \sum_{i = 1}^n \Norm{\chi_{R_i} - \sum_{j \in \Lambda_k(R_i)} \beta_{k, j}^{(r, d)}}_{L^p(\R^d)} = \varepsilon.
    \end{split}
    \]
    Thus, $B^t$ satisfies \ref{enu: Sigmainfty density}.

    \underline{Step 2:} $\Normm_{M^\alpha_{p, q}(\Omega, \varrho_r)} \sleq \Normm_{A^\alpha_q(L^p(\Omega), B^t)}$.
    
    \underline{Step 2.1:} $\Normm_{M^\alpha_{p, q}(\Omega, \varrho_r)} \sleq \Normm_{A^\alpha_q(L^p(\Omega), B^0)}$.

    First, let us see that
    \[
    B_n^0 \subset \overline{\Mtt_{n c}(L^p(\Omega), \varrho_r)}.
    \]
    Let us consider
    \[
    \sum_{k = 1}^n \lambda_k \beta^{(0)}(a_k \; \cdot \; + \gamma_k) \in B^0_n.
    \]
    By Lemma \ref{lemma: characteristic function almost varrho NN}, we know that there is $c \in \N$ such that, for every $\varepsilon > 0$, it exists a $\varrho_r$-realization $\varphi_{\varepsilon} \in \Mtt_c(L^p(\Omega), \varrho_r)$ satisfying
    \[
    \Norm{\chi_{[0, 1]^d} - \varphi_{\varepsilon}}_{L^p(\Omega)} < \varepsilon.
    \]
    
    Thus, for this $\varrho_r$-realization $\varphi_\varepsilon$,
    \begin{equation} \label{eq: B0 element approximation by realizations}
    \Norm{\sum_{k = 1}^n \lambda_k \beta^{(0)}(a_k \; \cdot \; + \gamma_k) - \sum_{k = 1}^n \lambda_k \varphi_\varepsilon(a_k \; \cdot \; + \gamma_k) }_{L^p(\Omega)} < n C_p^n \varepsilon \Abs{\lambda}_1
    \end{equation}
    where $C_p = 2^\frac{1}{p}$ if $p < 1$ and $C_p = 1$ otherwise and $\lambda = (\lambda_1, \dots, \lambda_n) \in \R^n$. Suppose that $\phi = ((T_1, \alpha_1), \dots, (T_L, \alpha_L))$ is a $\varrho_r$-NN such that $\varphi_\varepsilon = R(\phi)$ and $M(\phi) \leq c$ with $T_\ell:\R^{n_{\ell - 1}} \longrightarrow \R^{n_\ell}$ for $\ell = 1, \dots, L$ and
    \[
    W(\phi) = ((A_1, b_1), \dots, (A_L, b_L)).
    \]
    Then, the $\varrho_r$-NN $\Phi_n = \left ( S_1, \beta_1, \dots, S_L, \Id \right )$ with
    \begin{align*}
        \func{S_1}{\R^d}{(\R^{n_\ell})^n}{x}{(T_1(a_1 x + \gamma_1), \dots, T_1(a_n x + \gamma_n))} && \\
        \func{\beta_\ell}{(\R^{n_\ell})^n}{(\R^{n_\ell})^n}{(x_1, \dots, x_n)}{(\alpha_\ell(x_1), \dots, \alpha_\ell(x_n))} && \txt{for } \ell = 1,\dots, L,\\
        \func{S_\ell}{(\R^{n_{\ell - 1}})^n}{(\R^{n_\ell})^n}{(x_1, \dots, x_n)}{(T_\ell(x_1), \dots, T_\ell(x_n))} && \txt{for } \ell = 2, \dots, L - 1,\\
        \func{S_L}{(\R^{n_{L - 1}})^n}{\R}{(x_1, \dots, x_n)}{\lambda_1 T_L(x_1) + \dots + \lambda_n T_L(x_n)} && 
    \end{align*}
    satisfies that
    \begin{multline*}
    W(\Phi_n) = \left ( \left ( \begin{bmatrix} a_1 A_1 \\ \vdots \\ a_n A_1 \end{bmatrix}, \begin{bmatrix} A_1 \gamma_1 + b_1 \\ \vdots \\ A_1 \gamma_n + b_1 \end{bmatrix} \right ), \left ( \begin{bmatrix} A_2 & & \\ & \ddots & \\ & & A_2 \end{bmatrix}, \begin{bmatrix} b_2 \\ \vdots \\ b_2 \end{bmatrix} \right ), \right . \\ \left . \dots, \left ( \begin{bmatrix} A_{L - 1} & & \\ & \ddots & \\ & & A_{L - 1} \end{bmatrix}, \begin{bmatrix} b_{L - 1} \\ \vdots \\ b_{L - 1} \end{bmatrix} \right ), \left ( \begin{bmatrix} \lambda_1 A_L & \dots & \lambda_n A_L \end{bmatrix}, b_L(\lambda_1 + \dots + \lambda_n) \right ) \right )
    \end{multline*}
    and $R(\Phi_n) = \sum_{k = 1}^n \lambda_k \varphi_\varepsilon(a_k \; \cdot \; + \gamma_k)$.
    Therefore, we get that
    \[
    M(\Phi_n) \leq n M(\phi) \leq n c,
    \]
    implying that $\Phi_n \in \Mtt_{n c}(L^p(\Omega), \varrho_r)$ and
    \[
    B_n^0 \subset \overline{\Mtt_{n c}(L^p(\Omega), \varrho_r)}.
    \]
    Consequently, for every $f \in L^p(\Omega)$ and calling $\Sigma_n = \Mtt_{n}(L^p(\Omega), \varrho_r)$ for every $n \in \N \cup \set{0}$, it follows that
    \begin{equation*}
    E(f, \Sigma_{c n})_p = \inf \set{ \Norm{f - g}_{L^p(\Omega)} \st g \in \Sigma_{c n}} \leq E(f, B_n^0)_p.
    \end{equation*}
    We can then conclude that
    \[
    \begin{split}
        \sum_{n = 1}^\infty (n^\alpha E(f, \Sigma_{n - 1})_p)^q \frac{1}{n} &= \sum_{n = 0}^\infty \sum_{k = c n + 1}^{c (n + 1)} (k^\alpha E(f, \Sigma_{k - 1})_p)^q \frac{1}{k} \\
        &\leq \sum_{n = 0}^\infty \sum_{k = c n + 1}^{c(n + 1)} (k^\alpha E(f, \Sigma_{cn})_p)^q \frac{1}{k} \\
        &\leq \sum_{n = 0}^\infty \sum_{k = c n + 1}^{c(n + 1)} ((c(n + 1))^\alpha E(f, \Sigma_{c n})_p)^q \frac{1}{n + 1} \\
        &\leq c^{\alpha q} \sum_{n = 0}^\infty ((n + 1)^\alpha E(f, \Sigma_{c n})_p)^q \frac{1}{n + 1} \sum_{k = c n + 1}^{c (n + 1)} 1 \\
        &\leq c^{\alpha q + 1} \sum_{n = 0}^\infty ((n + 1)^\alpha E(f, B^0_n)_p)^q \frac{1}{n + 1}.
    \end{split}
    \]

    \underline{Step 2.2:} $\Normm_{M^\alpha_{p, q}(\Omega, \varrho_r)} \sleq \Normm_{A^\alpha_q(L^p(\Omega), B^r)}$ if $d = 1$.
    
    As it is clear from Step 2.1 that we only need to see that there is a constant $c \in \N$ such that, for every $n \in \N$,
    \[
    B^r_n \subset \overline{\Mtt_{c n}(L^p(\Omega), \varrho_r)}.
    \]
    We introduce for every $a > 0$ and $b \in \R$ the strict $\varrho_r$-NN $ \lambda \phi_{a, b}$ with weights
    \[
    W(\phi) = \left ( \left ( \begin{bmatrix} a \\ a \\ \vdots \\ a \end{bmatrix}, \begin{bmatrix} b \\ b - 1 \\ \vdots \\ b - (r + 1) \end{bmatrix} \right ), \left ( \frac{\lambda}{r!} \left [\binom{r + 1}{0}, \binom{r + 1}{1}, \dots, \binom{r + 1}{r + 1} \right ], 0 \right ) \right )
    \]
    which satisfies $\lambda \beta^{(r)}(a \; \cdot \; + b) = R(\lambda \phi_{a, b})$ by Definition \ref{eq: def B spline} and $M(\lambda \phi_{a, b}) \leq 3(r + 2)$. And so, repeating the same idea as in the case $t = 0$,
    \[
    B_n^r \subset \Mtt_{3(r + 2) n}(L^p(\Omega), \varrho_r).
    \]

    \underline{Step 2.3:} $\Normm_{M^\alpha_{p, q}(\Omega, \varrho_r)} \sleq \Normm_{A^\alpha_q(L^p(\Omega), B^r)}$ if $r, d \geq 2$.
    
    Finally, this case $r, d \geq 2$ can be summed up by the fact that the exact product is a $\varrho_r$-realization. Indeed, all polynomials of degree equal or less than $r$ are $\varrho_r$-realization. For all $a \in \R$, we can see that
    \[
    (x - a)^r = \varrho_r(x - a) + (-1)^r \varrho_r(a - x)
    \]
    is a $\varrho_r$-realization and $\Span \set{\R \ni x \mapsto (x - a)^r \st a \in \R}$ contains every polynomial of degree $r$ or less. Repeating the same ideas as in Corollary \ref{scalar mult NN def and bound}, using that
    \[
    xy = \frac{1}{4} \left ( (x + y)^2 - (x - y)^2 \right ),
    \]
    we conclude that there is a $\varrho_r$-NN $\phi = ((T_1, \alpha_1), \dots, (T_L, \alpha_L))$ such that $R(\phi)(x, y) = xy$ for all $x, y \in \R$. Recovering notion of equation \eqref{eq: tensor product}, we introduce for $k \in \N$ the $\varrho_r$-NN
    \[
    \Phi_k = ((T_1 \otimes \Id_k, \alpha_1 \otimes \Id_k), \dots, (T_L \otimes \Id_k, \alpha_L \otimes \Id_k))
    \]
    and using Definition \ref{concatenation def}
    \[
    \Phi = \phi \bullet \Phi_{1} \bullet \Phi_{2} \bullet \dots \bullet \Phi_{d - 1}.
    \]
    we get a $\varrho_r$-NN such that $R(\Phi)(x) = x_1 \cdots x_d$ for all $x = (x_1, \dots, x_d) \in \R^d$. In conclusion, $\beta^{(t)}_d$ is a $\varrho_r$-NN. Consequently,
    \[
    B_n^r \subset \Mtt_{c n}(L^p(\Omega), \varrho_r)
    \]
    for some $c \in \N$ and all $n \in \N$, and the proof concludes as in Step 2.1.
\end{proof}

\begin{deff}
    A set $\Omega \subset \R^d$ is a bounded Lipschitz domain if $\Omega$ bounded and for all $x \in \partial \Omega$ there is an open neighborhood $x \in U_x \subset \R^d$, $\varphi: \R^{d - 1} \longrightarrow \R$ a Lipschitz function, an open $V \subset \R^d$ and $T: \R^d \longrightarrow \R^d$ an affine transformation such that
    \[
    T(\Omega \cap U_x) = \set{ (x, x') \in \R^{d - 1} \times \R \st x' < \varphi(x) } \cap V.
    \]
\end{deff}

We recall that a function $\varphi: \R^d \longrightarrow \R$ is Lipschitz if there is a constant $M > 0$ such that for all $x, y \in \R^d$,
\[
\Abs{\varphi(x) - \varphi(y)}_2 \leq M \Abs{x - y}_2.
\]

\begin{theorem} \label{th: Bsov embedded in NN aproximation space}
        Let $\Omega \subset \R^d$ be a bounded Lipschitz domain, $r \in \N$, $p \in (0, \infty)$, $q \in (0, \infty]$,
        \begin{align*}
            r_0 = \begin{cases} 0 & \txt{if } r = 1 \txt{ and } d > 1 \\ r & \txt{otherwise} \end{cases} && \txt{ and } && \lambda = \frac{r_0 + \min \set{1, \frac{1}{p}}}{d}.
        \end{align*}
        Then, for all $\alpha \in (0, \lambda)$,
        \[
        B^{\alpha d}_{p, q}(\Omega) \hookrightarrow M^\alpha_{p, q}(\Omega, \varrho_r).
        \]
\end{theorem}
The proof that follows takes the ideas of Theorem 5.5 in \cite{gribonval2022approximation} but with some modifications for Step 2.
\begin{proof}
    We recall notation from Proposition \ref{prop: Besov norm equivalence on the unit cube} (where in this case $\Omega$ is an arbitrary bounded Lipschitz domain) and Definition \ref{eq: def B spline}
    \[
    \Sigma^{(r, d)}_k(\Omega) = \Span \set{ \beta_{k, j}^{(r, d)}|_\Omega}
    \]
    and from Lemma \ref{lemma: intermediate approximation space},
    \[
    B_n^t \coloneqq \set{ \sum_{k = 1}^n \lambda_k \beta^{(t)}_d(a_k \; \cdot \; + b_k) \st \lambda_k \in \R, a_k > 0, b_k \in \R^d \txt{ for } k = 1, \dots, n }
    \]
    where $\beta^{(t)}_d(x) = \beta^{(t)}(x_1) \cdots \beta^{(t)}(x_d)$ being $\beta^{(t)}$ the B-spline introduce in Definition \ref{eq: def B spline}. We introduce too $B^t = (B_n^t)_{n = 0}^\infty$.
    Then, by the same Proposition \ref{prop: Besov norm equivalence on the unit cube}, we know
    \[
    \Normm_{B^\alpha_{p, q}((0, 1)^d)} \approx \Normm_{L^p((0, 1)^d)} + \Norm{(2^{\alpha k} E(\; \cdot \;, \Sigma_k^{(r, d)})_p)_{k = 0}^\infty}_{\ell^q}
    \]
    for every $0 < \alpha < \min \set{r - 1 + \frac{1}{p}, r} = r - 1 + \min \set{\frac{1}{p}, 1} \eqqcolon \lambda(r - 1, p)$.

    \underline{Step 1:} $B_{p, q}^{\alpha d}((0, 1)^d) \hookrightarrow A^\alpha_q(L^p((0, 1)^d), B^{r})$ for every $\alpha \in (0, \lambda(r, p) / d)$.

    If $j \in \Z^d$ and $k \in \N$ such that $\beta^{(r + 1, d)}_{k, j}|_{(0, 1)^d} \neq 0$ (see Definition \ref{eq: def B spline} for notation) then $j \in [-r, 2^k - 1]^d$. Consequently
    \[
    \# \set{j \in \Z^d \st \beta_{k, j}^{(r, d)} \neq 0} \leq (2^k + r)^d.
    \]
    This implies that
    \[
    \Sigma_k^{(r + 1, d)} \subset B_{(2^k + r)^d}^{r},
    \]
    which suggests the notation $z_k = (2^k + r)^d$. Note that $z_k$ satisfies
    \begin{equation} \label{eq: bounds for zk}
        2^{k d} \leq z_k = 2^{(k - 1) d} \left (2 + \frac{r}{2^{(k - 1)}} \right )^d \leq 2^{(k - 1) d}(r + 2)^d.
    \end{equation}
    We get then, for every $f \in B^{\alpha d}_{p, q}((0, 1)^d)$ with $0 < \alpha d < \lambda(r, p)$, supposing $q < \infty$,
    \begin{align*}
        \Norm{f}_{A^\alpha_q(L^p((0, 1)^d), B^{r})} &= \sum_{n = 1}^\infty \frac{1}{n} \left ( n^\alpha E(f, B^{r}_{n - 1})_p \right )^q \\
        &= \sum_{n = 1}^{z_0} \frac{1}{n} \left ( n^\alpha E(f, B^{r}_{n - 1})_p \right )^q + \sum_{k = 0}^\infty \sum_{n = z_k + 1}^{z_{k + 1}} \frac{1}{n} \left ( n^\alpha E(f, B^{r}_{n - 1})_p \right )^q \\
        &\leq \Norm{f}_{L^p((0, 1)^d)}^q \sum_{n = 1}^{z_0} n^{\alpha q - 1} + \sum_{k = 0}^\infty \sum_{n = z_k + 1}^{z_{k + 1}} 2^{-k d} \left ( \left ( 2^{k d}(r + 2)^d \right )^\alpha E(f, B_{z_k}^{r})_p \right )^q \\
        &\leq C_1 \Norm{f}_{L^p((0, 1)^d)}^q + C_2 \sum_{k = 0}^\infty 2^{k d (\alpha q - 1)} E(f, \Sigma_k^{(r + 1, d)})_p^q (z_{k + 1} - z_k) \\
        &\leq C_1 \Norm{f}_{L^p((0, 1)^d)}^q + C_3 \sum_{k = 0}^\infty (2^{k (\alpha d)} E(f, \Sigma_k^{(r + 1, d)})_p)^q \\
        &\leq 2 \max \set{C_1, C_3} \left ( \Norm{f}_{L^p((0, 1)^d)} + \Norm{(2^{k (\alpha d)} E(f, \Sigma_k^{(r + 1, d)})_p)_{k = 0}^\infty}_{\ell^q} \right )^q \\
        &\leq 2 \max \set{C_1, C_3} C_4^q \Norm{f}_{B^{\alpha d}_{p, q}((0, 1)^d)}^q
    \end{align*}
    where
    \begin{align*}
        C_1 = \sum_{n = 1}^{z_0} n^{\alpha q - 1}, && C_2 = (r + 2)^{d \alpha q}, && C_3 = ((r + 2)^d - 1)C_2
    \end{align*}
    and $\Normm_{L^p((0, 1)^d)} +  \Norm{(2^{\alpha d k} E(\; \cdot \;, \Sigma_k^{(r, d)})_p)_{k = 0}^\infty}_{\ell^q} \leq C_4 \Normm_{B^{\alpha d}_{p, q}((0, 1)^d)}$. We also used \eqref{eq: bounds for zk} in the first inequality, 
    \begin{align*}
        E(f, B^{r}_{n - 1})_p \leq E(f, \Sigma_k^{(r + 1, d)})_p && \txt{ and } && z_{k + 1} - z_k \leq 2^{k d}((r + 2)^d - 1)
    \end{align*}
    in the second and third inequalities respectively. The case $q = \infty$ uses the same bounds:
    \begin{align*}
        \Norm{f}_{A^\alpha_\infty(L^p((0, 1)^d)), B^{r})} &= \sup_{n \in \N} n^\alpha E(f, B^{r}_{n - 1})_p \\
        &\leq \max_{n = 1, \dots, z_0} n^\alpha E(f, B^{r}_{n - 1})_p + \sup_{k \in \N \cup \set{0}} \max_{n = z_k + 1, \dots, z_{k + 1}} n^\alpha E(f, B^{r}_{n - 1})_p \\
        &\leq \Norm{f}_{L^p((0, 1)^d)} z_0^\alpha + \sup_{k \in \N \cup \set{0}} \max_{n = z_k + 1, \dots, z_{k + 1}} n^\alpha E(f, B^{r}_{z_k})_p \\
        &\leq \Norm{f}_{L^p((0, 1)^d)} z_0^\alpha + \sup_{k \in \N \cup \set{0}} z_{k + 1}^\alpha E(f, \Sigma_k^{(r + 1, d)})_p \\
        &\leq \Norm{f}_{L^p((0, 1)^d)} z_0^\alpha + (r + 2)^{d \alpha} \sup_{k \in \N \cup \set{0}} 2^{k (\alpha d)} E(f, \Sigma_k^{(r + 1, d)})_p \\
        &\leq \max \set{z_0, (r + 2)^d}^\alpha \left ( \Norm{f}_{L^p((0, 1)^d)} + \Norm{(2^{k (\alpha d)} E(f, \Sigma_k^{(r + 1, d)})_p}_{\ell^\infty} \right ) \\
        &\leq C_4 \max \set{z_0, (r + 2)^d}^\alpha \Norm{f}_{B^{\alpha d}_{p, \infty}((0, 1)^d)}.
    \end{align*}

    \underline{Step 2:} $B^{\alpha d}_{p, q}(\Omega) \hookrightarrow A^\alpha_q(L^p(\Omega), B^{r})$ for every $\alpha \in (0, \lambda(r, p) / d)$.

    Set $0 < \alpha d < \lambda(r, p)$ and $f \in B^{\alpha d}_{p, q}(\Omega)$. We use the fact that, for every $p \in (0, \infty)$, $q \in (0, \infty]$ and $s > 0$, between two Lipschitz domains $G \subset H$ there exists an extension operator
    \[
    \mathcal{E}: B_{p, q}^s(G) \longrightarrow B_{p, q}^s(H)
    \]
    such that for every $f \in B^s_{p, q}(G)$ $\mathcal{E} f|_G = f$ and there exists $K > 0$ satisfying $\Norm{\mathcal{E} f}_{B^s_{p, q}(H)} \leq K \Norm{f}_{B^s_{p, q}(G)}$. Indeed, its existence for $p \in (0, 1)$ and $p \in [1, \infty)$ is proved in Theorem 6.1 from \cite{devore1988interpolation} and Corollary 1 in Section 4 from \cite{johnen2006equivalence} respectively. Being $\Omega$ bounded, we find $a > 0$ and $b \in \R^d$ such that $\Omega \subset aQ + b$. We define $Tx = a x + b$ for every $x \in \R^d$ and we suppose for a later argument that $a > 1$. We consider then $\mathcal{E}: B^{\alpha d}_{p, q}(\Omega) \longrightarrow B^{\alpha d}_{p, q}(T (0, 1)^d)$ an extension operator. We show that there are constants $C_1, \dots, C_5 > 0$ such that
    \begin{align}
        \Norm{f}_{A^\alpha_q(L^p(\Omega), B^{r})} &\leq C_1 \Norm{\mathcal{E} f}_{A^\alpha_q(L^p(T(0, 1)^d), B^{r})} \label{align: inequality 1} \tag{I.1} \\
        &\leq C_2 \Norm{\mathcal{E} f \circ T}_{A^\alpha_q(L^p((0, 1)^d), B^r)} \label{align: inequality 2} \tag{I.2} \\
        &\leq C_3 \Norm{\mathcal{E} f \circ T}_{B^{\alpha d}_{p, q}((0, 1)^d)} \label{align: inequality 3} \tag{I.3} \\
        &\leq C_4 \Norm{\mathcal{E} f}_{B^{\alpha d}_{p, q}(T(0, 1)^d)} \label{align: inequality 4} \tag{I.4} \\
        &\leq C_5 \Norm{f}_{B^{\alpha d}_{p, q}(\Omega)} \label{align: inequality 5} \tag{I.5}.
    \end{align}
    Note that inequality \eqref{align: inequality 5} is a property of $\mathcal{E}$ and inequality \eqref{align: inequality 3} was proved at Step 1 as long as $Ef \circ T \in B^{\alpha d}_{p, q}((0, 1)^d)$ which is true if inequality \eqref{align: inequality 4} holds.

    \begin{enumerate}
        \item[\eqref{align: inequality 1}] This inequality follows from
        \begin{multline*}
            E(f, B^r_n)_{L^p(\Omega)} = \inf_{g \in B^r_n} \Norm{f - g}_{L^p(\Omega)} = \inf_{g \in B^r_n} \Norm{\mathcal{E} f - g}_{L^p(\Omega)} \\ \leq \inf_{g \in B^r_n} \Norm{\mathcal{E} f - g}_{L^p(T(0, 1)^d)} = E(\mathcal{E} f, B^r_n)_{L^p(T(0, 1)^d)}
        \end{multline*}
        and so $C_1 = 1$.

        \item[\eqref{align: inequality 2}] As the previous inequality, we give estimates of the error:
        \begin{align*}
            E(\mathcal{E} f, B^r_n)_{L^p(T(0, 1)^d)} &= \inf_{g \in B^r_n} \Norm{\mathcal{E} f - g}_{L^p(T(0, 1)^d)} \\
            &= a^{d / p} \inf_{g \in B^r_n} \Norm{\mathcal{E} f \circ T - g \circ T}_{L^p((0, 1)^d)} \\
            &= a^{d / p} \inf_{g \in B^r_n} \Norm{\mathcal{E} f \circ T - g}_{L^p((0, 1)^d)} = a^{d / p} E(\mathcal{E} f \circ T, B^r_n)_{L^p((0, 1)^d)}
        \end{align*}
        where we used that 
        \[
        \set{g \circ T \st g \in B^r_n} = B^r_n
        \]
        which follows from the definition of $B^r_n$ in Lemma \ref{lemma: intermediate approximation space}. We get then \eqref{align: inequality 2} is actually an equality with $C_2 = a^{d / p}$.

        \item[\eqref{align: inequality 4}] To prove this inequality, we use the modulus of smoothness \ref{def: modulus of smoothness}. Let $\tau = \Ceil{\alpha d}$ and $x \in (0, 1)^d$ such that $x, x + h, \dots, x + \tau h \in (0, 1)^d$. By the binomial formula, we get for a fixed $h \in \R^d$
        \begin{align*}
            \Delta_h^\tau(\mathcal{E} f \circ T, (0, 1)^d)(x) &= \sum_{k = 0}^\tau \binom{\tau}{k} (-1)^{\tau - k} \mathcal{E} f \circ T(x + k h) \\
            &= \sum_{k = 0}^\tau \binom{\tau}{k} (-1)^{\tau - k} \mathcal{E} f(Tx + k a h) = \Delta^\tau_{a h}(\mathcal{E} f, T(0, 1)^d)(Tx)
        \end{align*}
        because $Tx, Tx + a h, \dots, Tx + \tau a h \in T(0, 1)^d$. It follows that for any $t > 0$,
        \begin{align*}
        \omega_\tau(\mathcal{E} f \circ T, (0, 1)^d)_p (t) &= \sup_{\Abs{h}_2 < t} \Norm{\Delta_h^\tau(\mathcal{E} f \circ T, (0, 1)^d)}_{L^p((0, 1)^d)}\\
        &= \sup_{a \Abs{h}_2 < a t} \Norm{\Delta_{a h}^\tau(\mathcal{E} f, T (0, 1)^d)  \circ T}_{L^p((0, 1)^d)} \\
        &= a^{- d / p} \sup_{a \Abs{h}_2 < a t} \Norm{\Delta_{a h}^\tau(\mathcal{E} f, T (0, 1)^d)}_{L^p(T (0, 1)^d)} \\
        &= a^{-d/p} \omega_\tau(\mathcal{E} f, T(0, 1)^d)_p(a t).
        \end{align*}
        We can conclude when $q < \infty$:
        \begin{align*}
            \Abs{\mathcal{E} f \circ T}_{B^{\alpha d}_{p, q}((0, 1)^d)}^q &= \int_0^1 \left ( \frac{\omega_\tau(\mathcal{E} f \circ T, (0, 1)^d)_p(t)}{t^{\alpha d}} \right )^q \frac{dt}{t} \\
            &= a^{-d q / p} \int_0^1 \left ( \frac{\omega_\tau(\mathcal{E} f, T (0, 1)^d)_p(at)}{t^{\alpha d}} \right )^q \frac{dt}{t} \\
            &= a^{d q \left ( \alpha  - \frac{1}{p} \right )} \int_0^a \left ( \frac{\omega_\tau(\mathcal{E} f, T (0, 1)^d)_p(s)}{s^{\alpha d}} \right )^q \frac{ds}{s} \\
            &= a^{dq \left ( \alpha - \frac{1}{p} \right )} \left ( \Abs{\mathcal{E} f}_{B^{\alpha d}_{p, q}(T (0, 1)^d)}^q + \int_1^a \left ( \frac{\omega_\tau(\mathcal{E} f, T (0, 1)^d)_p(s)}{s^{\alpha d}} \right )^q \frac{ds}{s} \right ) \\
            &\leq a^{d q \left ( \alpha - \frac{1}{p} \right )} \left ( \Abs{\mathcal{E} f}_{B^{\alpha d}_{p, q}(T (0, 1)^d)}^q + 2^{\tau q} \frac{\Norm{\mathcal{E} f}_{L^p(T(0, 1)^d)}^q}{\alpha d q}\left ( 1 - a^{-\alpha d q} \right ) \right ) \\
            &\leq 2 a^{d q \left ( \alpha - \frac{1}{p} \right )} \max \set{1, \frac{2^{\tau q}}{\alpha d q} \left ( 1 - a^{-\alpha d q} \right )} \Norm{\mathcal{E} f}_{B^{\alpha	d}_{p, q}(T (0, 1)^d)}^q.
        \end{align*} 
        Calling $C = 2 a^{d q (\alpha - 1 / p)} \max \set{1, \frac{2^{\tau q}}{\alpha d q} \left ( 1 - a^{-\alpha d q} \right )}$, we get
        \begin{multline*}
        \Norm{\mathcal{E} f \circ T}_{B^{\alpha	d}_{p, q}(T (0, 1)^d)} = \Norm{\mathcal{E} f \circ T}_{B^{\alpha	d}_{p, q}(T (0, 1)^d)} + \Abs{\mathcal{E} f \circ T}_{B^{\alpha	d}_{p, q}(T (0, 1)^d)} \\ \leq a^{-d / p} \Norm{\mathcal{E} f}_{L^p(T(0, 1)^d)} + C \Abs{\mathcal{E} f}_{B^{\alpha d}_{p, q}(T(0, 1)^d)} \leq \max \set{a^{-d / p}, C} \Norm{\mathcal{E} f}_{B^{\alpha	d}_{p, q}(T (0, 1)^d)}
        \end{multline*}
    \end{enumerate}

    With this chain of inequalities proved, we have shown that $B^{\alpha d}_{p, q}(\Omega) \hookrightarrow A^\alpha_q(L^p(\Omega), B^r)$.

    \underline{Step 3:} $B_{p, q}^{\alpha d}(\Omega) \hookrightarrow M^\alpha_{p, q}(\Omega, \varrho_r)$ for every $\alpha \in (0, \lambda(r, p) / d)$ if $r \neq 1$ or $d = 1$.

    For any $0 < \alpha d < r + \min \set{1, 1 / p}$, we are under the hypothesis of Lemma \ref{lemma: intermediate approximation space} and so, by Step 2,
    \[
    B_{p, q}^{\alpha d}(\Omega) \hookrightarrow A^\alpha_q(L^p(\Omega), B^r) \hookrightarrow M^{\alpha}_{p, q}(\Omega, \varrho_r).
    \]

    \underline{Step 4:} $B^{\alpha d}_{p, q}(\Omega) \hookrightarrow M^\alpha_{p, q}(\Omega, \varrho_1)$ for every $\alpha \in (0, \min \set{1, 1/p} / d)$ if $r = 1$ and $d > 1$.

    By Lemma \ref{lemma: intermediate approximation space}, we know
    \[
    A^\alpha_q(L^p(\Omega), B^0) \hookrightarrow M^\alpha_{p, q}(\Omega, \varrho_1)
    \]
    for every $\alpha > 0$. By Step 2, we know
    \[
    B^{\alpha d}_{p, q}(\Omega) \hookrightarrow A^\alpha(L^p(\Omega), B^0)
    \]
    for every $\alpha \in (0, \min \set{1, 1/p} / d)$. Then, we conclude that
    \[
    B_{p, q}^{\alpha d}(\Omega) \hookrightarrow M^\alpha_{p, q}(\Omega, \varrho_1)
    \]
    for every $\alpha \in (0, \min \set{1, 1 / p} / d)$.
\end{proof}

Theorem \ref{th: Bsov embedded in NN aproximation space} tells us that the approximation pace to functions in the Besov spaces by $\varrho_r$-NNs grows with $r$.

\section{Additional Approximation Properties of Neural Networks.}

The paper \cite{gribonval2022approximation} also proves, in a sense, the converse result of Theorem \ref{th: Bsov embedded in NN aproximation space}:

\begin{theorem}
    Let $p \in (0, \infty)$, $r \in \N$, $\Omega = (0, 1)$ and $L \in \N$. Then, there is $\nu = \nu(L)$ such that, for all $\alpha > 0$,
    \[
    A^\alpha_q(L^p(\Omega, (\Sigma_n)_{n = 0}^\infty) \hookrightarrow B^{\alpha / \nu}_{q, q}(\Omega)
    \]
    where $q = \left ( \frac{\alpha}{\nu} + \frac{1}{p} \right )^{-1}$,
    \[
    \Sigma_n = \set{R(\phi) \st \phi \txt{ is a } \varrho_r, M(\phi) \leq n \txt{ and } L(\phi) \leq L}
    \]
    for $n \in \N$ and $\Sigma_0 = \set{0}$.
\end{theorem}
Note that this theorem is way less general than Theorem \ref{th: Bsov embedded in NN aproximation space} requiring a bound on the number of layers of the approximating NNs, a fixed value for $q$ and being in dimension one.

The results of these chapters are highly relevant to Physics Informed Neural Networks, defined in Section \ref{sec: comments and other results chap 0}. The objective of a PINN is, at the end of its training, approximate the solution and its associate PDE conditions. It is then necessary to see the approximation capabilities on spaces which considers some kind of regularities, like the Besov spaces or Sobolev spaces.

Density theorems for balls in the Sobolev spaces $W^{1, p}$ can be found in Theorem 4.1 of \cite{abdeljawad2022approximations} and in Theorem 4.7 \cite{shen2022approximation} while giving bounds on the number of parameters when approximating. For the fractional Sobolev spaces $W^{s, p}((0, 1)^d)$ with $s \in [0, 1]$, the Theorem 4.1 in \cite{guhring2020error} shows that functions on usual Sobolev spaces can be approximated in the norm of fractional Sobolev spaces. For higher order derivative approximation, Theorem 4.9 of \cite{guhring2021approximation} shows that classic NN can approximate any function $f \in W^{n, p}((0, 1)^d)$ in the space $W^{k, p}((0, 1)^d)$ if $k, n \in \N$ and $k < n$. Said theorem gives upper bounds for the size of the NN needed to approximate. 

There are other recent and important results. Concerning inverse problems, the paper \cite{puthawala2022globally} characterizes injective realizations of ReLU NN, shows that any continuous functions can be approximated by injective realization of ReLU NN and if we add some hypothesis to the functions we approximate, there are Lipschitz realization of ReLU NN. Moreover, the paper \cite{furuya2024globally} repeats the work of \cite{puthawala2022globally} but the authors work with Neural Operators to approximate operators. Finally, an important convergence result of the gradient descent algorithm for NN can be found in \cite{du2019gradient}.
\indiceFiguras
\biblioTFG{4}
\printglossaries
\end{document}